\documentclass{amsart}

\usepackage{amssymb, amsmath,  amsfonts }
\usepackage{color }
\usepackage{slashbox, longtable}
\usepackage[all,cmtip]{xy}

\newtheorem{thm}{Theorem}[section]
\newtheorem{prop}[thm]{Proposition}
\newtheorem{lemma}[thm]{Lemma}
\newtheorem{cor}[thm]{Corollary}
\newtheorem{defn}[thm]{Definition}
\newtheorem{example}[thm]{Example}
\newtheorem{remark}[thm]{Remark}
\newtheorem{notation}[thm]{Notation}
\newtheorem{qus}[thm]{Question}

\newtheorem{conjecture}[thm]{Conjecture}

\newcommand{\circline}{\!\;\!\!\circ\!\!\;\!\!-\!\!\!-\!}
\newcommand{\bulletline}{\bullet\!\!\;\!\!-\!\!\!-\!}

\newcommand{\bulletbb}{  \!\;\!\!\circ\!\!\!\Rightarrow\!\!=\!\!\!\bullet}
\newcommand{\bulletcc}{ \!\;\!\!\circ\!\!\!=\!\!\Leftarrow\!\!\!\bullet}

\numberwithin{equation}{section}

\begin{document}{\allowdisplaybreaks[4]

\title[Functorial relationships   between $QH^*(G/B)$ and $QH^*(G/P)$]{Functorial relationships \\ between $QH^*(G/B)$ and $QH^*(G/P)$}

\author{Naichung Conan Leung}
\address{The Institute of Mathematical Sciences and Department of Mathematics,
           The Chinese University of Hong Kong, Shatin, Hong Kong}
\email{leung@math.cuhk.edu.hk}
\thanks{ The first author  is supported in part by a RGC research grant from
the Hong Kong Government. The second author  is  supported in part
by KRF-2007-341-C00006.
 }

\author{Changzheng Li}
\address{School of Mathematics, Korea Institute for Advanced Study, 87
Hoegiro, Dongdaemun-gu, Seoul, 130-722, Korea}
\email{czli@kias.re.kr}

\date{
      }




\begin{abstract}
We   give a natural  filtration  $\mathcal{F}$ on $QH^*(G/B)$, which respects the quantum product structure.
 Its associated  graded algebra  $Gr^{\mathcal{F}}(QH^*(G/B))$ is isomorphic to the tensor product of
       $QH^*(G/P)$ and  a corresponding graded algebra of  $QH^*(P/B)$ after localization. When
        the quantum parameter goes to zero,
         this specializes to the filtration on
              $H^*(G/B)$ from the Leray spectral sequence associated to the fibration $P/B \rightarrow G/B\longrightarrow G/P$.

\end{abstract}

\maketitle

\section{Introduction}

Let $G$ be a simply-connected  complex simple Lie group, $B$ be a Borel subgroup  and  $P\supset B$ be
   a  parabolic subgroup  of $G$.
The natural  fibration   $P/B{\rightarrow} G/B{\longrightarrow} G/P$    
    of homogeneous varieties  
  gives rise to a $\mathbb{Z}^2$-filtration $\mathcal{F}$ on $H^*(G/B)$  over $\mathbb{Q}$ (or $\mathbb{C}$) such that  
         $Gr^{\mathcal{F}}(H^*(G/B))\cong H^*(P/B)\otimes H^*(G/P)$ as graded algebras by the Leray-Serre spectral sequence.
      Given another parabolic subgroup $P'$ with $B\subset P'\subset P$, we obtain the corresponding  natural
       fibration $P'/B\rightarrow P/B\longrightarrow P/P'$. 
       Combining it with the former one, we obtain a $\mathbb{Z}^3$-filtration on $H^*(G/B)$. We can continue  this procedure to obtain a (maximal)  $\mathbb{Z}^{r+1}$-filtration.

      In the present paper, we study the  small  quantum cohomology rings $QH^*(G/P)$'s of  homogeneous varieties $G/P$'s,
          which are   deformations of the ring structures on  $H^*(G/P)$'s by incorporating genus zero 3-pointed Gromov-Witten invariants of $G/P$'s into the cup product.
       We show the ``\textit{functorial relationships}" between $QH^*(G/B)$ and $QH^*(G/P)$ in the sense that
            the $\mathbb{Z}^{r+1}$-filtration on $H^*(G/B)$ can be  generalized to give a $\mathbb{Z}^{r+1}$-filtration on $QH^*(G/B)$ and there exist
         canonical maps between quantum cohomologies, in analog with  the classical ones. 
         We begin with a toy example to illustrate our results.

\begin{example}\label{exampfiltrforA2}
 When $G=SL(3, \mathbb{C})$,
          $G/B=\{V_1\leqslant    V_2\leqslant \mathbb{C}^{3}~|~ \dim_\mathbb{C}V_i=1,$ $i=1, 2\}=:F\ell_3$ is a complete flag variety.
Given a maximal parabolic subgroup $P\supset B$,  we have $P/B=\mathbb{P}^1$ and  $ G/P=\mathbb{P}^2$ together with
   a natural fibration $\mathbb{P}^1\overset{i}{\hookrightarrow} F\ell_{3}\overset{\pi}{\longrightarrow} \mathbb{P}^2$.
The quantum cohomology ring $QH^*(G/B)$ has a basis consisting  of Schubert classes $\sigma^w$'s over $\mathbb{Q}[q_1, q_2]$, indexed by the Weyl group
   $W =S_3=\{1, s_1, s_2, s_1s_2, s_2s_1, s_1s_2s_1\}$.
  To obtain  the $\mathbb{Z}^2$-filtration $\mathcal{F}$ on $QH^*(G/B)$, we need a deformation $gr$ of the classical grading map
         which satisfies $gr(q_1^aq_2^b\sigma^w)=a gr(q_1)+b gr(q_2)+gr(\sigma^w)$.
      In this   example, $gr$ is given explicitly by
              Table \ref{tabgraexamp}. 

 \renewcommand{\arraystretch}{2.3}

\begin{table}[h]
\caption{\label{tabgraexamp} $
          gr(q_1^aq_2^b \sigma^w)=(i, j) \mbox{ with }  -2\leq i\leq 4, 0\leq j\leq 6 
          $}
  \begin{tabular}{|c||c| c|c |c |c|c|c|}
    \hline
  $4$ & $q_1^{\,2}$    & $ q_1^{\,2}\sigma^{s_2}$  & $  q_1^{\,2}\sigma^{s_1s_2}$  &$  q_1^{\,2}q_2\sigma^{s_1}$  & $ q_1^{\,2}q_2\sigma^{s_2s_1}$ & $ q_1^{\,2}q_2\sigma^{s_1s_2s_1}$& $q_1^{\,3}q_2^{\,2}$  \\\hline
 $3$ & $q_1\sigma^{s_1}$    & $ q_1\sigma^{s_2s_1}$  & $q_1\sigma^{s_1s_2s_1}$& $q_1^{\,2}q_2$ & $ q_1^{\,2}q_2\sigma^{s_2}$ & $q_1^{\,2}q_2\sigma^{s_1s_2}$ &$q_1^{\,2}q_2^{\,2}\sigma^{s_1}\!\!{}$\\\hline
$2$ & $q_1$   & $ q_1\sigma^{s_2}$  & $q_1\sigma^{s_1s_2}$  & $
q_1q_2\sigma^{s_1}$  & $q_1q_2\sigma^{s_2s_1}$ &
$q_1q_2\sigma^{s_1s_2s_1}$& $q_1^{\,2}q_2^{\,2}$   \\\hline
   $1$ & $\sigma^{s_1}$    & $  \sigma^{s_2s_1}$  & $\sigma^{s_1s_2s_1}$& $q_1q_2$ & $q_1q_2\sigma^{s_2}$ & $q_1q_2\sigma^{s_1s_2}$ &$q_1q_2^{\,2}\sigma^{s_1}\!\!{}$  \\\hline
  $0$ & $1$                       & $ {\sigma^{s_2}}$  & $\sigma^{s_1s_2}$  &$ q_2\sigma^{s_1}$  & $q_2\sigma^{s_2s_1}$ & $q_2\sigma^{s_1s_2s_1}$& $q_1q_2^2$   \\\hline
   $\!\!\!-1$ & 0   & 0 & 0 & $q_2$ & $q_2\sigma^{s_2}$ & $q_2\sigma^{s_1s_2}$ &$q_2^{\,2}\sigma^{s_1}$   \\\hline
 $\!\!\!-2$ &   0             &  0  & 0   & 0  & 0 & 0 & $q_2^{\,2}$  \\ \hline \hline
\multicolumn{1}{|@{}l||}{\slashbox[0.8cm][l]{$i$}{$j$}}     &    $  0$ &  $ 1$   &  $ 2$    & $ 3$ & $ 4$ &  $ 5$&  $ 6$         \\\hline 
  \end{tabular}
\end{table}

This determines a $\mathbb{Z}^2$-filtration $\mathcal{F}=\{F_{\mathbf{c}}\}_{\mathbf{c}\in\mathbb{Z}^2}$ on $QH^*(G/B)$.
The main point is this filtration respects the quantum multiplication, i.e.
   $F_{\mathbf{c}}F_{\mathbf{d}}\subset F_{\mathbf{c}+\mathbf{d}}$.
  Indeed, this can be easily  checked with the following
        well-known    quantum products for $F\ell_3$:
      \begin{eqnarray*}
         \sigma^{s_1}\star \sigma^{s_1}\!=\!\sigma^{s_2s_1}+q_1, & \sigma^{s_1}\star\sigma^{s_1s_2}= \sigma^{s_1s_2s_1},& \sigma^{s_1}\star \sigma^{s_2s_1}\!=\!q_1\sigma^{s_2},\\
         \sigma^{s_2}\star \sigma^{s_2}\!=\! \sigma^{s_1s_2}+q_2, & \sigma^{s_2}\star\sigma^{s_2s_1}= \sigma^{s_1s_2s_1},& \sigma^{s_2}\star \sigma^{s_1s_2}\!=\!q_2\sigma^{s_1},\\
          \sigma^{s_1s_2}\star \sigma^{s_1s_2s_1}\!=\! q_1q_2\sigma^{s_2},& \sigma^{s_1s_2}\star\sigma^{s_1s_2}= q_2\sigma^{s_2s_1},&\sigma^{s_1}\star \sigma^{s_1s_2s_1}\!=\! q_1\sigma^{s_1s_2}+ q_1q_2,\\
        \sigma^{s_2s_1}\star \sigma^{s_1s_2s_1}\!=\! q_1q_2\sigma^{s_1},& \sigma^{s_2s_1}\star\sigma^{s_2s_1}= q_1\sigma^{s_1s_2},&\sigma^{s_2}\star \sigma^{s_1s_2s_1}\!=\! q_2\sigma^{s_2s_1}+ q_1q_2,\\
          \sigma^{s_1}\star \sigma^{s_2}\!=\! \sigma^{s_1s_2}+ \sigma^{s_2s_1}, & \sigma^{s_1s_2}\star\sigma^{s_2s_1}= q_1q_2,& \sigma^{s_1s_2s_1}\star \sigma^{s_1s_2s_1}\!\!=\! q_1q_2\,\sigma^{s_1}\!\!\star \sigma^{s_2}\!.\!{}
     \end{eqnarray*}

This $\mathbb{Z}^2$-filtration of the algebra structure on $QH^*(G/B)$ is a $\mathbf{q}$-deformation of the classical one on $H^*(G/B)$,
 which comes from the Leray spectral sequence. Due to the existence of such a filtration,
  we can easily check
  that there are \textbf{algebra} isomorphisms
      $\bar \varphi: QH^*(G/B)/\mathcal{I}\longrightarrow QH^*(P/B)$
                         and  $\bar \psi: QH^*(G/P)\longrightarrow \mathcal{A}/\mathcal{J}$,
 where  $\mathcal{A}:=\bigcup_{j\geq 0}F_{(0, j)}$ is a subalgebra of $QH^*(G/B)$,
        $\mathcal{J}:=F_{(0, -1)}$ is an ideal of $\mathcal{A}$ and
         $\mathcal{I}$ is the ideal in  $QH^*(G/B)$
          spanned by those  $q_1^aq_2^b\sigma^w$'s with their gradings $(d_1, d_2)$ satisfying $d_2>0$.
 Here    $\bar\varphi$ sends   $q_1^aq_2^b\sigma^w+\mathcal{I}$  to $y^j$ where
                   $q_1^aq_2^b\sigma^w$ is the (unique) one among  such expressions    with its grading equal to   $(j, 0)$,  and
      $\bar \psi$ sends $x^j$ to the (unique) $q_1^aq_2^b\sigma^w\in QH^*(G/B)$ whose grading equals   $(0, j)$, in which we have taken the
         well-known isomorphisms   $QH^*(P/B)\cong\mathbb{Q}[y]$ and  $QH^*(G/P)\cong\mathbb{Q}[x]$.
   In particular, $QH^*(G/P)$ is the quotient of the subalgebra $\mathcal{A}$ generated by
     $\{q_2\sigma^{s_1}, \sigma^{s_2}, q_1q_2^2, q_2\}$ by the ideal $\mathcal{J}=q_2\mathcal{A}$. (We remark that
       in this case $QH^*(G/B)$ itself is generated by   $\{\sigma^{s_1}, \sigma^{s_2}, q_1, q_2\}$.)

These algebra isomorphisms generalize the classical ones in an obvious  way, namely $\mathcal{A}$,
 $\mathcal{J}$ and $\mathcal{I}$ are $\mathbf{q}$-deformations of  $A:=\pi^*(H^*(G/P))$,  $J=0$ and
     $I=\mathbb{Q}\{\sigma^{s_2}, \sigma^{s_1s_2}, \sigma^{s_2s_1}, \sigma^{s_1s_2s_1}\}$ respectively.

    \end{example}

All the above descriptions for $G=SL(3, \mathbb{C})$ will be generalized to arbitrary   complex semi-simple
  Lie groups. For simplicity, we  assume $P/B$ is irreducible.   (Note that any homogeneous variety splits into a direct product of
     irreducible ones.) All the results     can  be easily
      generalized for reducible $P/B$'s and we will    describe such generalizations  in section \ref{secconclusion}.
   Note that $P/B$ is again a complete flag variety,
          isomorphic to $G'/B'$ for some other complex simple Lie group $G'$.   Then we denote by $r$ the rank of $G'$, which depends only on $P/B$. 
        Since we exclude the trivial cases, namely $P$ equals $B$ or $G$,
         we always have $r=1$ for
                $G=SL(3, \mathbb{C})$.

 In general, we  consider a special iterated fibration
    $\{P_{j-1}/P_0\rightarrow P_j/P_0\longrightarrow P_j/P_{j-1}\}_{j=2}^{r+1}$ in which $P_j$'s are parabolic subgroups with
      $B=P_0\subsetneq P_1\subsetneq \cdots \subsetneq P_r=P\subsetneq G=P_{r+1}$.
 Consequently, we   obtain a canonical $\mathbb{Z}^{r+1}$-filtration on $H^*(G/B)$.
Note that $QH^*(G/B)$ also has a natural basis of Schubert classes $\sigma^w$'s over $\mathbb{Q}[\mathbf{q}]$.
 As we will see in section \ref{secdefgrading}, there exists a  grading map $gr$ giving
                 gradings  $gr(q_\lambda \sigma^w)\in \mathbb{Z}^{r+1}$ for the ($\mathbb{Q}$-)basis $q_\lambda \sigma^w$'s.
 The Peterson-Woodward comparison formula in \cite{wo} (Proposition \ref{comparison})
   plays a key role in defining $gr$. It is the only  known formula that characterizes  the relations of genus zero 3-pointed Gromov-Witten invariants  between
        $G/B$ and a general $G/P$ explicitly.
             $gr$ defines a $\mathbb{Z}^{r+1}$-filtration  $\mathcal{F}=\{F_{\mathbf{a}}\}_{\mathbf{a}\in \mathbb{Z}^{r+1}}$ of subspaces in $QH^*(G/B)$, generalizing
            Example \ref{exampfiltrforA2}. The next theorem says that $\mathcal{F}$ respects the quantum product structure.

\begin{thm}\label{mainthm} $QH^*(G/B)$ is a $\mathbb{Z}^{r+1}$-filtered algebra   with filtration $\mathcal{F}$. 
\end{thm}

We can obtain several important consequences as below.
\begin{thm}\label{inclusisom}
 The   vector subspace   $\mathcal{I}$,
          spanned by those  $q_\lambda\sigma^w$'s with their gradings $(d_1,\cdots, d_{r+1})$ satisfying $d_{r+1}>0$, is    an ideal of
                $QH^*(G/B)$. 
  Furthermore, there is a canonical algebra isomorphism
      $$QH^*(G/B)/\mathcal{I}  \overset{\simeq}{\longrightarrow} QH^*(P/B) .$$
\end{thm}

  Since $QH^*(G/B)$ has a $\mathbb{Z}^{r+1}$-filtration $\mathcal{F}$,
          we obtain 
             an associated $\mathbb{Z}^{r+1}$-graded algebra
       $Gr^{\mathcal{F}}(QH^*(G/B))=\bigoplus_{\mathbf{a}\in \mathbb{Z}^{r+1}} Gr_\mathbf{a}^{\mathcal{F}},\mbox{ where }
                Gr_{\mathbf{a}}^{\mathcal{F}}:=F_{\mathbf{a}}\big/\cup_{\mathbf{b}<\mathbf{a}}F_{\mathbf{b}}.$
For each $j$, 
we denote
   $Gr_{(j)}^{\mathcal{F}}(QH^*(G/B)):=\bigoplus_{i\in \mathbb{Z}}Gr_{i\mathbf{e}_j}^{\mathcal{F}}.$

\begin{thm}\label{isomfordirection}
   For each $1\leq j\leq r$,  there exists   a canonical
    algebra isomorphism,
     $$\quad\Psi_j: QH^*({P}_j/{P}_{j-1}) \overset{\simeq}{\longrightarrow}  Gr_{(j)}^{\mathcal{F}}(QH^*(G/B)).  $$
   Furthermore if $P/B\cong F\ell_{r+1}$, then
       there exists   a canonical    algebra isomorphism,
     $$\quad\Psi_{r+1}: QH^*(G/P)  \overset{\simeq}{\longrightarrow}  Gr_{(r+1)}^{\mathcal{F}}(QH^*(G/B)).  $$
 \end{thm}

As a consequence, we have the following results for any $G$.
\begin{thm}\label{baseisom}
 Suppose $P/B\cong F\ell_{r+1}$. Then there exists a    subalgebra  $\mathcal{A}$   of  $QH^*(G/B)$ together with  an ideal $\mathcal{J}$  of $\mathcal{A}$, such that
    $QH^*(G/P)$ is canonically isomorphic to $\mathcal{A}/\mathcal{J}$ as algebras.
\end{thm}

 \begin{thm}\label{localizisom}
Suppose $P/B\cong F\ell_{r+1}$.  Then as graded algebras $Gr^{\mathcal{F}}(QH^*(G/B))$ is isomorphic to
         $ QH^*(\mathbb{P}^1)\otimes \cdots\otimes QH^*(\mathbb{P}^r)\otimes QH^*(G/P)$ after localization.
\end{thm}
We should point out that the requirement  ``$P/B\cong F\ell_{r+1}$"    in Theorem \ref{baseisom}
   and Theorem \ref{localizisom} is   not a  strong assumption,  because  both of theorems  can be easily generalized to
     the case ``$P/B$ is isomorphic to a product of $F\ell_{k}$'s" (see section \ref{secconclusion}). As a consequence, all $G/P$'s for $G$ being of $A$-type or $G_2$-type satisfy this
      assumption. Furthermore for each   remaining type, more than half of the homogeneous varieties $G/P$'s also satisfy this.
      (We could also show Theorem \ref{baseisom} holds for any $G/P$ with $G=Sp(2n, \mathbb{C})$.)

 As we saw in Example \ref{exampfiltrforA2}, the gradings of elements in $QH^*(F\ell_3)$ only form a proper sub-semigroup  $S$ of $\mathbb{Z}^2$,
  which looks like stairs, so that the $\mathbb{Z}^{2}$-filtration    comes from an $S$-filtration. In general, the $\mathbb{Z}^{r+1}$-filtration    comes from a similar filtration.
  For this  reason, we need   localization  to obtain the analog of graded-algebra isomorphism (in Theorem \ref{localizisom}).
In section \ref{secproof22}, we will restate theorems \ref{isomfordirection},  \ref{baseisom} and   \ref{localizisom} more concretely.
   As we will see later, all the relevant maps generalize the
  classical ones  in an obvious way, as in Example \ref{exampfiltrforA2}.

  Our results relate the quantum cohomologies  of the total space and the base space of the fibration  $P/B\rightarrow G/B\longrightarrow G/P$.
  Similar structures occur when one studies
            the   relationships of $J$-functions  between an  abelian quotient  and a nonabelian quotient.
         Such relations were studied by   Bertram, Ciocan-Fontanine and Kim in \cite{bck22} and \cite{bck}. (See also \cite{wozi}.)
         There were also relevant studies by Liu-Liu-Yau \cite{chllyau} and Paksoy  \cite{paksoy} by using mirror principle 
                   \cite{llyau11} .

 Let us   mention two more important problems on the study of $QH^*(G/P)$, for which our theorems may also be helpful.
  One can see the excellent survey   \cite{fu11} and references therein
           for more details on the  developments.
  As mentioned before,  the (small) quantum cohomology ring $QH^*(G/P)$ has a basis of  Schubert classes  $\sigma^w$'s over
          $\mathbb{Q}[\mathbf{q}]$.
   In order to understand $QH^*(G/P)$, one would like to have (i) a (good) presentation of the ring structure on  $QH^*(G/P)$ and
       (ii) a (nice) formula (or algorithm) for the quantum Schubert structure constants $N_{u, v}^{w, \lambda_P}$'s
                in the quantum product $\sigma^u\star \sigma^v=\sum_{w, \lambda_P}N_{u, v}^{w, \lambda_P}q_{\lambda_P}\sigma^w$.
   For classical cohomology $H^*(G/P)$, these  natural and important problems have been solved in
           \cite{borel}  for  (i) and  in   \cite{koku} and  \cite{duan}  for  (ii). However, for quantum cohomology $QH^*(G/P)$,
  the answer to (i) is only known in certain cases, for instance when $G$ is of $A$-type (see \cite{kimpartflag}, \cite{assa}) or $P=B$ is a Borel subgroup   \cite{kim}.
For  problem (ii),   there were early studies for a few cases, including    complex Grassmannians (see the survey \cite{fu11})
       and complete flag varieties of $A$-type \cite{fominGP},
          besides the quantum Chevalley formula \cite{fw} which works for all cases. Recently, Mihalcea \cite{mih} has given an algorithm and the authors
      (\cite{leungli22}, \cite{czli}) have given a
      combinatorial formula for these structure constants.

 All these problems were discussed  in the unpublished work \cite{peterson} by  Dale Peterson.
 In \cite{wo},  Woodward proves a comparison formula of Peterson.
  The Peterson-Woodward comparison formula explicitly characterizes
     the relations of  the quantum Schubert structure constants between $QH^*(G/P)$ and $QH^*(G/B)$. However,  it does not tell us the relations of the algebra structures
       between them.  Along Peterson's approach,
     Lam and Shimozono \cite{lamshi}   show  that the torus-equivariant extension of $QH^*(G/P)$ is isomorphic to
       a quotient of the torus-equivariant homology of a based loop group after localization.  In \cite{rietsch}, K. Rietsch discusses the relationships between
     Peterson's work and mirror symmetry.
 In \cite{peterson},   Peterson had also claimed there was an analogous isomorphism for the (un-iterated) fibration
     $P/B\rightarrow G/B\longrightarrow G/P$  in terms of torus-equivariant homology of based loop groups after localization.
 We were motivated by his claim and the results by Woodward and Lam-Shimozono.
  We succeeded in obtaining  natural generalizations of the classical isomorphisms.  It is interesting to compare our results with Peterson's claim.
  It is also interesting to compare our Theorem \ref{baseisom} with Theorem 10.16 of \cite{lamshi} by Lam and Shimozono. As commented by Thomas Lam, our results should be
  related to the discussions in section 10.4 of \cite{lamshi}.

  We hope our results could be used to solve
                 problem (i) by combining with Kim's early work \cite{kim}, where  a nice presentation  of the ring structure on
                      the complexified quantum cohomology  $QH^*(G/B)$ was given.

  This paper is organized as follows. In section 2, we  define a grading map and prove our main result, Theorem \ref{mainthm}, assuming   the Key Lemma.
     Then we devote the whole   section 3 to  the proof of    the Key Lemma. In section 4, we
      prove the remaining theorems   discussed as above. In section 5, we show how to generalize our results to the general case when $P/B$ is reducible. Finally in section 6,
       we give an appendix which deals with exceptional cases in the proof of the Key Lemma.
Our proofs are   combinatorial in nature. We  hope to find nice  geometrical explanations of them  later.

\subsection*{Acknowledgements}
 The authors thank Baohua Fu,
Bumsig Kim, Thomas Lam, Augustin-Liviu Mare and Leonardo Constantin
Mihalcea for useful discussions. We also  thank the referee for valuable suggestions.

\section{A filtration on $QH^*(G/B)$}
\subsection{Preliminaries}\label{prelimiar}


 We recall some basic notions and fix the notations. 
  See for example \cite{hum}, \cite{humalg} for more details on Lie theory.

Let $G$ be a simply-connected complex simple Lie group of rank $n$,
   $B\subset G$ be a Borel subgroup and $P\supset B$ be a proper parabolic subgroup of $G$. Then   $P$ corresponds canonically to a proper subset $\Delta_P$ of
$\Delta$. (In particular,
            $B$ corresponds to the empty subset $\emptyset$.)
 Fix a basis of simple roots $\Delta=\{\alpha_1, \cdots, \alpha_n\}$
(with respect to $(G, B)$).
    Let $\mathfrak{h}$ denote the corresponding Cartan subalgebra, then $\mathfrak{h}^*= \bigoplus_{i=1}^n\mathbb{C}\alpha_i$.
 Let  $\{\alpha_1^\vee, \cdots, \alpha_n^\vee\}\subset\mathfrak{h}$ be the fundamental coroots and
      $\{\chi_1, \cdots, \chi_n\}\subset \mathfrak{h}^*$ be the fundamental weights. For any $1\leq i, j\leq n$, we have
        $\langle \chi_i, \alpha_j^\vee\rangle =\delta_{i, j}$ with respect to the natural pairing  $\langle\cdot, \cdot\rangle
                    :\mathfrak{h}^*\times\mathfrak{h}\rightarrow \mathbb{C}$. Furthermore, we have $\rho:={1\over 2}\sum_{\gamma\in R^+}\gamma=\sum_{i=1}^n\chi_i$.
   For each $1\leq i\leq n$, the simple reflection $s_i:=s_{\alpha_i}$ acts on $\mathfrak{h}$ and $\mathfrak{h}^*$ by
        $$s_i(\lambda)=\lambda-\langle \alpha_i, \lambda\rangle\alpha_i^\vee, \mbox{ for }\lambda\in\mathfrak{h};\quad
                s_i(\beta)=\beta-\langle \beta, \alpha_i^\vee\rangle\alpha_i, \mbox{ for }\beta\in\mathfrak{h}^*.$$
        The Weyl group $W$, which is generated by 
           $\{s_1, \cdots, s_n\}$, acts on $\mathfrak{h}$ and $\mathfrak{h}^*$  and preserves the natural pairing.
                      The root system is given by $R=W\cdot\Delta=R^+\sqcup (-R^+)$, where
                       $R^+=R\cap \bigoplus_{i=1}^n{\mathbb{Z}_{\geq 0}}\alpha_i$ is the set of positive roots.
      Thus each root $\gamma\in R$ is given by $\gamma=w(\alpha_i)$ for some $w\in W$ and $1\leq i \leq n $.   
     Then we     define $\gamma^\vee=w(\alpha_i^\vee)$ and $s_\gamma=ws_iw^{-1}\in W$, which is independent of the expressions of
           $\gamma$.

      The length $\ell(w)$ of $w\in W$ (with respect to  $\Delta$) is defined by
     $\ell(1)\triangleq 0$  and
        $\ell(w)\triangleq\min\{k~|~ w=s_{i_1}\cdots s_{i_k}\}$ for $w\neq 1$.
   An expression
  $w=s_{i_1}\cdots s_{i_\ell}$ is called \textbf{reduced} if $\ell=\ell(w)$.
    Let $\tilde P=P_{\tilde \Delta}$  denote the (standard) parabolic subgroup  corresponding to a subset  $\tilde \Delta\subset \Delta$,   $W_{\tilde P}$
      denote the subgroup generated by $\{s_j~|~ \alpha_j\in \tilde \Delta\}$ and $\omega_{\tilde P}$ denote the longest element in $W_{\tilde P}$.
    For $\bar \Delta\subset \tilde \Delta$ with $  {\bar P}:=P_{   {\bar \Delta}}$, we denote
     $W_{\tilde P}^{\bar P}:= \{w\in W_{\tilde P} |  \ell(w)\leq \ell(v),\, \forall v\in wW_{\bar P}\}$.
       Each coset in $W_{\tilde P}/W_{\bar P}$ has a unique (minimal length) representative in $W_{\tilde P}^{\bar P}\subset W_{\tilde P}\subset W$.
        In particular, 
         we have        $P_{\Delta}=G$ and    $W_{G}=W$, and simply denote $W^{\bar P}:=W_{G}^{\bar P}$ and $\omega:=\omega_{G}$.

   The (co)homology of a homogeneous variety $X=G/P$ has an additive basis of Schubert (co)homology classes indexed by $W^P$:
       $ H_*(X,\mathbb{Z})=\bigoplus_{v\in W^P}\mathbb{Z}\sigma_{v}$, \,\, $H^*(X,\mathbb{Z})=\bigoplus_{u\in W^P}\mathbb{Z}\sigma^{u}$ with
        $   \langle\sigma^{u}, \sigma_v\rangle =\delta_{u, v}$ for any $u, v\in W^P$ \cite{bgg}. In particular,
          $H_2(X,\mathbb{Z})=\bigoplus_{\alpha_i\in \Delta\setminus\Delta_P}
         \mathbb{Z}\sigma_{s_i}$.
     Set $Q^\vee=\bigoplus_{i=1}^n\mathbb{Z}\alpha_i^\vee$ and
          $Q^\vee_P=\bigoplus_{\alpha_i\in \Delta_P}\mathbb{Z}\alpha_i^\vee$. Then we can identify $H_2(X, \mathbb{Z})$ with $Q^\vee/Q^\vee_P$ canonically,
            by mapping $\sum_{\alpha_j\in \Delta\setminus\Delta_P}a_j\sigma_{s_j}$ to $\lambda_P=\sum_{\alpha_j\in \Delta\setminus\Delta_P}a_j\alpha_j^\vee+Q^\vee_P$.
        For each $\alpha_j\in \Delta\setminus\Delta_P$, we introduce a formal variable $q_{\alpha_j^\vee+Q^\vee_P}$.
          For such $\lambda_P$, we denote $q_{\lambda_P}=\prod_{\alpha_j\in \Delta\setminus\Delta_P}q_{\alpha_j^\vee+Q^\vee_P}^{a_j}$.

  Let $\overline{\mathcal{M}}_{0, m}(X, \lambda_P)$ be the moduli space of stable maps of degree $\lambda_P\in H_2(X, \mathbb{Z})$ of $m$-pointed genus zero curves into $X$  \cite{fupa},
 and $\mbox{ev}_i$ denote the $i$-th canonical evaluation map $\mbox{ev}_i: \overline{\mathcal{M}}_{0, m}(X, \lambda_P)\to X$ given by
    $\mbox{ev}_i([f: C\to X; p_1, \cdots, p_m])=f(p_i)$. The genus zero Gromov-Witten invariant for $\gamma_1,\cdots, \gamma_m\in H^*(X)=H^*(X, \mathbb{Q})$ is defined as
         $ I_{0, m, \lambda_P}(\gamma_1, \cdots, \gamma_m)=\int_{\overline{\mathcal{M}}_{0, m}(X, \lambda_P)}\mbox{ev}_1^* (\gamma_1)\cup\cdots
                            \cup \mbox{ev}_m^*(\gamma_m). $
 The
   (small) quantum product  for $a, b\in H^*(X)$ is a deformation of the cup product, defined by
  $a\star b$ $\triangleq \sum_{u\in W^P, \lambda_P\in H_2(X, \mathbb{Z})} I_{0, 3, \lambda_P}  (a, b, (\sigma^u)^\sharp) \sigma^u q_{\lambda_P}$, where
     $\{(\sigma^{u})^\sharp~|~ u\in W^P\}$ are the elements in $H^*(X)$ satisfying $\int_X(\sigma^u)^\sharp \cup\sigma^v=\delta_{u, v}$ for any $u, v\in W^P$.
  The quantum  product  $\star$ is associative, making
        $(H^*(X)\otimes\mathbb{Q}[\mathbf{q}],  \star)$  a commutative ring. This ring is denoted as  $QH^*(X)$ and called the
     (\textbf{small}) \textbf{quantum cohomology ring} of $X$.
      The same Schubert classes $\sigma^u=\sigma^u\otimes 1$ form
      a basis for $QH^*(X)$ over $\mathbb{Q}[\mathbf{q}]$ and we write
                $$\sigma^u\star \sigma^v = \sum_{w\in W^P, \lambda_P\in Q^\vee/Q^\vee_P}    N_{u,v}^{w, \lambda_P}q_{\lambda_P}\sigma^w.$$
  The coefficients $N_{u, v}^{w, \lambda_P}$'s are called the \textit{quantum Schubert structure constants.} They generalize the  well-known Littlewood-Richardson coefficients
    when $X=Gr(k, n+1)$ is a complex Grassmannian. It is also well-known that the quantum Schubert structure constants are non-negative.

    When $P=B$, we have $Q^\vee_P=0$, $W_P=\{1\}$ and $W^P=W$. In this case, we simply denote $\lambda=\lambda_P$ and $q_j=q_{\alpha_j^\vee}$. A combinatorial formula
     for $N_{u, v}^{w, \lambda}$'s has been given by the authors recently \cite{leungli22}.
     As a consequence, we can obtain the combinatorial formula for  $N_{u, v}^{w, \lambda_P}$'s for general $G/P$, due to the
       following comparison formula.
  \begin{prop}[Peterson-Woodward comparison formula  \cite{wo}; see also \cite{lamshi}]\label{comparison}
    ${}$
    \begin{enumerate}
      \item Let $\lambda_P\in Q^\vee/Q_P^\vee$. Then there is a unique $\lambda_B\in Q^\vee$ such that $\lambda_P=\lambda_B+Q_P^\vee$ and
                $\langle \alpha, \lambda_B\rangle  \in \{0, -1\}$ for all $\alpha\in R^+_P \,\,(=R^+\cap \bigoplus_{\alpha_j\in \Delta_P}\mathbb{Z}\alpha_j)$.
      \item              For every $u, v, w\in W^P$, we have  $$N_{u,v}^{w, \lambda_P }=N_{u, v}^{w\omega_P\omega',   \lambda_B},$$
               where  $\omega'=\omega_{P'}$ with $\Delta_{P'}=\{\alpha_i\in \Delta_P~|~\langle  \alpha_i, \lambda_B\rangle =0\}$.
      \end{enumerate}
    \end{prop}

Thanks to Proposition \ref{comparison}, we have canonical
representatives of $W/W_P\times Q^\vee/Q^\vee_P$ in $W\times Q^\vee$
with respect to the pair $(\Delta, \Delta_P)$,
  which is a generalization of the case $W/W_P\overset{\simeq}{\rightarrow }W^P\subset W$. We will discuss them in more details in the next subsection.

 When $v$ is a simple reflection $s_i$, we  have the following (Peterson's) quantum  Chevalley formula  for
         $\sigma^u\star \sigma^{s_i}$, which  has been   proved earlier in \cite{fw}.
 \begin{prop}[Quantum Chevalley Formula for $G/B$]\label{quanchevalley}
   For $u\in  W, 1\leq i\leq n,$
     $$ \sigma^u\star\sigma^{s_i} =\sum_\gamma \langle \chi_i, \gamma^\vee\rangle \sigma^{us_\gamma}+\sum_\gamma \langle \chi_i, \gamma^\vee\rangle q_{\gamma^\vee}\sigma^{us_\gamma},$$
   where the first sum is over  roots $\gamma$ in $R^+$ for which $\ell(us_\gamma)=\ell(u)+1$,  and the second sum is over roots
   $\gamma$ in $R^+$ for which $ \ell(us_\gamma)=\ell(u)+1-\langle 2\rho, \gamma^\vee\rangle$.
  \end{prop}

Note that we have fixed a base $\Delta=\{\alpha_1, \cdots, \alpha_n\}$. As a subset of $\Delta$, we can write $\Delta_P=\{\alpha_{i_1}, \cdots, \alpha_{i_r}\}$.
  Then giving an order on $\Delta_P$ is equivalent to  giving a permutation of $\Delta_P$. Once such a  permutation $\Upsilon$ is given,
   we denote  $\alpha'_j =\Upsilon(\alpha_{i_j})$ for each $1\leq j\leq r$ and then  naturally rewrite the remaining simple roots so that
    $\Delta=\{\alpha'_1,\cdots, \alpha_n'\}$. In the present paper, we always keep the information on the order, whenever referring to $(\Delta_P, \Upsilon)$ or
     (an ordered set) $\Delta_P=(\alpha_1', \cdots, \alpha_r')$. Furthermore for convenience, we simply denote $\alpha_j'$'s by $\alpha_j$'s.
      (In other words, we take $\Upsilon=\mbox{id}_{\Delta_P}$ under the assumption in the beginning that  $\Delta_P=\{\alpha_1, \cdots, \alpha_r\}$ satisfies
         certain properties on its associated Dynkin diagram.)

\begin{notation}  \label{definofu}
  Let $(\Delta_P, \Upsilon)$ be given with $\Delta_P=(\alpha_1,   \cdots, \alpha_r)$.

    For any integers  $k, m$   with $1\leq k\leq m\leq r$, we denote  $u_{[k, m]}^{\Delta_P, (r)}
                  :=s_{\alpha_k}s_{\alpha_{k+1}}\cdots s_{\alpha_m}.$
    If $k>m$, then we just denote  $u_{[k, m]}^{\Delta_P, (r)}:=1$.
                        Furthermore, we define
       $u_{[k, m]}^{\Delta_P, (m)}=u_{[k, m]}^{\Delta_P, (r)}$ and denote  $u_i^{\Delta_P, (m)}:= u_{[m-i+1, m]}^{\Delta_P, (m)}$ for $i=0, 1, \cdots, m.$
   Whenever there is no confusion, we simply denote
      $$s_{j}:=s_{\alpha_j},\quad u_{[k, m]}^{(r)}:=u_{[k, m]}^{\Delta_P, (r)} \quad \mbox{and} \quad
        u_{i}^{(m)}:= u_{i}^{\Delta_P, (m)}.$$

  Let $\Delta_j:=\{\alpha_1, \cdots, \alpha_j\}$ and   $P_j:=P_{\Delta_j}$   for each $1\leq j\leq r$. Denote $P_0=B$ and $P_{r+1}=G$.
    A  \textbf{decomposition of $w\in W$ associated to $(\Delta_P, \Upsilon)$} is an expression
      $w=v_{r+1}\cdots v_1$ with  $v_i\in W_{P_{i}}^{P_{i-1}}$ for   each $1\leq i\leq r+1$, where $W_{P_1}^{P_0}=W_{P_1}$.
   By the  \textbf{iterated fibration    associated to   $(\Delta_P, \Upsilon)$}, we mean
 the family of fibrations of  homogeneous varieties, given by
         $\{P_{j-1}/P_0\rightarrow P_j/P_0\longrightarrow P_j/P_{j-1}\}_{j=2}^{r+1}$.
\end{notation}

 We denote by $Dyn(\Delta')$  the Dynkin diagram associated to a base $\Delta'$.

\begin{example}
  Suppose $Dyn(\Delta_P)$   is given by  \begin{tabular}{l} \raisebox{-0.4ex}[0pt]{$  \circline\!\;\!\!\circ\cdots\, \circline\!\!\!\;\circ $}\\
                 \raisebox{1.1ex}[0pt]{${\hspace{-0.2cm}\scriptstyle{\alpha_1}\hspace{0.3cm}\alpha_2\hspace{0.9cm}\alpha_{r} } $}
  \end{tabular}\!.
   Consider the  iterated fibration $\{P_{j-1}/P_0\rightarrow P_j/P_0\longrightarrow P_j/P_{j-1}\}_{j=2}^{r+1}$ associated to
   $\Delta_P=(\alpha_1, \cdots, \alpha_r)$. Then we have $P_{r+1}/P_r=G/P$ and $P_j/P_{j-1}=\mathbb{P}^j$ for each $1\leq j\leq r$.
  Furthermore, the natural
 inclusion $\{\alpha_1, \cdots, \alpha_{r-1}\}\hookrightarrow \Delta_P$ (or  $SL(r, \mathbb{C})$ $\hookrightarrow SL(r+1, \mathbb{C})$)
   induces a  canonical embedding
    $P_{r-1}/B=F\ell_{r-1}\hookrightarrow F\ell_{r}=P/B$ of   complete flag varieties, which  maps a flag $V_1\leqslant \cdots \leqslant V_{r-1}$ in $\mathbb{C}^{r}$ to
     the flag $V_1\leqslant \cdots \leqslant V_{r-1}\leqslant \mathbb{C}^{r}$ in $\mathbb{C}^{r+1}$.
\end{example}

  Due to the following well-known lemma (see e.g. \cite{humr}), we obtain Corollary \ref{uniqexprforWeyl}.

\begin{lemma}\label{strongexchange}
    Let  $\gamma\in R^+$ and  {\upshape $w=s_{i_1}\cdots s_{i_\ell}$} be a reduced expression of $w\in W$.
  \begin{enumerate}
     \item $w\in W^P$ if and only if $w(\alpha)\in R^+$ for   any $\alpha\in \Delta_P$. 
    \item    If  $\ell(ws_\gamma)<\ell(w)$,
  then $w(\gamma)\!\!\in\!\!-R^+$ and  there is a  unique $1\!\leq\! k \!\leq\! \ell$ such that
                    $$ s_{i_k} \cdots s_{i_\ell}s_\gamma = s_{i_{k+1}}\cdots s_{i_{\ell}}\quad\mbox{and}\quad \gamma=s_{i_\ell}s_{i_{\ell-1}}\cdots s_{i_{k+1}}(\alpha_{i_k}).$$
          Furthermore for $1\leq j\leq n$,   $\ell(ws_j)=\ell(w)-1$ if and only if $w(\alpha_j)\in -R^+$.
   \end{enumerate}
  \end{lemma}

  \begin{cor}\label{uniqexprforWeyl}
For  each $w\in W$, there exists a unique decomposition
$w=v_{r+1}\cdots v_1$ associated to $\Delta_P=(\alpha_1,   \cdots, \alpha_r)$.
 Furthermore, we assume that $Dyn(\{\alpha_1, \cdots, \alpha_m\})$  is given by
  \begin{tabular}{l} \raisebox{-0.4ex}[0pt]{$  \circline\!\;\!\!\circ\cdots\, \circline\!\!\!\;\circ $}\\
                 \raisebox{1.1ex}[0pt]{${\hspace{-0.2cm}\scriptstyle{\alpha_1}\hspace{0.3cm}\alpha_2\hspace{0.9cm}\alpha_m } $}
  \end{tabular}\!, where $m\leq r$. Then for each $1\leq j\leq m$,  $\ell(v_j)=i_j$   if and only if
     $v_j= u_{i_j}^{(j)}$. (In particular, the expression $u_{i_j}^{(j)}$ itself is reduced.)
\end{cor}
\begin{proof}
  The former is also well-known (see e.g. \cite{humr}). The latter statement is a direct consequence of Lemma \ref{strongexchange},
   by noting $|W_{P_{j}}^{P_{j-1}}|=j+1$ and $u_0^{(j)}, \cdots, u_{j}^{(j)}$ are distinct elements of $W_{P_{j}}$ for which
    (1) of Lemma \ref{strongexchange}   can be applied.
\end{proof}

The following lemma should also be well-known.

\begin{lemma}\label{weylonelemma}
  Let $\bar \Delta\subset \tilde \Delta\subset \Delta$, $\tilde P=P_{\tilde \Delta}$ and $\bar P=P_{\bar \Delta}$.
 Let  $w=vu$ with $u\in W_{\bar P}$ and $v=s_{i_1}\cdots s_{i_m}$ being a reduced expression of $v\in W_{\tilde P}^{\bar P}$.
    For any $1\leq j\leq m$,
    we have $v':=s_{i_{j+1}}\cdots s_{i_m}\in W_{\tilde P}^{\bar P}$  and $(v'u)^{-1}(\alpha_{i_j})\in R_{\tilde P}^+\setminus R_{\bar P}$.
\end{lemma}
\begin{proof}
     Assume that the set $\{a~|~ s_{i_{a+1}}s_{i_{a+2}}\cdots s_{i_m}\notin W^{\bar P}_{\tilde P}, 1\leq a\leq m\}$ is non-empty.
     Then we can take
      the minimum $k$ of this set. Consequently,  $w':= s_{i_{k+1}}\cdots s_{i_m}\notin W^{\bar P}_{\tilde P}$ and $s_{i_k}w'\in W^{\bar P}_{\tilde P}$.
     Hence, there exists $\alpha\in \bar\Delta$ such that  $w'(\alpha)\in-R^+$ and
           $s_{i_k} w'(\alpha)\in R^+$. Since $s_{i_k}$ preserves $-R^+\setminus \{-\alpha_{i_k}\}$, we have
               $w'(\alpha)=-\alpha_{i_k}$. Thus $w' s_\alpha w'^{-1}=s_{-\alpha_{i_k}}=s_{i_k}$ so that $\ell(w' s_\alpha)=\ell(s_{i_k}  w')=
                \ell(w')+1$.
                This    implies $w' (\alpha)\in R^+$ by Lemma \ref{strongexchange} and therefore deduce a contradiction.
      Hence, for any $1\leq j\leq m$, we have  $v':=s_{i_{j+1}}\cdots s_{i_m}\in W_{\tilde P}^{\bar P}$.

   Note that $(v'u)^{-1}(\alpha_{i_j})\in R_{\tilde P}^+$ and $\gamma:=v'^{-1}(\alpha_{i_j})\in R_{\tilde P}^+$.
    We claim  $\gamma\notin R_{\bar P}$; otherwise we would conclude
         $v's_\gamma(\gamma)=-v'(\gamma)\in -R_{\bar P}^+$, contrary to $v's_\gamma(\gamma)=s_{i_j}v'(\gamma)\in R^+$.
    Since $u\in W_{\bar P}$, we have  $(v'u)^{-1}(\alpha_{i_j})=u^{-1}(\gamma)\notin R_{\bar P}$.
\end{proof}

  \subsection{Definition of gradings}\label{secdefgrading}
   In this    subsection,   we define a grading map $gr$ with respect to an ordered set $(\Delta_P, \Upsilon)$,  which is used for        constructing a filtration  on               $QH^*(G/B)$.
    In order to obtain $gr$, we first define   ``PW-lifting" (Peterson-Woodward lifting) as follows.

\begin{defn}
\label{defgrading} Given $(\Delta_P, \Upsilon)$ with $\Delta_P=(\alpha_1, \cdots, \alpha_r)$, we denote $\Delta_j=\{\alpha_1, \cdots,$ $ \alpha_j\}$,
  $P_j=P_{\Delta_j}$ and $Q_j^\vee=\bigoplus_{i=1}^j\mathbb{Z}\alpha_i^\vee$ for each $j\leq r$.
   By the \textbf{PW-lifting} associated to $(\Delta_P, \Upsilon)$, we mean the family $\{\psi_{\Delta_{j+1}, \Delta_j}\}_{j=1}^r$ of injective maps
    defined as follows. (We denote $Q_{r+1}^\vee=Q^\vee, \Delta_{r+1}=\Delta$ and $P_{r+1}=G$.)
          For each  $1\leq j\leq r$, the map
  $$\psi_{\Delta_{j+1}, \Delta_j}: W_{P_{j+1}}^{P_j}\times Q_{j+1}^\vee/Q_j^\vee \longrightarrow  W\times Q^\vee
               $$
       is defined by sending $(v, \bar{\lambda})$ to its associated elements $(v\omega_{P_j}\omega_{P_j'}, \lambda')$ as described by
      the Peterson-Woodward comparison formula (see  Proposition \ref{comparison}) with respect to $(\Delta_{j+1}, \Delta_j)$.
    That is, $\lambda' $ is the
               unique  element in $Q_{j+1}^\vee\subset Q^\vee$ satisfying  $\bar\lambda=\lambda'+Q_j^\vee$ and
                $\langle \alpha, \lambda'\rangle  \in \{0, -1\}$ for all $\alpha\in  R^+\cap \bigoplus_{i=1}^j\mathbb{Z}\alpha_i$;
                                $\Delta_{P'_j}=\{  \alpha  \in \Delta_j~|~ \langle  \alpha, \lambda'\rangle =0\}$.
\end{defn}
\begin{remark}
   Each $\psi_{\Delta_{j+1}, \Delta_j}$ also defines an injective map in the canonical way:
      $$QH^*(P_{j+1}/P_j)\longrightarrow QH^*(P_{j+1}/B); q_{\bar \lambda}\sigma^v\mapsto q_{\lambda'}\sigma^{v\omega_{P_j}\omega_{P_j'}}.$$
 \end{remark}

 Recall that a natural basis of $QH^*(G/B)[q_1^{-1}, \cdots, q_n^{-1}]$ is given by  $q_\lambda\sigma^{w}$'s labelled by $(w, \lambda)\in W\times Q^\vee$. We simply denote
  both of them as $q_\lambda w$ (or $wq_\lambda$) by abuse of notations.   Note that
   $q_\lambda w\in QH^*(G/B)$ if and only if $q_\lambda\in \mathbb{Q}[\mathbf{q}]$ is a polynomial.

\vspace{0.2cm}

 \noindent\textbf{Definition \ref{defgrading}.\,\,(continued)} 
{\itshape  Let $\{\mathbf{e}_{1},
                                              \cdots, \mathbf{e}_{r+1}\}$ be the standard basis of $\mathbb{Z}^{r+1}$.
           We define a \textbf{grading map}     $gr: W\times Q^\vee  \longrightarrow \mathbb{Z}^{r+1}$ associated to
            $(\Delta_P, \Upsilon)$ as follows.

  \begin{enumerate}
     \item For $w\in W$, we take its (unique) decomposition $w=v_{r+1}\cdots v_1$ associated to  $(\Delta_P, \Upsilon)$. Then we define
         $gr(w):=gr(w, 0)= \sum_{j=1}^{r+1}\ell(v_j)\mathbf{e}_{j}$. 
    \item For all $\alpha\in \Delta$, we simply denote   $gr(q_{\alpha^\vee}):=gr(1, q_{\alpha^\vee})$.
    Using the PW-lifing associated to  $(\Delta_P, \Upsilon)$, we can define all $gr(q_j)$'s recursively  in the following way.
    Define $gr(q_1)=2\mathbf{e}_1$; for any $\alpha\in \Delta_{j+1}\setminus \Delta_j$,  we define
                $$gr(q_{\alpha^\vee})= \big( \ell(\omega_{P_j}\omega_{P_j'})+2+\sum\nolimits_{i=1}^j2a_i\big)\mathbf{e}_{j+1}-gr(\omega_{P_j}\omega_{P_j'})-
                \sum\nolimits_{i=1}^ja_igr(q_i),$$
     where  $\omega_{P_j}\omega_{P_j'}$ and  $a_i$'s satisfy
          $(\omega_{P_j}\omega_{P_j'}, \alpha^\vee\!\!+\! \sum\limits_{i=1}^j a_i\alpha_i^\vee)\!\!=\psi_{\Delta_{j+1}, \Delta_j}(1, \alpha^\vee+Q_j^\vee)$.
      \item In general, $x=w\prod_{k=1}^nq_k^{b_k}$, then we define  $gr(x)=gr(w)+\sum_{k=1}^nb_kgr(q_k)$.
  \end{enumerate}

Furthermore for $1\leq k\leq m\leq |\Delta_P|$, we define $$
 gr_m:=gr_{[1, m]}\quad \mbox{with}\quad gr_{[k, m]}: W\times Q^\vee \rightarrow \mathbb{Z}^{m-k+1}$$ being the composition of the natural projection map
 and the grading map $gr$. Precisely,
 write $gr(q_\lambda w)=\sum_{i=1}^{r+1}d_i\mathbf{e}_i$, then we define $gr_{[k, m]}(q_\lambda w)=\sum_{i=k}^{m}d_i\mathbf{e}_i$.
}

Recall that the \textit{inversion set} of $w\in W$ is defined to be
     $$\mbox{Inv}(w)=\{\gamma\in R^+~|~ w(\gamma)\in -R^+\}.$$
It is well-known that $\ell(w)=|\mbox{Inv}(w)|$ (see e.g. \cite{humr}). Take the decomposition $w=v_{r+1}\cdots v_1$ of $w$ associated to $(\Delta_P, \Upsilon)$. For each
 $k$, we note $v_{r+1}\cdots v_{k+1}\in W^{P_k}$ and $v_{k}\cdots v_1\in W_{P_k}$.
 Thus for $\gamma\in R_{P_k}$, $v_{k}\cdots v_1(\gamma)\in -R^+$ if and only if $w(\gamma)\in -R^+$. Consequently,
 $\ell(v_k\cdots v_1)=|\{\gamma\in R_{P_k}^+~|~ w(\gamma)\in -R^+\}|=|\mbox{Inv}(w)\cap R_{P_k}^+|$. Note that $\ell(v_k\cdots v_1)=\sum_{i=1}^k\ell(v_k)$.
 Hence, we have
       $$gr(w)=\sum_{k=1}^{r+1}|\mbox{Inv}(w)\cap (R_{P_k}^+\setminus R_{P_{k-1}}^+)|\mathbf{e}_k.$$

\begin{remark}
       We would like to thank the referee for reminding us of the above  expression of $gr(w)$.
      Following the suggestions of  the referee,  the  proof of Proposition \ref{maingracompare} has been simplified substantially  in the present version.
          In type $A$, the vector $gr(w)$ is essentially what is known as an ``inversion table"     (see e.g. \cite{stanley}).
      The referee has also made the following  conjecture:
        $$gr(q_{\gamma^\vee})=\sum_{k=1}^{r+1}\langle \sum_{\beta\in R_{P_k}^+\setminus R_{P_{k-1}}^+}\beta, \gamma^\vee\rangle \mathbf{e}_k. $$
    If it is true, the proofs of our main results   might also be simplified substantially.
\end{remark}

\bigskip

      In Proposition \ref{graquanvar}, Proposition \ref{gradingforconn}, Lemma \ref{gengralemm111} and (the proof of) Lemma \ref{gengralemm222},
       we will explicitly describe all the gradings $gr(q_j)$'s with respect to a fixed    $(\Delta_P, \Upsilon)$. In particular, we will see that
       $gr(q_j)=(1-j)\mathbf{e}_{j-1}+(1+j)\mathbf{e}_{j}$ for $2\leq j\leq r-1$ (which also holds for $j=r$ if $\Delta_P$ is of $A$-type).

\subsection{Proof of Theorem \ref{mainthm}}\label{subsecmainthm} Assuming $Dyn(\Delta_P)$ is connected, we always consider $(\Delta_P, \Upsilon)$
  with the fixed order $\Delta_P=(\alpha_1, \cdots, \alpha_r)$ in  a special way as it will be explained   in section \ref{arrangement}.
  In this subsection,   we  construct a filtration  on
               $QH^*(G/B)$ with respect to a totally-ordered sub-semigroup $S$ of $\mathbb{Z}^{r+1}$ and prove Theorem \ref{mainthm}, which is the most essential part of
               our main results.

Unless otherwise stated, we will always use the \textbf{lexicographical
order}, whenever referring to a partial order
   on  (a sub-semigroup of)  $\mathbb{Z}^m$ in the present paper.
 (Recall that
    $\mathbf{a}<\mathbf{b}$, where  $\mathbf{a}=(a_1, \cdots, a_m)$ and $\mathbf{b}=(b_1, \cdots, b_m)$,
   if and only if there is $1\leq j\leq m$ such that   $a_j<b_j$ and $a_{k}=b_{k}$ for each $1\leq k< j$.)

\begin{defn}
  We define a subset $S$ of $\mathbb{Z}^{r+1}$ and a family    $\mathcal{F}=\{F_{\mathbf{a}}\}_{\mathbf{a}\in S}$ of subspaces of $QH^*(G/B)$ as follows:
   $$ S \triangleq\{gr(q_{\lambda}w)~|~ q_{\lambda} {w}\in QH^*(G/B)\};\quad
           F_{\mathbf{a}}\triangleq \bigoplus_{gr(q_{\lambda}w)\leq \mathbf{a}}\mathbb{Q}q_{\lambda}w\subset QH^*(G/B).$$

\end{defn}

As will be shown in section \ref{secproof22}, we have
\begin{lemma}\label{lemmasemigp}
   $S$ is a totally-ordered sub-semigroup of $\mathbb{Z}^{r+1}$.
\end{lemma}

Now we can state Theorem \ref{mainthm} more explicitly as follows.
\vspace{0.2cm}

\noindent\textbf{Theorem \ref{mainthm}.}
  {\itshape $QH^*(G/B)$ is an $S$-filtered algebra with filtration $\mathcal{F}$. Furthermore, this  $S$-filtered algebra structure
   is naturally   extended to a $\mathbb{Z}^{r+1}$-filtered algebra structure on $QH^*(G/B)$.
   }
\vspace{0.2cm}

That is, we need to show $  F_{\mathbf{a}}  F_{\mathbf{b}}\subset  F_{\mathbf{a+b}}$ for any ${\mathbf{a}},  {\mathbf{b}}\in S$. In order to prove it, we need to
 assume the following Key Lemma first.
 \vspace{0.2cm}

\noindent\textbf{Key Lemma.}
   {\itshape   Let $u\in W$ and $\gamma\in R^+$.
   \begin{enumerate}
     \item[a)]   If $\ell(us_\gamma)=\ell(u)+1$, then we have  $gr(us_\gamma) \leq  gr(u)+ gr(s_i)$ whenever the fundamental weight $\chi_i$ satisfies
                  $\langle \chi_i, \gamma^\vee\rangle\neq0$.
     \item[b)]   If $\ell(us_\gamma)=\ell(u)+1-\langle 2\rho, \gamma^\vee\rangle$, then we have
      $gr(q_{\gamma^\vee}us_\gamma) \leq  gr(u)+ gr(s_i)$ whenever $\langle \chi_i, \gamma^\vee\rangle\neq0$.

   \end{enumerate}
 }

\begin{lemma}\label{lemmaformainthm}
  For any $1\neq w\in W$, there exist   $w'\in W$ and $1\leq j\leq n$ such that
             $gr(w)=gr(w')+gr(s_j)$ and the quantum structure constant $N_{s_j, w'}^{w, 0}$ is positive.
\end{lemma}

\begin{proof}
 Take the decomposition $w=v_{r+1}\cdots v_1$ of $w$ associated to  $(\Delta_P, \Upsilon)$. Since $w\neq 1$, the
  set $\{i~|~ \ell(v_i)>0\}$   is non-empty, so that we can take the minimum $k$ of this set.  Thus we have $v_1=\cdots=v_{k-1}=1$ and
      $v_k=s_p\bar v$ with   $\ell(s_p\bar v)=1+\ell(\bar v)$. Note that $\gamma:=\bar v^{-1}(\alpha_p)\in R_{P_k}^+$,
        and   $\bar v s_\gamma=v_k$. Consequently for  $w':=v_{r+1}\cdots v_{k+1}\bar v$, we have $w=w's_\gamma$  and
       $\ell(w' s_\gamma)=\ell(w)+1$.  
 By Lemma \ref{weylonelemma}, we have    $\bar v\in W^{P_{k-1}}_{P_k}$ and  $\gamma \not\in R^+_{P_{k-1}}$. Hence, there exists $1\leq j\leq n$ with $\alpha_j\in \Delta_k\setminus\Delta_{k-1}$
    such that $\langle \chi_j, \gamma^\vee\rangle>0$.
 For any one such $j$,   by Proposition \ref{quanchevalley} we have
    $N_{s_j, w'}^{w, 0}=N_{s_j, w'}^{w's_\gamma, 0}=\langle \chi_j, \gamma^\vee\rangle>0$. Furthermore, we have
     $gr(w)=\sum_{i=k}^{r+1}\ell(v_i)\mathbf{e}_i=(\ell(\bar v)\mathbf{e}_k+\sum_{i=k+1}^{r+1}\ell(v_i)\mathbf{e}_i)+\mathbf{e}_k=gr(w')+gr(s_j)$.
\end{proof}

\bigskip

\begin{proof}[Proof of Theorem \ref{mainthm}]
  For the first half of the statements, it  suffices to show:
                    $\sigma^w\star q_\lambda \sigma^u\in F_{\mathbf{a}+\mathbf{b}}$,
                        for any $\sigma^w, q_\lambda \sigma^u\in QH^*(G/B)$ with $\mathbf{a}=gr(w)$ and $\mathbf{b}=gr(q_\lambda u)$.
           We use induction on $\ell(w)$.

   If $\ell(w)=0$, then $\sigma^w$ is the unit   and it is   done.
      If $\ell(w)=1$, then $w=s_j$ and consequently we have   $\sigma^{s_j}\star\sigma^u\in F_{gr(s_j)+gr(u)}=F_{\mathbf{a}+\mathbf{b}-gr(q_\lambda)}$,
          by using Proposition \ref{quanchevalley} and  the Key Lemma.
          Thus we have $\sigma^w\star q_\lambda \sigma^u\in  F_{\mathbf{a}+\mathbf{b}}$ in this case.
        Assume $\ell(w)>1$.
       By Lemma \ref{lemmaformainthm},    there exist $w'\in W$ and  $1\leq j\leq n$ such that
                     $gr(w)=gr(w')+gr(s_j)$ and $\sigma^{w'}\star \sigma^{s_j}=c \sigma^w+\sum_{ {v, \mu}}c_{v, \mu} q_{\mu}\sigma^{v}$,
         where $c=N_{w', s_j}^{w, 0}>0$ and the summation is  only over  finitely many non-zero terms  for which  $c_{v, \mu}>0$.
       In particular, we have $\ell(w')=\ell(w)-1$. Using the induction hypothesis, we have $\sigma^{w'}\star q_\lambda u \in F_{gr(w')+\mathbf{b}}$.
        Thus
   $(c\sigma^w +\sum_{ {v, \mu}}c_{v, \mu} q_{\mu}\sigma^{v})\star q_\lambda \sigma^u=\sigma^{s_j}\star(\sigma^{w'}\star q_\lambda \sigma^u) \in F_{gr(s_j)+gr(w')+\mathbf{b}}
       =  F_{\mathbf{a}+\mathbf{b}}$.
 Since   all the quantum Schubert structure constants are non-negative, there is no cancellation in the summation on the left hand side of the equality.
       Hence, we conclude $\sigma^w  \star q_\lambda \sigma^u\in F_{\mathbf{a}+\mathbf{b}}$, by noting $c>0$.

  The second half is a direct consequence of the first half. Indeed,  $QH^*(G/B)$ has  a  $\mathbb{Z}^{r+1}$-filtration
    $\{F_{\mathbf{a}}\}_{\mathbf{a}\in \mathbb{Z}^{r+1}}$, which is a natural extension of $\mathcal{F}$.
  Here we just need to  set $F_{\mathbf{a}}:=\bigcup_{\mathbf{b}\leq \mathbf{a}, \mathbf{b}\in S} F_\mathbf{a}$
    for any $\mathbf{a}\in \mathbb{Z}^{r+1}\setminus S$ (note $S$ is sub-semigroup of $\mathbb{Z}^{r+1}$).
\end{proof}

 The next proposition   follows directly from Definition \ref{defgrading}.

\begin{prop}\label{qredclas}
 The evaluation of    $\mathbf{q}$ at $\mathbf{0}$ reduces the $\mathbb{Z}^{r+1}$-filtration on $QH^*(G/B)$ to the classical
    $\mathbb{Z}^{r+1}$-filtration on $H^*(G/B)$, which comes from the iterated fibration
      $\{P_{j-1}/P_0\rightarrow P_j/P_0\longrightarrow P_j/P_{j-1}\}_{j=2}^{r+1}$. (Recall that $P_0=B$ and $P_{r+1}=G$.)
\end{prop}

 \subsection{A canonical order   $(\Delta_P, \Upsilon)$}\label{arrangement}
  When referring to $(\Delta_P, \Upsilon)$, we have already
     given an order on  $\Delta_P$ via the permutation $\Upsilon$.
     It is done if $r=1$, since $\Upsilon=\mbox{id}_{\Delta_P}$ is the only permutation map.
     In this subsection, we introduce the special choice of the orders for $r\geq 2$ as mentioned at the beginning of section \ref{subsecmainthm}.
     We will  use this special order throughout the
      present paper, which is in fact obtained in a canonical way.
  We  introduce it first for a subbase of $A$-type  and then   for others    by reducing them to the case for $A$-type.

 Suppose  $\Delta_P$ is of $A_r$-type.  We   rewrite the simple roots so that $\Delta=\{\beta_1, \cdots, \beta_n\}$ and $Dyn(\Delta)$ is given by
        one of the cases in Table \ref{tabrelativeposi}.
       In terms of the order   $(\beta_1, \cdots, \beta_n)$, we  obtain a canonical  order $(\Delta_P, \Upsilon)$,
         in the sense that   $Dyn(\Delta_P)$ is inside $Dyn(\Delta\setminus \{\mbox{marked points}\})$ in a natural way. That is, we require the condition $(*)$ to be satisfied.
           $$(*):\qquad \mbox{there exists } o\geq 0 \mbox{ such that } \alpha_j=\beta_{o+j}   \mbox{ for each } 1\leq j\leq r. $$
 Furthermore, the additional conditions in Table \ref{tabrelativeposi} tell us the information on the starting point $\alpha_1 (=\beta_{o+1})$ and the ending point
   $\alpha_r (=\beta_{\kappa}=\beta_{o+r})$.  For instance, any one case of $\mbox{C}8), \mbox{C}9)$ and $\mbox{C}10)$ implies that $o=0$ and $\Delta_P=(\alpha_1, \alpha_2)=(\beta_1, \beta_2)$. That is,
    the order of $\Delta_P=\{\alpha_1, \alpha_2\}$
    is expressed in terms of the order of $\{\beta_1, \beta_2\}$ with respect to the corresponding case.

 \begin{table}[h]
 \caption{\label{tabrelativeposi}   $(\Delta_P, \Upsilon)$ when $r\geq 2$}
   \begin{tabular}{|c|c|c|}
     \hline
      & Dynkin diagram of $\Delta$ &$\mbox{}$\!\! Additional conditions ($\kappa:=o+r$)\!\!\!\! $\mbox{}$\\
     \hline\hline
     $\mbox{C} 1)$ &  \begin{tabular}{l} \raisebox{-0.4ex}[0pt]{$  \circline\!\;\!\!\circ\cdots\, \circline\!\!\!\;\circ $}\\
                 \raisebox{1.1ex}[0pt]{${\hspace{-0.2cm}\scriptstyle{\beta_1}\hspace{0.3cm}\beta_2\hspace{0.7cm}\beta_{n-1} } $}
                \end{tabular}\!\!,  $\beta_n$ is adjacent to $\beta_{n-1}$  & $\begin{array}{c}  \kappa\leq n-1;  \\ \Delta  \mbox{ is of type } A_n,  B_n
                            \mbox{ or } C_n\end{array}$  \\ \hline
          $\mbox{C} 2)$ &  \begin{tabular}{l}  \vspace{-0.15cm} $\hspace{1.649cm}\bullet \scriptstyle{\beta_{n}}$ \\
                     \vspace{-0.20cm}{ $\hspace{1.643cm}\!\!\;\!\big|$} \\
                   \raisebox{-0.4ex}[0pt]{$ \hspace{-0.02cm} \circline\!\;\!\!\circ\cdots\, \circline\circline\!\!\;\circ$}\\
                 \raisebox{1.1ex}[0pt]{${\hspace{-0.2cm}\scriptstyle{\beta_1}\hspace{0.3cm}\beta_2\hspace{0.7cm} \beta_{n-2}\hspace{0.1cm}\beta_{n-1}} $}
                \end{tabular}      &     $\kappa\leq n-2$ or $\Big\{\begin{array}{l} r\geq 3\\ \kappa=n-1\end{array}$ \\ \hline
        $\mbox{C} 3)$ &  \begin{tabular}{l}  \vspace{-0.15cm} $\hspace{1.649cm}\circ \scriptstyle{\beta_{3}}$ \\
                     \vspace{-0.20cm}{ $\hspace{1.643cm}\!\!\;\!\big|$} \\
                   \raisebox{-0.4ex}[0pt]{$\hspace{-0.17cm}  \bulletline\!\;\!\!\bullet\cdots\, \bulletline\circline\!\!\;\circ$}\\
                 \raisebox{1.1ex}[0pt]{${\hspace{-0.1cm}\scriptstyle{\beta_n}\hspace{0.1cm}\beta_{n-1}\hspace{0.2cm}\beta_4\hspace{0.2cm} \beta_{2}\hspace{0.2cm}\beta_{1}} $}
                \end{tabular}      &    $\kappa=r\leq 3$ \\ \hline
       $\mbox{C} 4)$ &  \begin{tabular}{l}  \vspace{-0.168cm} $\hspace{2.083cm}\!\bullet \scriptstyle{\beta_{8}}$ \\
                     \vspace{-0.20cm}{ $\hspace{2.038cm}\!\!\;\!\!\;\!\big|$} \\
                   \raisebox{-0.4ex}[0pt]{$ \, \circline\circline\circline\circline\circline\circline\!\!\;\!\,\!\circ$}\\
                 \raisebox{1.1ex}[0pt]{${\hspace{-0.15cm}\scriptstyle{\beta_1}\hspace{0.2cm}\beta_2 \hspace{0.2cm} \beta_3\hspace{0.25cm}\beta_{4} \hspace{0.2cm}\beta_{5} \hspace{0.2cm}\beta_{6}\hspace{0.2cm}\beta_{7}} $}
              \end{tabular}  & $\mbox{}$ \hspace{-0.4cm} $\kappa\leq 5$ or $\Big\{\begin{array}{l} r\geq 3\\ \kappa=6\end{array}$ or $\Big\{\begin{array}{l} r\geq 5\\ \kappa=7\end{array}$\hspace{-0.4cm}$\mbox{}$\\  \hline
    $\mbox{C}5)$ &  \begin{tabular}{l}  \vspace{-0.166cm} $\hspace{2.06cm}\!\circ \scriptstyle{\beta_{4}}$ \\
                     \vspace{-0.20cm}{ $\hspace{2.025cm}\!\!\;\!\!\;\!\big|$} \\
                   \raisebox{-0.4ex}[0pt]{$ \hspace{-0.035cm}\bulletline\!\;\!\!\bulletline\!\;\!\!\bulletline\!\;\!\!\bulletline\circline\circline\!\!\;\!\,\!\circ$}\\
                 \raisebox{1.1ex}[0pt]{${\hspace{-0.01cm}\scriptstyle{\beta_8}\hspace{0.2cm}\beta_7 \hspace{0.2cm} \beta_6\hspace{0.25cm}\beta_{5} \hspace{0.2cm}\beta_{3} \hspace{0.2cm}\beta_{2}\hspace{0.2cm}\beta_{1}} $}
              \end{tabular}  & $\mbox{}$ \hspace{-0.4cm} $\kappa\leq 3$ or $\Big\{\begin{array}{l} r\geq 3\\ \kappa=4\end{array}$  \hspace{-0.4cm}$\mbox{}$\\  \hline
 $\mbox{C} 6)$ &  \begin{tabular}{l}  \vspace{-0.165cm} $\hspace{2.078cm}\!\bullet \scriptstyle{\beta_{8}}$ \\
                     \vspace{-0.20cm}{ $\hspace{2.033cm}\!\!\;\!\!\;\!\big|$} \\
                   \raisebox{-0.4ex}[0pt]{$\hspace{-0.01cm}\bulletline\!\;\!\!\bulletline\!\;\!\!\bulletline\circline\circline\circline\!\!\;\!\,\!\circ$}\\
                 \raisebox{1.1ex}[0pt]{${\hspace{-0.01cm}\scriptstyle{\beta_7}\hspace{0.2cm}\beta_6 \hspace{0.2cm} \beta_5\hspace{0.22cm}\beta_{4} \hspace{0.2cm}\beta_{3} \hspace{0.2cm}\beta_{2}\hspace{0.2cm}\beta_{1}} $}
              \end{tabular}  & $\mbox{}$ \hspace{-0.4cm} $\kappa=r=4$     \hspace{-0.4cm}$\mbox{}$\\  \hline
 $\mbox{C} 7)$ &  \begin{tabular}{l}  \vspace{-0.165cm} $\hspace{2.073cm}\!\circ \scriptstyle{\beta_{6}}$ \\
                     \vspace{-0.20cm}{ $\hspace{2.033cm}\!\!\;\!\!\;\!\big|$} \\
                   \raisebox{-0.4ex}[0pt]{$ \, \circline\circline\circline\circline\circline\!\;\!\!\bulletline\!\!\;\!\,\!\bullet$}\\
                 \raisebox{1.1ex}[0pt]{${\hspace{-0.015cm}\scriptstyle{\beta_1}\hspace{0.2cm}\beta_2 \hspace{0.2cm} \beta_3\hspace{0.25cm}\beta_{4} \hspace{0.2cm}\beta_{5} \hspace{0.2cm}\beta_{7}\hspace{0.2cm}\beta_{8}} $}
              \end{tabular}  & $\mbox{}$ \hspace{-0.4cm} $\kappa=6, r\geq 3$   \hspace{-0.4cm}$\mbox{}$\\  \hline

 $\mbox{C} 8)$ &  \begin{tabular}{l}  \vspace{-0.165cm} $\hspace{2.073cm}\!\circ \scriptstyle{\beta_{1}}$ \\
                     \vspace{-0.20cm}{ $\hspace{2.033cm}\!\!\;\!\!\;\!\big|$} \\
                   \raisebox{-0.4ex}[0pt]{$ \hspace{-0.020cm}\bulletline\!\;\!\!\bulletline\!\;\!\!\bulletline\!\;\!\!\bulletline\circline\!\;\!\!\bulletline\!\!\;\!\,\!\bullet$}\\
                 \raisebox{1.1ex}[0pt]{${\hspace{-0.015cm}\scriptstyle{\beta_6}\hspace{0.2cm}\beta_5 \hspace{0.2cm} \beta_4\hspace{0.25cm}\beta_{3} \hspace{0.2cm}\beta_{2} \hspace{0.2cm}\beta_{7}\hspace{0.2cm}\beta_{8}} $}
              \end{tabular}  & $\mbox{}$ \hspace{-0.4cm} $\kappa=2$   \hspace{-0.4cm}$\mbox{}$\\  \hline
 $\mbox{C} 9)$ &  \begin{tabular}{l} \raisebox{-0.4ex}[0pt]{$\circline\,\bulletbb\!\!\;\!\!-\!\!\!-\!\bullet$}\\
                 \raisebox{1.1ex}[0pt]{${\hspace{-0.1cm}\scriptstyle{\beta_1}\hspace{0.2cm}\beta_2\hspace{0.25cm}\beta_3\hspace{0.2cm}\beta_{4}} $}
                 \end{tabular}  & $\kappa=2$ \\ \hline
       $\!\!\mbox{C}10)\!\!\!\mbox{}$ &  \begin{tabular}{l} \raisebox{-0.4ex}[0pt]{$\circline\,\bulletcc\!\!\;\!\!-\!\!\!-\!\bullet$}\\
                 \raisebox{1.1ex}[0pt]{${\hspace{-0.1cm}\scriptstyle{\beta_1}\hspace{0.2cm}\beta_2\hspace{0.25cm}\beta_3\hspace{0.2cm}\beta_{4}} $}
                 \end{tabular}  & $\kappa=2$ \\ \hline

   \end{tabular}
  \end{table}

 \begin{remark}
  In Table \ref{tabrelativeposi}, we have treated   bases of type $E_6$ and $E_7$ as  subsets of a base of type $E_8$ canonically.
   Because of our assumption $2\leq r<n=|\Delta|$, a base of $G_2$-type does not occur there.
 \end{remark}

\begin{remark}
  Intrinsically, we obtain the canonical order $(\Delta_P, \Upsilon)$ as follows.  $\Delta_P$ admits canonical orders in the sense that
    $Dyn(\Delta_P)$ is given by   \begin{tabular}{l} \raisebox{-0.4ex}[0pt]{$  \circline\!\;\!\!\circ\cdots\, \circline\!\!\!\;\circ $}\\

                 \raisebox{1.1ex}[0pt]{${\hspace{-0.2cm}\scriptstyle{\alpha_1}\hspace{0.3cm}\alpha_2\hspace{0.9cm}\alpha_r } $}
  \end{tabular}\!.
    There are   two ways to denote an ending point (by $\alpha_1$ or $\alpha_r$).
     We fix one  in the following way.
               There is at most one root in $\Delta_P$, saying $\alpha$, such that
      the Dynkin diagram of $\Delta_P\cup\{\alpha_k\in\Delta\setminus\Delta_P~|~ \langle  \alpha_k, \alpha^\vee\rangle \neq 0\}$ is not of $A$-type.
      We denote an ending point by $\alpha_1$ such that both the ending point and the connected component of $\Delta\setminus\Delta_P$ adjacent to it are as far away from
      $\alpha$ as possible.
\end{remark}

Comparing it with Table \ref{tabrelativeposi}, we can easily see
that $\Delta_P$ must occur  in at least one case of Table
\ref{tabrelativeposi} (together with
   condition $(*)$ being satisfied). If it occurs in more than one case, then we just choose any one of these cases.
    The choice does not affect all the results, since all the relevant statements hold with respect to all cases in Table \ref{tabrelativeposi} as we will see later.

\begin{remark}
  If $(\Delta, \Delta_P)$ occurs in more than one case in Table \ref{tabrelativeposi}, (for instance in case {\upshape $\mbox{C} 2)$} and {\upshape $\mbox{C}3)$ } with respect to the condition
    $\kappa=r=n-1=3$,) then the corresponding orders $(\Delta_P, \Upsilon)$ and $(\Delta_P, \Upsilon')$ are isomorphic. That is, there exists an isometry
     of $\phi: \Delta\rightarrow \Delta$ such that $\phi(\Delta_P)=\Delta_P$ and $\Upsilon\circ \phi=\Upsilon'$.
\end{remark}


Now we assume $\Delta_P$ is not of $A$-type and denote $\varsigma:=r-1$.
Note that there always exists $\alpha\in \Delta_P$ such that $Dyn(\Delta_P\setminus\{\alpha\})$ is of $A_\varsigma$-type. Thus when $r>2$,
we obtain a canonical order $(\Delta_P, \Upsilon)$ by requiring:
 \begin{enumerate}
   \item[a)] the restriction of $\Delta_P$ to  $\Delta_\varsigma=(\alpha_1, \cdots, \alpha_\varsigma)$ is the canonical order obtained by directly
              replacing $r$ with $\varsigma$ in Table \ref{tabrelativeposi};
   \item[b)] $\alpha_r=\beta_{o+r}$ (note that $\alpha_{r-1}=\beta_{o+r-1}$ once a) holds).
 \end{enumerate}
  Precisely,   $\Delta_P$  fulfills one and only one of the followings (note that $\kappa=o+\varsigma$ and condition $(*)$ is satisfied):
    \begin{enumerate}
      \item  $\Delta_P$ is not of $D$-type. It occurs in a unique case (among $\mbox{C} 1), \mbox{C}4) \mbox{ for } \kappa=7, \mbox{C}9) \mbox{ and } \mbox{C}10))$ in Table \ref{tabrelativeposi}.
      \item    $\Delta_P$ and $\Delta$ are both of $D$-type. It occurs in   case  $\mbox{C} 2)$.
       \item  $\Delta_P$ is of $D$-type and $\Delta$ is of $E$-type.   It occurs in  either of  cases  $\mbox{C} 5)$,  $\mbox{C}7)$.
    \end{enumerate}
  As a consequence, the canonical order $(\Delta_P, \Upsilon)$ is determined by the corresponding case in which $\Delta_P$ occurs. For convenience,  if $\Delta_P$ occurs in both
  $\mbox{C}5)$ and    $\mbox{C}7)$, then we always choose case $\mbox{C}7)$ for use.

  When $r=2$, we can still give an order on $\Delta_P$ so that it  is compatible with our arrangements for  $r>2$.
   Indeed, we do this as follows. Since $\Delta_P$ is a proper subset of $\Delta_P$, the case of $G_2$-type does not occur.
      Since $\Delta_P$ is not of $A$-type, $\Delta$ must be of type $B, C$ or $F$.  we   take $(\alpha_1, \alpha_2)$ to be
      $(\beta_{n-1}, \beta_{n})$ for the former two cases, or $(\beta_2, \beta_3)$ in $\mbox{C}10)$ for    the last case.

 \begin{remark}
   $\Delta_P$ occurs in case {\upshape $\mbox{C}5)$} other than in case {\upshape $\mbox{C}7)$} only if $r=5$ and $\Delta$ is of $E_7$-type or $E_8$-type.
    \end{remark}

\section{Proof of the Key Lemma}

 This whole section is devoted to the proof of the Key Lemma. The readers, who wish to see more concrete  statements of our theorems as well as their proofs,
    can skip this section by assuming  the Key Lemma and two consequences (Proposition \ref{gracomponemore}
  and Proposition \ref{allgradcomp}) of  a special case of  it first. For emphasis, we restate the Key Lemma as follows.

  \noindent\textbf{Key Lemma.}
   {\itshape   Let $u\in W$ and $\gamma\in R^+$.
   \begin{enumerate}
     \item[a)]   If $\ell(us_\gamma)=\ell(u)+1$, then we have  $gr(us_\gamma) \leq  gr(u)+ gr(s_i)$ whenever the fundamental weight $\chi_i$ satisfies
                  $\langle \chi_i, \gamma^\vee\rangle\neq0$.
     \item[b)]   If $\ell(us_\gamma)=\ell(u)+1-\langle 2\rho, \gamma^\vee\rangle$, then we have
      $gr(q_{\gamma^\vee}us_\gamma) \leq  gr(u)+ gr(s_i)$ whenever  the fundamental weight $\chi_i$ satisfies $\langle \chi_i, \gamma^\vee\rangle\neq0$.

   \end{enumerate}
 }

We first do some preparations in section \ref{subsecpropweyl} and section \ref{subsecgrading}. Then we prove the Key Lemma for the special case when $\Delta_P$ is of $A$-type
 in section \ref{subseckeylemma}, and obtain two consequences in section \ref{subsectwoconseq}. Finally in section \ref{sectiongenconn}, we prove the Key Lemma for general cases.
In addition, we also give the explicit descriptions of all $gr(q_j)$'s in section \ref{subsecgrading} and section \ref{sectiongenconn}.

We would like to remind our readers of the notation  $ gr(w)=\sum_{j=1}^{r+1}i_j\mathbf{e}_j=(i_1, \cdots, i_{r+1})$ for $w\in W$ and the notions ``$gr_m$",
 ``$gr_{[k, m]}$" in Definition \ref{defgrading}.
Furthermore,  we assume \textit{$\Delta_P$ to be of $A$-type   throughout this section except section \ref{sectiongenconn}}.
As a consequence, we have $w=u_{i_r}^{(r)}\cdots    u_{i_1}^{(1)}\in W_P$ by  Corollary \ref{uniqexprforWeyl}, once assuming  $i_{r+1}=0$.
  Unless otherwise stated,  by $w=vu_{i_r}^{(r)}\cdots    u_{i_1}^{(1)}$ we always mean the decomposition of $w$ associated  $(\Delta_P, \Upsilon)$ when $\Delta_P$
   is of $A$-type;
  equivalently, we have $v\in W^P$.

\subsection{Some properties on $W$}\label{subsecpropweyl}
The main results of this subsection   are Proposition
\ref{maingracompare} and Proposition \ref{gracompare000}, 
which compare  the gradings of certain elements in $W$.

  \begin{prop}\label{maingracompare}
   Let   $\gamma\in R^+$ satisfy
      $\langle \alpha, \gamma^\vee\rangle=0$ for all $\alpha\in \Delta_{\tilde P}=\Delta_P\setminus\{\alpha_a\}$,  
       where $1\leq a\leq r$.
   For any $w\in W$,  we have  $$gr_{a-1}(ws_\gamma)=gr_{a-1}(w) \mbox{ and } gr_{r}(ws_\gamma)\leq gr_{r}(w)+\sum\nolimits_{k=a}^r a\mathbf{e}_k.$$
       \end{prop}

\begin{lemma}\label{gracompare11111}
Let $\gamma\in R, \Delta_{\tilde P}\subset \Delta$
      and
   $w=vu \mbox{ with } v\in  W^{\tilde P} \mbox{ and}$ $u\in W_{\tilde P}$.
 If  $\langle \alpha_j, \gamma^\vee\rangle =0$ for all   $\alpha_j\in \Delta_{\tilde P}$,
         then $ws_\gamma=\tilde v u$ with $\tilde v\in W^{\tilde P}$.
In particular if $\Delta_{\tilde P}=\{\alpha_1, \cdots, \alpha_a\}$
where $a\leq r$,
 then  $gr_a(ws_\gamma)=gr_a(w)$.
  \end{lemma}
 \begin{proof}
       Let  $ws_\gamma=\tilde v\tilde u$ where $\tilde v\in  W^{\tilde P}$ and $\tilde u\in W_{\tilde P}$.
   By the assumption, we conclude $s_\gamma(\alpha_j)=\alpha_j$ and $s_js_\gamma=s_\gamma s_j$ for any $\alpha_j\in \Delta_{\tilde P}$. Hence,
    $us_\gamma=s_\gamma u$ and consequently we have $\tilde v \tilde u u^{-1}= ws_\gamma u^{-1}=wu^{-1}s_\gamma=vs_\gamma$.
   If $\tilde u\neq u$, then there exists $\beta\in R^+_{\tilde P}$ such that $\tilde \beta:=\tilde uu^{-1}(\beta)\in -R^+_{\tilde P}$.
   Hence, we conclude $vs_\gamma (\beta)=v(\beta)\in R^+$, contrary to
      $vs_\gamma(\beta)=\tilde v\tilde uu^{-1}(\beta)=\tilde v(\tilde \beta)\in-R^+$.
 The latter statement becomes a direct    consequence. 
  \end{proof}

 \bigskip

 \begin{proof}[Proof of Proposition \ref{maingracompare}]
 Write $gr(w)=\sum_{k=1}^{r+1}i_k\mathbf{e}_k$ and   $gr(ws_\gamma)=\sum_{k=1}^{r+1}\tilde i_k\mathbf{e}_k$.
      By Lemma \ref{gracompare11111}, we conclude
     $gr_{a-1}(ws_\gamma)=gr_{a-1}(w)$. That is, $\tilde i_k=i_k$ for $1\leq k\leq a-1$.

     Clearly, $\tilde i_a\leq a\leq i_a+a$. For $a+1\leq k\leq r$, we note that
     $R_{P_{k}}^+\setminus R_{P_{k-1}}^+=\{\sum_{t=j}^k\alpha_t~|~ 1\leq j\leq k\}$.
     In addition, we have $i_k=|\mbox{Inv}(w)\cap (R_{P_{k}}^+\setminus R_{P_{k-1}}^+)|$ and
             $\tilde i_k=|\mbox{Inv}(ws_\gamma)\cap (R_{P_{k}}^+\setminus R_{P_{k-1}}^+)|$.
        Since $\langle \alpha_t, \gamma^\vee\rangle=0$ for any $a+1\leq t\leq r$, we have $ws_\gamma(\sum_{t=j}^k\alpha_t)=w(\sum_{t=j}^k\alpha_t)$
         whenever $j\geq a+1$. Hence, $\tilde i_k-i_k\leq |\{\sum_{t=j}^k\alpha_t~|~ 1\leq j\leq a\}|=a$.

 Hence, we have $gr_r(ws_\gamma)\leq gr_r(w)+\sum_{k=a}^ra\mathbf{e}_k$.
    \end{proof}

\begin{lemma}\label{prodAtype}
  For any $1\leq i\leq j\leq m\leq r$ and $1\leq k\leq m$, we have
      $$u_{[i, j]}^{(m)}u_{[k, m]}^{(m)} = \left\{\begin{array}{ll} u_{[k, m]}^{(m)} u_{[i, j]}^{(m)}  ,&\mbox{if } k\geq j+2\\
                                  u_{[i, m]}^{(m)}, &\mbox{if } k=j+1\\
                                  u_{[k+1, m]}^{(m)} u_{[i, j-1]}^{(m)},&\mbox{if  }  i\leq k\leq j    \\
                              u_{[k, m]}^{(m)} u_{[i-1, j-1]}^{(m)},&\mbox{if }
                                  k<i\end{array}\right.\!\!.
            $$
\end{lemma}

For the above lemma, we recall that $u^{(m)}_{[i, j]}=s_is_{i+1}\cdots s_j$. As a direct  consequence, we obtain    the following grading comparisons.

\begin{prop}\label{gracompare000}
   Let  $w=u_{i_r}^{(r)}\cdots u_{i_1}^{(1)}$. Suppose  $j\leq m\leq r$. 
    \begin{enumerate}
      \item[a)] If $\ell(u_j^{\!(m)}\!w)\!=\!j+\ell(w)$, then $gr(u_{j}^{\!(m)}w)\!=\!gr(w)+j\mathbf{e}_k$ for a    unique $1\leq \!k\!\leq r$.
       \item[b)] If $\ell(s_jw)=\ell(w)-1$, then $gr(s_jw)=gr(w)-\mathbf{e}_k$ for a    unique $1\leq k\leq r$.
      \item[c)]  $\ell(ws_j)=\ell(w)-1$ if and only if  $i_j\geq i_{j-1}+1$ (where   $i_0:=0$). When this happens,
         we have $gr(ws_j)=\sum_{k=1}^{j-2} i_k\mathbf{e}_k+  (i_j-1)\mathbf{e}_{j-1}+ i_{j-1}\mathbf{e}_j+ \sum_{k=j+1}^{r} i_k\mathbf{e}_k$.
     \end{enumerate}
   Furthermore if    $w'\!\in\! W_P$ satisfies $\ell(w'w)\!=\!\ell(w)\pm\ell(w')$, then
          there exist non-negative integers $p_k$'s  such that $\sum_{k=1}^r p_k=\ell(w')$ and  $gr(w'w)=gr(w)\pm\sum_{k=1}^rp_k\mathbf{e}_k$.
\end{prop}
\begin{proof}
      Note that   $u_j^{(m)}u_{i_r}^{(r)}=u_{[m-j+1, m]}^{(m)}u_{[r-i_r+1, r]}^{(r)}=u_{[m-j+1, m]}^{(r)}u_{[r-i_r+1, r]}^{(r)}$.
    By Lemma \ref{prodAtype}, there are exactly four possibilities for this product. Since
     $\ell(u_j^{(m)}w)= j+\ell(w)$, the (third) case $m-j+1\leq r-i_r+1\leq m$ cannot occur. If $m=r-i_r$ (i.e. the second case occurs), then it is done by taking $\mathbf{e}_k=\mathbf{e}_r$.
     If $r-i_r+1\geq m+2$, we have   $m\leq r-1$ and $u_j^{(m)}u_{i_r}^{(r)}=u_{[r-i_r+1, r]}^{(r)}u_{[m-j+1, m]}^{(r)}=u_{i_r}^{(r)}u_{j}^{(m)}$;
     if $r-i_r+1<m-j+1$, we have $u_j^{(m)}u_{i_r}^{(r)}=u_{[r-i_r+1, r]}^{(r)}u_{[m-j, m-1]}^{(r)}=u_{[r-i_r+1, r]}^{(r)}u_{[m-j, m-1]}^{(m-1)}=u_{i_r}^{(r)}u_{j}^{(m-1)}$.
   That is, in either of the remaining two cases, we always have
        $u_j^{(m)}w=u_{i_r}^{(r)}u_j^{(m')}w'$ in which $\ell(u_j^{(m')}w')=j+\ell(w')$ with $m'\leq r-1$ and $w'=u_{i_{r-1}}^{(r-1)}\cdots u_{i_1}^{(1)}$.
      Hence, a) follows by induction.

      The arguments for the remaining parts of the statement are also easy and similar, which we leave to the readers.
\end{proof}
\bigskip

\begin{proof}[Proof of Lemma \ref{prodAtype}]
  Note that 
    $s_js_k=s_ks_j$ if $|j-k|\geq 2$, and $s_ks_js_k=s_js_ks_j$ if $|j-k|=1$.
   The first two cases are trivial.
   For   $1\leq k<b\leq m$, we have
    \begin{align*} s_b\cdot u_{[k, m]}^{(m)}=s_b\cdot(s_k \cdots s_m)
                     &= s_k\cdots s_{b-2}s_bs_{b-1}s_{b}s_{b+1}\cdots s_m\\
                     &= s_k \cdots s_{b-2}s_{b-1}s_bs_{b-1}s_{b+1}\cdots s_m\\
                     &=(s_k\cdots s_m)\cdot s_{b-1}=u_{[k, m]}^{(m)}s_{b-1}.
    \end{align*}
   Thus if $k<i$, then  $u_{[i, j]}^{(m)}u_{[k, m]}^{(m)} = u_{[k, m]}^{(m)} u_{[i-1, j-1]}^{(m)}$.
    If  $i\leq k\leq j$, then
     \begin{align*} u_{[i, j]}^{(m)}u_{[k, m]}^{(m)}&= (s_i \cdots s_k)(s_{k+1}\cdots s_j)\cdot (s_{k}\cdots s_m)\\
                     &=(s_i\cdots s_k)\cdot (s_k \cdots s_m)(s_k\cdots s_{j-1})\\
                     &=(s_i\cdots s_{k-1}) (s_{k+1}\cdots s_m)(s_k\cdots s_{j-1})\\
                     &=(s_{k+1}\cdots s_m)(s_i \cdots s_{j-1})= u_{[k+1, m]}^{(m)} u_{[i, j-1]}^{(m)}.
    \end{align*}

    \vspace{-0.5cm}

    $\mbox{}$
\end{proof}

Let us recall the following well-known fact, which holds in general.

\begin{lemma}\label{charalongest}
 Let $\bar P\subset \tilde P$ be parabolic subgroups of $G$.
        If $w\in W_{\tilde P}^{\bar P}$, then       $\ell(w)\leq \ell(\omega_{\tilde P}\omega_{\bar P})$. Furthermore, the equality holds if and only if
          $w=\omega_{\tilde P}\omega_{\bar P}$.
\end{lemma}
\begin{proof}
  Note that     $\omega_{\bar P}$ sends positive roots $R^+_{\bar P}$ to negative roots $-R^+_{\bar P}\subset -R^+_{\tilde P}$ and $\omega_{\tilde P}$ sends  $-R_{\tilde P}^+$
         to  $R_{\tilde P}^+$.
    Hence,   $\omega_{\tilde P}\omega_{\bar P}(R^+_{\bar P})\subset R^+$, implying $\omega_{\tilde P}\omega_{\bar P}\in W_{\tilde P}^{\bar P} $.
    Hence the statement follows, by noting $\omega_{\tilde P}$ is the unique longest element in $W_{\tilde P}$ and
       $\ell(\omega_{\tilde P}v) =\ell(\omega_{\tilde P})-\ell(v)$ for any $v\in W_{\bar P}$.
\end{proof}

\begin{lemma}\label{longestcomp000}
   Let $\Delta_{\tilde P}=\Delta_P\setminus\{\alpha_k\}$  where $1\leq k\leq r$. We have  {\upshape a)}  $gr(\omega_P\omega_{\tilde P})=\sum\nolimits_{p=k}^{r}k\mathbf{e}_p$ and
   {\upshape b)} for any  $v\in  W_P^{\tilde P}$,
   $gr(v)=\sum\nolimits_{p=k}^{r}j_p\mathbf{e}_p  \mbox{ with } j_r\leq  \cdots\leq j_k\leq k.$   

\end{lemma}

  \begin{proof}
    Write  $gr(v)=\sum\nolimits_{p=1}^{r}j_p\mathbf{e}_p$ and set $j_0=0$.
       For each $\alpha_p\in\Delta_{\tilde P}$, we have $\ell(vs_p)=\ell(v)+1$  by Lemma \ref{strongexchange}.
     This  implies
           $j_{p}\leq j_{p-1}$  by Proposition  \ref{gracompare000}. That is, $j_r\leq \cdots\leq j_{k+1}\leq  j_k\leq k$ and $0\leq j_{k-1}\leq\cdots \leq j_1\leq j_0=0$.
       Thus b) follows.

      Let $w=u_k^{(r)}\cdots u_k^{(k)}$.  Note that   $w\in W_P^{\tilde P}$  and
             $\ell(w)=k(r-k+1) =|R^+_{P}|-|R^+_{\tilde P}|=\ell(\omega_P\omega_{\tilde P})$.
    By Lemma \ref{charalongest}, we have $w=\omega_P\omega_{\tilde P}$. That is, a) follows.
   \end{proof}

 In addition, we introduce the next three useful lemmas.
\begin{lemma}[see e.g. \cite{mare}] \label{lengthofposiroot000}
     For any  $\gamma\in R^+$, we have
             $\ell(s_\gamma)\leq \langle 2\rho, \gamma^\vee\rangle-1$.
 \end{lemma}

\begin{lemma} \label{lengthofposiroot}
    Let $\gamma\in R^+\setminus \Delta$ satisfy $\ell(s_\gamma)= \langle 2\rho, \gamma^\vee\rangle-1$.  
                 For any $1\leq j\leq n$ with $\langle \alpha_j, \gamma^\vee\rangle>0$, we have $\langle \alpha_j, \gamma^\vee\rangle=1$.
                 Furthermore for  $\beta:=s_j(\gamma)$,  we have $
                          \beta^\vee=\gamma^\vee-\alpha_j^\vee$ and $
             \ell(s_\beta)=\ell(s_\gamma)-2= \langle 2\rho, \beta^\vee\rangle-1.$
     \end{lemma}
\begin{lemma} \label{lengthofpos222}
    Let $\gamma\in R^+\setminus \Delta$. If
                  $\ell(us_\gamma)=\ell(u)+1-\langle 2\rho, \gamma^\vee\rangle$ where   $u\in W$,   then
                   $\ell(s_\gamma)= \langle 2\rho, \gamma^\vee\rangle-1$. Furthermore, we take   any $1\leq j\leq n$ with $\langle \alpha_j, \gamma^\vee\rangle>0$,
               and set   $\beta:=s_j(\gamma)$. Then   all the followings hold:
      $$\ell(us_j)=\ell(u)-1,\quad \ell(us_js_\beta)=\ell(us_j)-\ell(s_\beta),\quad \ell(us_\gamma)=\ell(us_js_\beta s_j)=\ell(us_js_\beta)-1.$$
\end{lemma}

\begin{proof}
    We prove all these three statements together, including the proof of Lemma \ref{lengthofposiroot000} from \cite{mare}  by induction on $\ell(s_\gamma)$.

     If   $\ell(s_\gamma)=1$, then  $\gamma\in \Delta$ and consequently $\ell(s_\gamma)=1=2\langle\rho, \gamma^\vee\rangle-1$. Now
    we assume  $\gamma\in R^+\setminus \Delta$. Take any $1\leq j\leq n$ such that   $\langle\gamma, \alpha_j^\vee\rangle>0$ (such $j$ does exist; otherwise,
         we would conclude $2=\langle \gamma, \gamma^\vee\rangle\leq 0$).
     Consequently,  $\langle \alpha_j, \gamma^\vee\rangle>0$. Thus
        $ s_\gamma(\alpha_j)=\alpha_j-\langle\alpha_j, \gamma^\vee\rangle\gamma\in-R^+$.
          Also
             $s_j s_\gamma(\alpha_j)=(\langle\gamma, \alpha_j^\vee\rangle \langle\alpha_j, \gamma^\vee\rangle-1)\alpha_j-\langle\alpha_j,\gamma^\vee\rangle\gamma$
           is a negative root. By Lemma \ref{strongexchange}, we have $\ell(s_j s_\gamma s_j)=\ell(s_\gamma)-2$.
        Because  $s_j(\gamma)^\vee=s_j(\gamma^\vee)=\gamma^\vee-\langle\alpha_j, \gamma^\vee\rangle\alpha_j^\vee,$
        we have $\langle \rho, s_j(\gamma)^\vee\rangle= $ $\langle\rho, \gamma^\vee\rangle-\langle\alpha_j, \gamma^\vee\rangle$.
         By the induction hypothesis, we conclude   the following:
           \begin{eqnarray} \ell(s_\gamma)=\ell(s_j s_\gamma s_j)+2&\leq& 2\langle \rho, s_j(\gamma)^\vee \rangle-1+2\\
                                                                   &=&2\langle \rho, \gamma^\vee\rangle-1+2(1-\langle\alpha_j, \gamma^\vee\rangle)\\
                                                                   &  \leq &\langle2\rho, \gamma^\vee\rangle-1.
           \end{eqnarray}

 If  $\ell(s_\gamma)=\langle 2\rho, \gamma^\vee\rangle-1$, then
              both (2.1) and (2.3) must be   equalities. In particular, we conclude $\langle \alpha_j, \gamma^\vee\rangle=1$,
                      $\beta^\vee=\gamma^\vee-\alpha_j^\vee$ and
     $\ell(s_\beta)=\ell(s_\gamma)-2= \langle 2\rho, \beta^\vee\rangle-1$.

      It remains to show   Lemma \ref{lengthofpos222}. Indeed, we have
      $$\ell(u)-\ell(s_\gamma)\leq \ell(us_\gamma)=\ell(u)-(\langle 2\rho, \gamma^\vee\rangle-1)\leq \ell(u)-\ell(s_\gamma).$$
  Hence, both inequalities become equalities. Thus $\ell(s_\gamma)=\langle 2\rho, \gamma^\vee\rangle-1$.

Furthermore, we note  $\ell(us_js_\beta s_j)=\ell(us_\gamma)=\ell(u)-\ell(s_\gamma)$ and
             $$\ell(us_js_\beta s_j)\geq \ell(us_js_\beta)-1\geq \ell(us_j)-\ell(s_\beta)-1\geq \ell(u)-1-\ell(s_\beta)-1 =\ell(u)- \ell(s_\gamma).$$
    Hence, the statements in   Lemma \ref{lengthofpos222} also follow.
              \end{proof}

\subsection{Explicit gradings of $q_j$'s}\label{subsecgrading}
The main results of this subsection are Proposition \ref{graquanvar} and Proposition \ref{gradingforconn},
 giving   explicit formulas for  gradings   $gr(q_j)$'s.

\begin{prop}\label{graquanvar} Let $2\leq j\leq r$.
   Following the notations in Definition \ref{defgrading}, we have
    $\psi_{\Delta_j, \Delta_{j-1}}(1, \alpha_j^\vee+Q^\vee_{j-1})=(u_{j-1}^{(j-1)}, \alpha_j^\vee)$ and
    $gr(q_j)=(1-j)\mathbf{e}_{j-1}+(1+j)\mathbf{e}_j$.
\end{prop}

\begin{proof}
  Note that $\Delta_{j-1}\subset \Delta_j$ with $\Delta_j\setminus \Delta_{j-1}=\{\alpha_j\}$.
   Clearly,     $\langle \alpha, \alpha_j^\vee\rangle  \in \{0, -1\}$ for all $\alpha\in  R^+\cap \bigoplus_{i=1}^{j-1}\mathbb{Z}\alpha_i$.
    Hence, we have $\Delta_{P_{j-1}'}=\{\alpha  \in \Delta_{j-1}~|~ \langle  \alpha, \alpha_j^\vee\rangle =0\}$ $=\Delta_{j-1}\setminus\{\alpha_{j-1}\}$.
       Therefore we conclude $gr(\omega_{P_{j-1}}\omega_{P_{j-1}'})= (j-1)\mathbf{e}_{j-1}$ by using Lemma \ref{longestcomp000} (with respect to $\Delta_{j-1}$).
       Thus the former equality holds. Consequently, the latter equality follows   by Definition \ref{defgrading}.
\end{proof}

The next lemma    works in general, namely we do not need to assume  $\Delta_P$ to be of $A$-type.
\begin{lemma}\label{coroofsumgrad} Let $u\in W$ and   $\lambda \in Q^\vee$.
    \begin{enumerate}
      \item Write $gr(q_\lambda u)=(j_1, \cdots, j_{r+1})$.  Then $\sum_{k=1}^{r+1}j_k=\ell(u)+\langle 2\rho, \lambda\rangle$.
      \item  Let   $\gamma\in R^+$  satisfy $\ell(us_\gamma)=\ell(u)+1-\langle 2\rho, \gamma^\vee\rangle$.
    For any $1\leq p\leq n$,   $gr(q_{\gamma^\vee}us_\gamma)\leq gr(u)+gr(s_p)$ if and only if
                 $gr_r(q_{\gamma^\vee}us_\gamma)\leq gr_r(u)+gr_r(s_p)$.
    \end{enumerate}
\end{lemma}

 \begin{proof}
   Denote $|(a_1, \cdots, a_{r+1})|=\sum_{k=1}^{r+1}a_k$.
   Note that  $\ell(u)=|gr(u)|$. Furthermore, we conclude $|gr(q_{\alpha^\vee})|=2$ for all $\alpha\in \Delta$  by induction. Thus   (1)
         follows.

   Write  $gr(u)+gr(s_p)=(i_1, \cdots, i_{r+1})$ and $gr(q_{\gamma^\vee}us_\gamma)=(\tilde i_1, \cdots, \tilde i_{r+1})$.
      Assume  $(\tilde i_1, \cdots, \tilde i_{r+1})\leq ( i_1, \cdots,  i_{r+1})$, then we have $(\tilde i_1, \cdots, \tilde i_{r})\leq ( i_1, \cdots,  i_{r})$ by definition.
    Assume    $(\tilde i_1, \cdots, \tilde i_{r})\leq (i_1, \cdots, i_{r})$. If ``$<$" holds,
      then  it is already done 
      by the definition
   of the lexicographical order. If ``$=$" holds, then we conclude $\tilde i_{r+1}=i_{r+1}$, by
     noting $\sum_{k=1}^{r+1}\tilde i_k=|gr(q_{\gamma^\vee}us_\gamma)|=\ell(us_\gamma)+\langle 2\rho, \gamma^\vee\rangle =\ell(u)+1=\sum_{k=1}^{r+1}i_k$.
     Thus (2)  follows.
\end{proof}

 \begin{prop}\label{gradingforconn}
      For any  $\alpha \in \Delta\setminus\Delta_P$,
   one and only one of the cases in Table \ref{tabgrading} occurs,  
       where    we     require  $r\geq 2$ (resp. $3$ and $5$)
    for case {\upshape b)} (resp. {\upshape e)}  and  {\upshape f)}).
  \begin{table}[h]
  \caption{\label{tabgrading} Explicit grading  $gr(q_{\alpha^\vee})$ for $\alpha\in \Delta\setminus \Delta_P$ } 
      \begin{tabular}{|c|c|c|c|}
         \hline
         &  $Dyn(\Delta_P\cup\{\alpha\})$  &$\psi_{\Delta, \Delta_P}(1, \alpha^\vee+Q^\vee_P)$&$gr(q_{\alpha^\vee})$ \\\hline\hline
        {\upshape a)} &   \begin{tabular}{l} \raisebox{-0.4ex}[0pt]{$  \circline\!\;\!\!\circ\cdots\, \circline\circline \bullet$}\\
                 \raisebox{1.1ex}[0pt]{${\hspace{-0.2cm}\scriptstyle{\alpha_1}\hspace{0.3cm}\alpha_2\hspace{0.9cm}\alpha_r\hspace{0.17cm}\alpha} $}
  \end{tabular}& $q_{\alpha^\vee}u_{r}^{(r)}$   & $ (r+2)\mathbf{e}_{r+1}-r\mathbf{e}_r$  \\\hline
       {\upshape b)} & \begin{tabular}{l} \raisebox{-0.4ex}[0pt]{$  \bulletline\circline\!\;\!\!\circ\cdots\, \circline\!\;\!\!\circ$}\\
                 \raisebox{1.1ex}[0pt]{${\hspace{-0.05cm}\scriptstyle{\alpha}\hspace{0.38cm}\alpha_1\hspace{0.2 cm}\alpha_2\hspace{0.85cm}\alpha_r} $}
  \end{tabular}  &$q_{\alpha^\vee}u_{1}^{(r)}\cdots u_{1}^{(1)}$   & $ (r+2)\mathbf{e}_{r+1}-\sum_{j=1}^r\mathbf{e}_j$ \\\hline
        {\upshape c)} &  \begin{tabular}{l} \raisebox{-0.4ex}[0pt]{$  \circline\!\;\!\!\circ\cdots\, \circline\bulletbb$}\\
                 \raisebox{1.1ex}[0pt]{${\hspace{-0.2cm}\scriptstyle{\alpha_1}\hspace{0.3cm}\alpha_2\hspace{0.9cm}\alpha_r\hspace{0.3cm}\alpha} $}
  \end{tabular}& $q_{\alpha^\vee}q_{r}u_{r-1}^{(r)}u_{r-1}^{(r-1)}$   & $ (2r+2)\mathbf{e}_{r+1}-2r\mathbf{e}_r$  \\\hline
        {\upshape d)}& \begin{tabular}{l} \raisebox{-0.4ex}[0pt]{$  \circline\!\;\!\!\circ\cdots\, \circline\bulletcc$}\\
                 \raisebox{1.1ex}[0pt]{${\hspace{-0.2cm}\scriptstyle{\alpha_1}\hspace{0.3cm}\alpha_2\hspace{0.9cm}\alpha_r\hspace{0.3cm}\alpha} $}
  \end{tabular}& $q_{\alpha^\vee}u_{r}^{(r)}$   & $ (r+2)\mathbf{e}_{r+1}-r\mathbf{e}_r$  \\\hline
       {\upshape e)}& \begin{tabular}{l}  \vspace{-0.20cm} $\hspace{1.67cm}\bullet \scriptstyle{\alpha}$ \\
                     \vspace{-0.20cm}{ $\hspace{1.645cm}\!\!\;\!\big|$} \\
                   \raisebox{-0.4ex}[0pt]{$  \circline\!\;\!\!\circ\cdots\, \circline\circline\!\!\;\circ$}\\
                 \raisebox{1.1ex}[0pt]{${\hspace{-0.2cm}\scriptstyle{\alpha_1}\hspace{0.3cm}\alpha_2\hspace{0.7cm} \hspace{0.7cm}\alpha_{r}} $}
  \end{tabular}& $q_{\alpha^\vee}u_{r-1}^{(r)}u_{r-1}^{(r-1)}$   & $ 2r\mathbf{e}_{r+1}+(1-r)(\mathbf{e}_r+\mathbf{e}_{r-1})$  \\\hline
       {\upshape f)}& \begin{tabular}{l}  \vspace{-0.20cm} $\hspace{1.21cm}\bullet \scriptstyle{\alpha}$ \\
                     \vspace{-0.20cm}{ $\hspace{1.205cm}\!\!\;\!\!\;\!\big|$} \\
                   \raisebox{-0.4ex}[0pt]{$  \circ\cdots\, \circline\circline\circline\!\!\;\!\circ$}\\
                 \raisebox{1.1ex}[0pt]{${\hspace{-0.15cm}\scriptstyle{\alpha_1}  \hspace{1.3cm}\alpha_{r-1} \hspace{0.1cm}\alpha_{r}} $}
  \end{tabular}          & $q_{\alpha^\vee}u_{r-2}^{(r)}u_{r-2}^{(r-1)}u_{r-2}^{(r-2)}$   &    {\upshape $(3r-4)\mathbf{e}_{r+1}+(2-r)\sum\limits_{j=r-2}^r\mathbf{e}_{j}$} 
                  \\\hline
       {\upshape g)} & \begin{tabular}{l}  $\bulletcc$ \vspace{-0.425cm}$\hspace{-0.75cm}\;\!-\!\!\;\!-$ \\
                 \raisebox{-1.1ex}[0pt]{${\hspace{-0.1cm}\scriptstyle{\alpha_1}  \hspace{0.25cm}\alpha  }$}
  \end{tabular}  & $q_{\alpha^\vee}s_1$   & $(-1, 3)$  \\\hline
      {\upshape h)} & \begin{tabular}{l}  $\bulletbb$ \vspace{-0.425cm}$\hspace{-0.75cm}\;\!-\!\!\;\!-$ \\
                 \raisebox{-1.1ex}[0pt]{${\hspace{-0.1cm}\scriptstyle{\alpha_1}  \hspace{0.25cm}\alpha  }$}
  \end{tabular}  & $q_{\alpha^\vee}q_1s_1$   & $(-3, 5)$  \\\hline
       {\upshape i)}& \begin{tabular}{l}  
                   \raisebox{-0.4ex}[0pt]{$  \circline\!\!\;\!\!\,\circ\cdots\, \!\!\;\circline\!\!\;\!\circ\;\;\;\bullet$}\\
                 \raisebox{1.1ex}[0pt]{${\hspace{-0.15cm}\scriptstyle{\alpha_1}  \hspace{1.4cm}\alpha_{r}\hspace{0.3cm}\alpha} $}
  \end{tabular}          & $q_{\alpha^\vee}$   &    $2\mathbf{e}_{r+1}$
                  \\\hline
 \end{tabular}
   \end{table}
\end{prop}

\begin{proof}
  Clearly,   $Dyn(\Delta_P\cup\{\alpha\})$
   is  given by a unique   case  in Table \ref{tabgrading}.

 Let $\lambda_P=\alpha^\vee+Q^\vee_P$ and     $\psi_{\Delta, \Delta_P}(1, \lambda_P)=q_{\lambda_B}\omega_P\omega_{P'}$. Here
      $\lambda_B\in Q^\vee$ is the (unique) element satisfying $\langle \beta, \lambda_B\rangle\in \{0, -1\}$ for all $\beta\in R^+_P$.
      Since $\Delta_P$ is of $A$-type, this is equivalent to requiring $\langle \alpha_j, \lambda_B\rangle=0$ for all $\alpha_j\in \Delta_P$ but
       at most one and if such unique $\alpha_j$ exists then $\langle \alpha_j, \lambda_B\rangle=-1$.
        For each case in Table \ref{tabgrading},
     it is easy to see
       that the element $\lambda_B$ as provided does satisfy this property.
       Consequently,   $\Delta_{P'}=\{\alpha_i\in \Delta_P~|~   \langle \alpha_i, \lambda_B\rangle=0\}=\Delta_P\setminus\{\alpha_k\}$
        for a certain $1\leq k\leq r+1$.
       Hence, we can directly write down $\omega_P\omega_{P'}$ by using  Lemma \ref{longestcomp000}.
       Finally, we obtain  $gr(q_{\alpha^\vee})$ as is listed in Table \ref{tabgrading},  by  direct calculations (with
                Definition \ref{defgrading} and Proposition \ref{graquanvar}).
   \end{proof}

The next corollary follows directly from Table
\ref{tabgrading}.
\begin{cor}\label{gradforonenode}
  If $r=1$, then   $gr(q_j)=(a, -a+2)$ with $a=\langle \alpha_1, \alpha_j^\vee\rangle$ for each $j$.
  Consequently for any  $\lambda\in Q^\vee$, we have   $gr(q_{\lambda})=(\langle \alpha_1, \lambda\rangle, \langle 2\rho-\alpha_1, \lambda\rangle)$.
\end{cor}

As we will see later, we use induction on $\ell(s_\gamma)$ to prove the Key Lemma. The next proposition shows the special case of the Key Lemma when $\ell(s_\gamma)=1$.

\begin{prop}\label{simrefinparabolic}
 Let  $u\in W$ and $1\leq j\leq n$. If $\ell(us_j)=\ell(u)-1$, then $gr(q_j {us_j})\leq gr(u)+gr({s_j})$.
\end{prop}

\begin{proof}
           Let  $gr(u)=(i_1,  \cdots, i_{r+1})$ and
            $gr(us_j)=(\tilde i_1, \cdots, \tilde i_{r+1})$. When $\alpha_j\in \Delta_P$, we have $1\leq j\leq r$.
     If $j=1$, then we have $i_1=1$ and $gr(us_1)=(0, i_2, \cdots, i_{r+1})$.
      Hence,  $gr(q_1us_1)=gr(q_1)+gr(us_1)=(2, 0, \cdots, 0)+(0, i_2, \cdots, i_{r+1})=gr(u)+gr(s_1).$
     If $2\leq j\leq r$, then
         by  Proposition \ref{gracompare000} and Proposition \ref{graquanvar}, we conclude
              $$gr(q_jus_j)-gr(u)-gr(s_j)=((1-j)+i_{j}-1-i_{j-1})\mathbf{e}_{j-1}+((1+j)+i_{j-1}-i_j-1)\mathbf{e}_{j}. $$
          Thus we have  $gr(q_jus_j)\leq gr(u)+gr(s_j)$, by noting $0\leq i_{j-1}<i_j\leq j$.

 When $\alpha_j\in \Delta\setminus \Delta_P$, we note that $gr(s_{j})=\mathbf{e}_{r+1}$.
  By Lemma \ref{coroofsumgrad}, it suffices to show $gr_r(q_jus_j)\leq gr_r(u)$. Write   $gr(q_jus_j)=(\hat i_1, \cdots, \hat i_{r+1})$.
       and  $\psi_{\Delta, \Delta_P}(1, \alpha_j^\vee+Q^\vee_P)=\lambda_B\omega_P\omega'$.
        We first assume $\lambda_B=\alpha_j^\vee$. Then $\Delta_{P'}=\{\alpha\in \Delta_P~|~
             \langle \alpha, \alpha_j^\vee\rangle =0\}$. If $\Delta_{P'}=\Delta_P$ (i.e. case i) of Table \ref{tabgrading} occurs),
            then we have $gr_r(q_j)=\mathbf{0}$ and $gr_r(us_j)=gr_r(u)$ (by Lemma \ref{gracompare11111}). Thus it is done in this case.
            Otherwise, we conclude  $\Delta_{P'}= \Delta_{P}\setminus\{\alpha_a\}$ for a unique $1\leq a\leq r$ (from Table \ref{tabgrading}).
  Consequently, we have $gr_r(q_j)=-gr_r(\omega_P\omega')$ by definition,
              $gr_r(us_j)\leq gr_r(u)+\sum_{k=a}^{r}a\mathbf{e}_k$ by Proposition \ref{maingracompare} and  $gr(\omega_P\omega')=
            \sum_{k=a}^{r}a\mathbf{e}_k$ by Lemma \ref{longestcomp000}. Hence, we do have $gr_r(q_jus_j)\leq gr_r(u)$  in this case.
     Now we assume   $\lambda_B\neq \alpha_j^\vee$. Due to Table \ref{tabgrading}, it remains to consider case c) and h).
    If case h)  occurs, then $n=2, r=1, gr(q_j)=(-3, 5)$ and we do have
     $\hat i_{1}=-3+\tilde i_1\leq -2<i_1$. If case c)  occurs,
  then $gr_r(q_{j})=-2r\mathbf{e}_r$ and
           we have   $gr_{r-1}(us_{j})=gr_{r-1}(u)$  by Lemma \ref{gracompare11111}.
      Hence, $gr_r(q_jus_j)-gr_r(u)=(-2r+\tilde i_r-i_r)\mathbf{e}_r\leq (-2r+r-0)\mathbf{e}_r<\mathbf{0}$.
      Hence, the statement follows.
  \end{proof}

\subsection{Proof of the Key Lemma when $\Delta_P$ is of $A$-type}\label{subseckeylemma}
 Recall that we have assumed $\Delta_P$ to be of $A$-type in this subsection.

\begin{prop}\label{mainsecondprop} Part {\upshape a)} of the Key Lemma holds.
\end{prop}

\begin{proof}
           Write $us_\gamma=v_{r+1}v_r\cdots v_1$, where $v_{r+1}\in W^P$ and $v_k=u_{i_k}^{(k)}$ for $1\leq k\leq r$. Thus
          $gr(us_\gamma)=(i_1, \cdots, i_r, \ell(v_{r+1}))$.
        Fix a reduced expression of $v_{r+1}$.
         Since $\ell(u)=\ell(us_\gamma s_\gamma)<\ell(us_\gamma)$, by Lemma  \ref{strongexchange} we have
       $u=v_{r+1}\cdots v_{m+1}\bar v_m v_{m-1}\cdots v_1$ for some $1\leq m\leq r+1$, in which
          $\bar v_m$ is the element
         obtained by  with deleting a (unique) simple reflection from $v_m$.
          Since $\ell(u)=\ell(us_\gamma)-1$, the induced expression of $u$ is also reduced.  Hence,  $\ell(\bar v_m)=\ell(v_m)-1$ and
              if we   write $\bar v_m=v'w$ with $v'\in W^{P_{m-1}}_{P_m}$ and  $w\in W_{P_{m-1}}$,  
            then $\ell(\bar v_m)=\ell(v')+\ell(w)$ and $\ell(wv_{m-1}\cdots v_1)=\ell(w)+\ell(v_{m-1}\cdots v_1)$.
         By Proposition  \ref{gracompare000}, there exist non-negative integers $p_k$'s such that
                       $gr(u)=(i_1+p_1, \cdots, i_{m-1}+p_{m-1}, \ell(v'), i_{m+1}, \cdots, i_r, \ell(v_{r+1}))$
              with  $\sum_{k=1}^{m-1}p_k=\ell(w)$.
         On the other hand, by Lemma \ref{weylonelemma} we conclude $\gamma\in R_{P_m}^+\setminus R_{P_{m-1}}$, so that
            $\min\{gr(s_i)~|~ \langle \chi_i, \gamma^\vee\rangle\neq0\}=\mathbf{e}_m$.
         Hence, we have
                $gr(us_\gamma)= (i_1, \cdots, i_r, \ell(v_{r+1}))
                             \leq (i_1+p_1, \cdots, i_{m-1}+p_{m-1}, \ell(v')+1, i_{m+1}, \cdots, i_r, \ell(v_{r+1}))
                       =gr(u)+\mathbf{e}_m$, by noting $\ell(v_{r+1})+\sum_{k=1}^ri_k=\ell(us_\gamma)=\ell(u)+1= \ell(v')+\sum_{k=1}^ri_k+ \sum_{k=1}^{m-1}p_k +1$.
  \end{proof}

The remaining part of this subsection is devoted to a proof of the
following.

\begin{prop}\label{mainpropprod} Part {\upshape b)} of the Key Lemma holds. That is,
   for any  $u\in W$ and $\gamma\in R^+$,   if  {\upshape\textbf{(L1)}: } $\ell(us_\gamma)=\ell(u)+1-\langle 2\rho, \gamma^\vee\rangle$,    then we have
 {\upshape\textbf{(L2)}: }       $gr(q_{\gamma^\vee}us_\gamma) \leq  gr(u)+\min\{gr(s_i)~|~ \langle \chi_i, \gamma^\vee\rangle\neq0\}$.
\end{prop}

\begin{lemma}\label{inductionforonenode}
       Part {\upshape b)} of the Key Lemma holds when  $\Delta_P=\{\alpha_1\}$.
 \end{lemma}

 \begin{proof} We use
                   induction on $\ell(s_\gamma)$.
     If $\ell(s_\gamma)=1$, then $\gamma\in \Delta$ and consequently  (L2) follows from  Proposition \ref{simrefinparabolic}.
           Now we assume $\gamma\in R^+\setminus\Delta$.
             Write $gr(u)=(i_1, i_2)$,  $gr(us_{\gamma})=(j_1, j_2)$ and $gr(q_{\gamma^\vee})=(k_1, k_2)$, in which
           $k_1=\langle \alpha_1, \gamma^\vee\rangle$  by Corollary \ref{gradforonenode}.
     If $k_1<0$, then $j_1+k_1\leq 1+k_1\leq 0$.
          If $k_1=0$, then we have $i_1=j_1$  by
        Lemma  \ref{gracompare11111}. In either of the cases, we conclude $j_1+k_1\leq i_1$. Thus
        (L2) holds by Lemma \ref{coroofsumgrad}.
    Otherwise,
      $\langle \alpha_1, \gamma^\vee\rangle=k_1>0$. Then by Lemma \ref{lengthofposiroot} and Lemma \ref{lengthofpos222},
      we conclude that for  $\beta:=s_1(\gamma)$ the followings hold:
        $\beta^\vee=\gamma^\vee-\alpha_1^\vee$;
         $gr(q_1)+gr(us_1)\leq  gr(u)+ gr(s_1)$ (by Proposition \ref{simrefinparabolic});
        $gr(q_{\beta^\vee})+gr(us_1s_\beta) \leq  gr(us_1)+\mathbf{e}_c$ (by the induction
                hypothesis), where we denote
               $\mathbf{e}_c:= \min\{gr(s_i)~|~ \langle \chi_i, \beta^\vee\rangle\neq0\};$
        $gr(us_1s_\beta s_1)=gr(us_1s_\beta)-gr(s_1)$ (by Proposition  \ref{gracompare000}).
           Hence, we conclude (L2) holds, by noting $\mathbf{e}_c=\min\{\mathbf{e}_c, gr(s_1)\}=
            \min\{gr(s_i)~|~ \langle \chi_i, \gamma^\vee\rangle\neq0\}$.
 \end{proof}

When $\Delta$ is also of $A$-type, it is easy to obtain $\lambda_B$ and $gr(q_{\lambda_B})$ associated to a given $\lambda_P\in Q^\vee/Q^\vee_P$.
  For instance, by direct calculations we conclude the following lemma. (Recall that $\alpha_j=\beta_{o+j}$ for $1\leq j\leq r$ in Table \ref{tabrelativeposi}.)
\begin{lemma}\label{polycompforAAtype000000}
  Let   $\Delta$ be of $A$-type and $m\leq r+1$. Following the notations in case   {\upshape $\mbox{C}1)$}, we set $\lambda=\sum_{k=1}^{m}k\beta_{o+r+1-m+k}^\vee$.
      Then the followings hold  (where $0\cdot \mathbf{e}_0:=\mathbf{0}$).

     \begin{enumerate}
       \item If  $m=r+1$, then  $\langle\alpha, \lambda\rangle=0$ for all $\alpha\in \Delta_P$; if $m<r+1$, then for any $\alpha\in \Delta_P$,
                          $\langle\alpha, \lambda\rangle$ is equal to $-1$ if $\alpha=\alpha_{r+1-m}$, or equal to $0$ otherwise.
                  In particular,   $\lambda$ is the element associated to $m\beta_{o+r+1}^\vee+Q^\vee_P$ via PW-lifing.
        \item   $gr(q_{\lambda})\!=\!m(r+2)\mathbf{e}_{r+1}-(r+1-m) \sum_{k=r+1-m}^{r} \mathbf{e}_{k}$.  In particular if $m=r+1$, then $gr_r(q_\lambda)= \mathbf{0}$.
     \end{enumerate}
   Furthermore, we have $gr_r(\sum_{k=o}^{o+r+1}\beta_k^\vee)=\mathbf{0}$, whenever $o\geq 1$.
\end{lemma}

Since the case of $A$-type is relatively easy to handle, we would like to compare all relevant information  for
       $\Delta$ being  of general type with those when $\Delta$ is of $A$-type.
Due to Lemma \ref{coroofsumgrad}, we only need to care about $gr_r(q_\lambda w)$. 
 For these purposes, we bring in a base $\dot\Delta$ of $A$-type and introduce the notion of ``\textit{virtual coroot}" as below  for $r\geq 2$.

  Let $\dot\Delta=\{\dot \beta_1, \cdots, \dot\beta_n\}$ be a base with
    $Dyn(\dot\Delta)$ given by  \begin{tabular}{l} \raisebox{-0.4ex}[0pt]{$  \circline\!\;\!\!\circ\cdots\, \circline  \!\!\!\;\circ$}\\
                 \raisebox{0.9ex}[0pt]{${\hspace{-0.1cm}\scriptstyle{\dot\beta_1}\hspace{0.25cm}\dot\beta_2\hspace{0.9cm}\dot\beta_{n}} $}
             \end{tabular}.
  Denote  $\dot\alpha_i=\dot \beta_{o+i}$ for each $1\leq i\leq r$ (the notation ``$o$" is the same one as in Table \ref{tabrelativeposi}).
  Set $\dot\Delta_P=\{\dot \alpha_1, \cdots, \dot \alpha_r\}$. Following Definition \ref{defgrading},   we can   obtain a grading map
  with respect to $(\dot \Delta_P, \mbox{Id}_{\dot\Delta_P})$, which we also denote as $gr$ by abuse of notations.
  Clearly, $\dot\beta_j\mapsto \beta_j$ extends to an isometry  $\dot\Delta \setminus \{\dot\beta_\eta\}
           \rightarrow\Delta\setminus\{\beta_\eta\}$ of bases, where $\eta$ is  given in Table \ref{tabvircoroot}. 
    Denote $\dot Q^\vee=\bigoplus_{i=1}^n\mathbb{Z}\dot\beta_i^\vee$.
\begin{defn}

 Let $\lambda\in Q^\vee$.     We call $\dot \lambda\in \dot Q^\vee$  a \textbf{virtual coroot} of $\lambda$  (at level $\eta$),
                 if $\dot \lambda $ satisfies both $gr_r(q_{\dot\lambda})=gr_r(q_{\lambda})$ and $\langle \dot\alpha_i, \dot\lambda\rangle =\langle \alpha_i,  \lambda\rangle \mbox{ for }  1\leq i\leq r.$
 \end{defn}

\begin{lemma}\label{virrootforalltype}
 For each case in Table \ref{tabrelativeposi} (where we have assumed $r\geq 2$), there is a virtual coroot
 $\dot \lambda$    of $\lambda=\sum_{j=1}^nc_j\beta_j^\vee$ (at level $\eta$),  given by  Table \ref{tabvircoroot}.

  \begin{table}[h]
 \caption{ \label{tabvircoroot} Virtual coroot $\dot\lambda\!\!=c_\eta\dot\mu\!+\sum_{j=1}^{\eta-1}c_j\dot\beta_j^\vee$}

\begin{tabular}{|c||c|c|c|c|c|c|c|c|c|c|}
  \hline
    & {\upshape $\mbox{C}1)$} & {\upshape $\mbox{C}9)$}& {\upshape $\mbox{C}10)$} & {\upshape $\mbox{C}2)$} &{\upshape $\mbox{C}3)$} & {\upshape $\mbox{C}5)$} &
              {\upshape $\mbox{C}7)$} & {\upshape $\mbox{C}4)$ } & {\upshape $\mbox{C}6)$ } & {\upshape $\mbox{C}8)$ }\\
       \hline
    $\eta$ & $n$ & $3$ & $3$ & $n$ & $4$ & $5$ & $7$ & $8$ & $8$ & $7$ \\
       \hline
       $\dot\mu$  & \multicolumn{3}{c|}{$-\langle \beta_{\eta-1}, \beta_\eta^\vee\rangle\dot \beta_\eta^\vee$ }   &
        \multicolumn{4}{c|}{ $ \dot\beta_{\eta-1}^\vee+2\dot \beta_\eta^\vee$ }   & $\sum\limits_{j=1}^3j\dot\beta^\vee_{5+j}$ &  $\sum\limits_{j=1}^5j\dot\beta^\vee_{3+j}$  &  $\sum\limits_{j=1}^5j\dot\beta^\vee_{2+j}$   \\
               \hline
\end{tabular}
\end{table}
 \end{lemma}
\begin{proof}
 Note that
      $\dot\Delta \setminus \{\dot\beta_\eta\}$ is canonically isomorphic to $\Delta\setminus\{\beta_\eta\}$ as bases and that
       $\Delta_P\subset \{\beta_1, \cdots, \beta_{\eta-1}\}$.
   It is easy to see $\dot\beta_j^\vee$ is a virtual coroot of $\beta_j^\vee$ (resp. 0) for each $j\leq \eta-1$ (resp. $j\geq \eta+1$).
        Combining Table \ref{tabrelativeposi} and Table \ref{tabgrading}, we  conclude that
            $gr_r(q_{\dot\mu})=gr_r(q_{\beta_\eta^\vee})$   and
                 $\langle \dot\alpha_i, \dot\mu\rangle =\langle \alpha_i,  \beta_\eta^\vee\rangle $ for   $1\leq i\leq r$. That is, $\dot\mu$ is
                 a virtual coroot of $\beta_\eta^\vee$.   Hence, the statement follows.
\end{proof}

 \begin{remark} Lemma \ref{virrootforalltype} tells us about the existence of a virtual coroot. Due to Lemma \ref{polycompforAAtype000000}, we
    note that the uniqueness does not hold:
   if  $\dot \lambda$ is a   virtual coroot   of $\lambda$,
   so is $\dot\lambda+ \sum_{j=1}^\eta j\dot \beta_j^\vee$.
 \end{remark}

Due to Lemma \ref{inductionforonenode}, it remains to care about the case when $r\geq 2$. The next proposition shows that we can describe most of the coroots
  uniformly with the help of the notion of ``virtual coroots".

\begin{prop}\label{coeffvircoroot} Assume $r\geq 2$.
    Let $\gamma\in R^+\setminus \Delta $ satisfy
          $\ell(s_\gamma)=\langle 2\rho, \gamma^\vee\rangle-1$.
 Then  one and only one of the    followings holds.
  \begin{enumerate}
    \item There exists a virtual  coroot $\dot\gamma^\vee=\sum_{p=1}^{r+1}\dot c_p\dot \alpha_p^\vee$ of $\gamma^\vee$, where  
                      $\dot\alpha_{r+1}:=\dot\beta_{o+r+1}$, $\dot c_{r+1}\leq r \mbox{ and } \dot c_{p}-1\leq \dot c_{p-1}\leq \dot c_{p}$ for each $p\in \{1, \cdots, r+1\}$ (where $\dot c_0:=0$).
    \item  $\gamma^\vee=\sum_{p=d}^m  \beta_p^\vee$ where $o\leq m\leq o+r$ and $d<m$.
    \item Case  {\upshape $\mbox{C} 9)$}  occurs and $\gamma^\vee=\beta_3^\vee+\beta_4^\vee$.
  \end{enumerate}
 \end{prop}

\begin{proof}
  Let    $\gamma^\vee\!\!=\!\!\sum_{j=1}^nc_j\beta_j^\vee$ , which has a  virtual coroot
    $\sum_{j=1}^\eta\tilde c_j\dot \beta_j^\vee$  by Lemma \ref{virrootforalltype}.

    We first assume $c_{\eta}\neq 0$. 
    Set $(c_1', \cdots, c_\eta'):=(c_1, \cdots, c_\eta)-[{c_\eta\over \eta}](1, \cdots, \eta)$. Then we
      obtain another virtual coroot $\sum_{j=1}^\eta c_j'\dot \beta_j^\vee$ of $\gamma^\vee$, by noting  that $\sum_{j=1}^\eta j\dot \beta_j^\vee$ is a virtual coroot of $0$.
   We claim $c_j'-1\leq c_{j-1}'\leq c_j'$ for each $i$ (where $c_0':=0$) and show this by discussing all possible coroots with respect to the type of $\Delta$.

  When $\Delta$ is of $A$-type,  clearly it is done  (by noting
     $(c_1', \cdots, c_\eta')=(\tilde c_1, \cdots, \tilde c_\eta)=(0, \cdots, 0, 1, \cdots, 1))$.
  When $\Delta$ is of $D$-type, either $\mbox{C}2)$ or $\mbox{C}3)$ will occur. For the former case,  we have $\eta=n$ and $c_{n-2}\in \{1, 2\}$.
    If $c_{n-2}=2$, then we have $c_{n-1}=1$ and $(c_1', \cdots, c_n')=(\tilde c_1, \cdots, \tilde c_n)=(0,\cdots, 0, 1, \cdots, 1, 2, \cdots, 2)$.
   If $c_{n-2}=1$, then $\gamma^\vee=\beta_n^\vee+\sum_{p=a}^b\beta_p^\vee$ for some $a\leq n-2\leq b\leq n-1$.  Hence,
      $(c_1', \cdots, c_n')=(\tilde c_1, \cdots, \tilde c_n)=(0,\cdots, 0, 1, \cdots, 1,1+\delta_{b, n-1}, 2)$.
    Thus our claim holds.
   For the latter case, we have $\eta=4$ and can show our claim with similar arguments.
    When $\Delta$ is of $E$-type, there are only finite coroots which are listed in  Plate
              V, VI and VII of \cite{bour}. In this case, our claim still holds by direct calculations.

 When $\Delta$ is of type $B_n$  (resp. $C_n$),
    then our claim follows immediately from Plate III (resp.   II)   of \cite{bour},
         except for the following  coroots.

\begin{center}
\begin{tabular}{|c||c|c|}
  \hline
  & $\mbox{C}1)$ for type $B_n$ & $\mbox{C}1)$ for type $C_n$  \\ \hline
   &&\\
 \raisebox{1.5ex}[0pt]{  $\gamma^\vee$}& \raisebox{1.7ex}[0pt]{ $\beta_n^\vee+2\sum\limits_{i\leqslant p<n}\beta_p^\vee$ \,\,\,($1\leq i<n$)} &
  \raisebox{1.7ex}[0pt]{ $\sum\limits_{i\leqslant p<j}\beta_p^\vee+2\sum\limits_{j\leqslant p\leqslant n}\beta_p^\vee$ \,\,\,($1\leq i<j\leq n$)}
                 \\
  \hline
\end{tabular}
\end{center}
  However,  none of the above coroots satisfies our condition: $\ell(s_\gamma)=\langle 2\rho, \gamma^\vee\rangle -1.$
Indeed if they satisfied this condition, then
   for the former case  we would have $\langle \beta_i, \gamma^\vee\rangle =2>1$, contrary to Lemma \ref{lengthofposiroot}.
 For the latter case, we     denote     $\gamma_k^\vee=\sum_{i\leq p<k}\beta_p^\vee+2\sum_{k\leq p\leq n}\beta_p^\vee$ for $j\leq k\leq  n$.
           Note that  $\gamma_j^\vee=\gamma^\vee$ and $\langle \beta_k, \gamma_k^\vee\rangle >0$ for all $k$.  By Lemma \ref{lengthofposiroot},
            we have    $\gamma_{j+1}^\vee=\gamma_j^\vee-\beta_j^\vee$ and $\ell(s_{\gamma_{j+1}})=\langle 2\rho, \gamma_{j+1}^\vee\rangle-1$.
             Thus by induction          we conclude      $\ell(s_{\gamma_{n}})=\langle 2\rho, \gamma_{n}^\vee\rangle-1$.
     However,  $\langle \beta_n, \gamma_n^\vee\rangle =\langle \beta_n, \beta_{n-1}^\vee+2\beta_n^\vee\rangle=2>1$, contrary to Lemma \ref{lengthofposiroot} again.
 Hence, our claim holds in this case.
 When $\Delta$ is of type $F_4$, which is the remaining case we need to consider since $r \geq 2$,   case  $\mbox{C} 10)$ or $\mbox{C}9)$ must occur.
                When $\mbox{C}10)$  occurs,  our claim follows immediately from  Plate VIII of \cite{bour} and Table \ref{tabvircoroot}.
       When $\mbox{C} 9)$ occurs, we
        denote           $M:=\max\{c_1,  c_2, c_3, c_4\}$.
      If $M>1$, then there are 14 coroots in total (see Plate VIII  of \cite{bour}), only five coroots among which satisfy our condition
       on the length.  Explicitly, $(c_1, c_2, c_3, c_4)=(1, 2, 1, 0), (1, 2, 1, 1),$ $(1, 2, 2, 1), (1, 3, 2, 1)$ or $(2, 3, 2,1)$.
       If $M=1$,  then $\gamma^\vee=\sum_{p=a}^{b}\beta_p^\vee$ for some $1\leq a<b\leq 4$. Clearly, our claim follows, except for
             the coroot $\gamma^\vee= \beta_3^\vee+\beta_4^\vee$.

     Note that $\Delta_P\subset \{\beta_1,\cdots, \beta_{\eta-1}\}$. We conclude $\dot \beta_j^\vee$ is a virtual coroot of 0 whenever  $j<o$ or $j>o+1$.
     In particular, we set $\dot c_i=c_{o+i}'-c_o'$ for each $0\leq i\leq r+1$.
     Then we obtain a virtual coroot
       $\dot\gamma^\vee=\sum_{p=1}^{r+1}\dot c_p\dot \alpha_p^\vee$ of $\gamma^\vee$ satisfying
                       $\dot c_{p}-1\leq \dot c_{p-1}\leq \dot c_{p}$ for each $p$, whenever $c_\eta\neq 0$ except when case (3) of our statements occurs.
    Furthermore, we note that $\dot c_{r+1}\leq r+1$ and if ``$=$" holds then we must have  $\dot\gamma^\vee=\sum_{p=1}^{r+1}p\dot \alpha_p^\vee$, which
     is still a virtual coroot of 0. In this case, we just replace $\dot \gamma^\vee$ with $0=\sum_{p=1}^{r+1}0\cdot \dot\alpha_{p}^\vee$.

   Now we assume $c_{\eta}=0$. Note that $Dyn(\{\beta_1, \cdots, \beta_{\eta-1}\})$ is of $A$-type and that $o+r+1\leq \eta$. Thus
    if $0$ is not a virtual coroot of $\gamma^\vee$, then
        we must have  $\gamma^\vee=\sum_{p=d}^m\beta_p^\vee$ for some $1\leq d< m\leq \eta-1$. Hence, one of the followings must hold:
         (i) $m<o$; (ii) $m\geq o+r+1$ and $d\leq o$; (iii) $m\geq o+r+1$ and $d>o$; (iv)  $o\leq m\leq o+r$.
         If either (i) or (ii) held, then
               0 would be a virtual coroot of $\gamma^\vee$. If (iii) holds, then
       $\sum_{p=d}^{r+1}\dot\beta^\vee$ is a virtual coroot of $\gamma^\vee$, so that   (1) of our statements holds.
      If (iv) holds, then (2) of our statements holds.
\end{proof}

\bigskip

\begin{proof}[Proof of Proposition \ref{mainpropprod}]
 Due to Lemma \ref{inductionforonenode}, we assume $r\geq 2$ and then use induction on  $\ell(s_\gamma)$.

             If $\ell(s_\gamma)=1$, then $\gamma\in \Delta$ and consequently    (L2)  follows from Proposition \ref{simrefinparabolic}.

                Now we  assume  $\gamma\in R^+\setminus \Delta$.  Take any   $1\leq j\leq n$ with  $\langle\alpha_j, \gamma^\vee\rangle>0$.
   Write $\beta=s_j(\gamma)$,
         $gr(q_{\beta^\vee})=(\lambda_1, \cdots, \lambda_{r+1}),
                         \min\{gr(s_i)~|~ \langle \chi_i, \beta^\vee\rangle\neq0\}=\mathbf{e}_c \mbox{ and }$
       \begin{align*} gr(q_j)+gr(us_j)&= gr(u)+ (a_1, \cdots, a_{r+1}),\\
                   gr(q_{\beta^\vee})+gr(us_js_\beta) &=  gr(us_j)+\mathbf{e}_c+(\mu_1,\cdots, \mu_{r+1}),\\
                   gr(us_js_\beta s_j)&= gr(us_js_\beta)+(b_1, \cdots, b_{r+1}).
                  \end{align*}
   Thus we have $gr(q_{\gamma^\vee} us_\gamma)=gr(u)+\mathbf{e}_c+\sum_{p=1}^{r+1}(a_p+b_p+\mu_p)\mathbf{e}_p$,
   taking the summation of the last three equalities.
     Due to  Lemma \ref{lengthofposiroot} and Lemma \ref{lengthofpos222}, we conclude  $\min\{gr(s_i) |  \langle \chi_i, \gamma^\vee\rangle\neq0\}\!=\min\{\mathbf{e}_c, gr(s_j)\}$
           and $(\mu_1, \cdots, \mu_{r+1})\leq \mathbf{0}$ by the induction hypothesis.
       Furthermore, we have     $(a_1, \cdots, a_{r+1})\leq gr(s_j)$ by Proposition \ref{simrefinparabolic}.   We first make several observations  as follows.
    \begin{enumerate}
      \item[(Ob1)]  Assume $\mathbf{e}_c\leq gr(s_j)$ and $\langle \alpha, \alpha_j^\vee\rangle =0$ for all $\alpha\in\Delta_P$.  Then
          $(b_1, \cdots, b_{r+1})=-\mathbf{e}_{r+1}=-gr(s_j)$ (by Lemma \ref{gracompare11111}). Consequently, (L2) follows.

      \item[(Ob2)]  Assume $j=1$.  Then we have $\mathbf{e}_c\leq gr(s_1)$,
        $(b_1, \cdots, b_{r+1})=-gr(s_1)$ by Proposition  \ref{gracompare000}  and consequently   (L2) follows.

       \item[(Ob3)]    Assume $2\leq j\leq r$ and $\mathbf{e}_c<\mathbf{e}_j$. Then $gr(q_j)=(j+1)\mathbf{e}_{j}-(j-1)\mathbf{e}_{j-1}$. 
                   Write $gr(u)=\sum_{p=1}^{r+1}i_p\mathbf{e}_p$ and $gr(us_js_\beta)=\sum_{p=1}^{r+1}k_p\mathbf{e}_p$. 
                  Note that
              $\ell(us_j)=\ell(u)-1$ and $\ell(us_js_\beta s_j)=\ell(us_js_\beta)-1$.
           By  Proposition  \ref{gracompare000}, we  have
                   \begin{align*}
                        (a_1, \cdots, a_{r+1})&=(i_j-i_{j-1}-j)\mathbf{e}_{j-1}+(j+1+i_{j-1}-i_j)\mathbf{e}_j,\\
                        (b_1, \cdots, b_{r+1})&=(k_j-k_{j-1}-1)\mathbf{e}_{j-1}+(k_{j-1}-k_j)\mathbf{e}_j,\\
                    k_j=i_{j-1}+ \mu_{j}-\lambda_j\;\;&\mbox{and} \;\;  k_{j-1}=i_j-1+\mu_{j-1}-\lambda_{j-1}.
                   \end{align*}
      As a consequence, we have
       $(a_1+b_1+\mu_1, \cdots, a_{r+1}+b_{r+1}+\mu_{r+1})=(\mu_1, \cdots, \mu_{j-2}, \mu_j+M, \mu_{j-1}-M, \mu_{j+1},
                           \cdots, \mu_{r+1})$, where $M:=\lambda_{j-1}-\lambda_j-j$.
       Thus if $M=0$ and   $(\mu_1, \cdots, \mu_{j-2}, \mu_j, \mu_{j-1})\leq (0, \cdots, 0)$, 
        then (L2) follows.
    \end{enumerate}

  Now we begin to discuss all possibilities for $\gamma^\vee$, using Proposition \ref{coeffvircoroot}.

  When case  (1) of Proposition \ref{coeffvircoroot} holds, there exists a virtual
     coroot $\dot\gamma^\vee=\sum_{p=1}^{r+1}\dot c_p\dot \alpha_p^\vee$ of $\gamma^\vee$ such that
                     $\dot c_{r+1}\leq r \mbox{ and } \dot c_{p}-1\leq \dot c_{p-1}\leq \dot c_{p}$ for each $p$ (recall that
                       $\dot\alpha_{p}:=\dot\beta_{o+p}$ and  $\dot c_0=0$).
                    Clearly, if a): $\dot c_{r+1}=0$, then all $\dot c_p$'s are equal to 0. If b):
                    $1\leq \dot c_{r+1}\leq r$ and any two non-zero $\dot c_p$ and $\dot c_{p'}$ are distinct, then
                     we have $\sum_{p=1}^{r+1}\dot c_p\dot \alpha_p=\sum_{p=1}^{m+1} p\dot \alpha_{r-m+p}$ where $0\leq m<r$.
                    Otherwise, we have c): $1\leq \dot c_{r+1}\leq r$ and there exist distinct $p<p'$ such that
                            $\dot c_p=\dot c_{p'}\neq 0$.  This must imply $\dot c_p=\dot c_{p+1}$ since $(\dot c_1, \cdots, \dot c_{r+1})$ is a non-decreasing sequence.
       Corresponding to these three cases, we have the following conclusions.
  \begin{enumerate}
    \item[a)]   $(\dot c_1, \cdots, \dot c_{r+1})=(0, \cdots, 0)$. 
                   Then we  have $gr_r(q_{\gamma^\vee})=\mathbf{0}$ and $ gr_r(us_\gamma)=gr_r(u)$
                    by Lemma \ref{gracompare11111}.
                   Thus   (L2) holds by Lemma \ref{coroofsumgrad}.
    \item[b)]  $\dot \gamma^\vee=\sum_{p=1}^{m+1} p\dot \alpha_{r-m+p}$, where $0\leq m<r$.
                  Hence,
                 we have   $gr_r(q_{\gamma^\vee})=gr_r(q_{\dot \gamma^\vee})=(m-r)\sum_{p=r-m}^r\mathbf{e}_p$
                   (by Lemma \ref{polycompforAAtype000000}) and
                 $\langle \alpha_p, \gamma^\vee\rangle=\langle \dot \alpha_p, \dot \gamma^\vee\rangle=0$ for   $p\in\{1, \cdots, r\}\setminus\{r-m\}$.
               By Proposition \ref{maingracompare}, we have
                  $gr_r(us_\gamma)\leq gr_r(u)+(r-m)\sum_{p=r-m}^r\mathbf{e}_p$.
              Thus  (L2) holds by
               Lemma \ref{coroofsumgrad}.

    \item[c)] In this case, we can take $j:=\min\{p~|~ 1\leq p\leq r, \dot c_p=\dot c_{p+1}\neq 0\}$.
           That is, $\sum_{p=1}^{j+1} \dot c_p\dot \alpha_{p}=m\dot \alpha_{j+1}+\sum_{p=1}^{m} p\dot \alpha_{j-m+p}$,
            where $1\leq m\leq j\leq r$.
              As a consequently, we have $\langle\alpha_j, \gamma^\vee\rangle=
                    \langle \dot \alpha_j, \dot \gamma^\vee\rangle>0$, $\dot\beta^\vee:=\dot\gamma^\vee -\dot \alpha_j^\vee$ is virtual coroot of $\beta^\vee(=\gamma^\vee-\alpha_j^\vee)$
                     and  $\mathbf{e}_c<\mathbf{e}_j$.     If $j=1$, then it is done by  (Ob2).
        If $j\geq 2$, then we use  (Ob3).
       Note that $gr_r(q_{\beta^\vee})=gr_r(q_{\dot\beta^\vee})$. By using Lemma \ref{polycompforAAtype000000},
    we conclude $(\lambda_1, \cdots, \lambda_{j-2})=(m-j)\sum_{p=j-m}^{j-2}\mathbf{e}_p$, $\lambda_{j-1}=(m-j)-(-j+1)=m-1$ and $\lambda_j=m(j+1)-mj-(j+1)=m-1-j$.  Hence, $M=\lambda_{j-1}-\lambda_j-j=0$.

   By the induction hypothesis, we have   $\sum_{p=1}^{j-2}\mu_p\mathbf{e}_p\leq \mathbf{0}$. If
     ``$<$" holds, it is already done. If ``$=$" holds,    we have $\mu_1= \cdots =\mu_{j-2} =0$ and  consequently $\mu_{j-1}\leq 0$.
      Write $gr(us_\gamma)=(\tilde k_1, \cdots, \tilde k_{r+1})$,  $gr(q_{\gamma^\vee})=(\tilde \lambda_1, \cdots, \tilde \lambda_{r+1})$ and $gr(us_j)=(\tilde i_1, \cdots, \tilde i_{r+1})$.
    Then  $\lambda_p+k_p=\tilde i_p +\mu_p$, $\tilde \lambda_p=\lambda_p$, $\tilde i_p=i_p$ and $\tilde k_p=k_p$ for $1\leq p \leq j-2$.
  Hence, $(\tilde k_1, \cdots, \tilde k_{j-2})=(i_1, \cdots, i_{j-2})+(j-m)\sum_{p=j-m}^{j-2}\mathbf{e}_p$.
         Since     $\langle \alpha_p, \gamma^\vee\rangle=\langle \dot \alpha_p, \dot \gamma^\vee\rangle=0$ for all  $p\in\{1, 2, \cdots, j-1\}\setminus\{j-m\}$,
    we conclude $(\tilde k_1, \cdots, \tilde k_{j-1})\leq (i_1, \cdots, i_{j-1})+(j-m)\sum_{p=j-m}^{j-1}\mathbf{e}_p$, by using Proposition \ref{maingracompare} with respect to
      $(\Delta, \Delta_{\bar P})$ with $\Delta_{\bar P}=(\alpha_1, \cdots, \alpha_{j-1})$.
   Thus  we have $\tilde k_{j-1}\leq i_{j-1}+(j-m)$.
   Since $\ell(us_j)=\ell(u)-1$ and $\ell(us_\gamma)=\ell(us_js_\beta s_j)=\ell(us_js_\beta)-1$, we have
     $\tilde i_j=i_{j-1}$ and $\tilde k_{j-1}=k_j-1$.
     Hence, $\mu_j= k_j+\lambda_j-\tilde i_j=\tilde k_{j-1}+1+(m-1-j)-i_{j-1}\leq 0$.

     Therefore, we conclude
      $(\mu_1, \cdots, \mu_{j-2}, \mu_{j}, \mu_{j-1})\leq (0, \cdots, 0)$ and consequently (L2) holds by (Ob3).
   \end{enumerate}

 When case (2) of Proposition \ref{coeffvircoroot} holds, we have  $\gamma^\vee=\sum_{p=d}^m\beta_p^\vee$ where $o\leq    m\leq o+r$ and $d<m$.
   If $m=o$, then $d<o$; consequently, we take $\alpha_j=\beta_d$ and use   (Ob1).
     If $m=o+1$ or $d=o+1$, then we   take 
                     $j=1$    and use  (Ob2).
     Otherwise, we have either $d\leq o< m=o+j$ or $o+j=d<m\leq o+r$, where $2\leq j\leq r$. Then we take such $j$
                 and  use  (Ob3). Note that   $\beta^\vee=\gamma^\vee-\alpha_j^\vee$ and  $\mathbf{e}_c<\mathbf{e}_j$.
                  For the former case, we have
                               $\lambda_{j-1}=j-1$  and $\lambda_j=-1$;  for the latter case, we have   $\lambda_{j-1}=0$  and $\lambda_j=-j$.
            Hence, we always have $M=\lambda_{j-1}-\lambda_j-j=0$.
    By the induction hypothesis again, we have   $\sum_{p=1}^{j-2}\mu_p\mathbf{e}_p\leq \mathbf{0}$. If
     ``$<$" holds, it is already done. If ``$=$" holds,    we have $\mu_1= \cdots =\mu_{j-2} =0$ and  consequently $\mu_{j-1}\leq 0$.
                  For the former case,  we conclude  $\mu_{j-1}=0$ and consequently $\mu_j\leq 0$, by noting $0\geq  \mu_{j-1}=k_{j-1}+(j-1)-(i_j-1)\geq 0$.
               For the latter case, we have $\mu_j=k_j-i_{j-1}+(-j)\leq 0$.
                Hence, we always have  $(\mu_1, \cdots, \mu_{j-2}, \mu_j, \mu_{j-1})\leq (0, \cdots, 0)$.   Thus (L2) holds.

  It remains to consider  the case when (3) of Proposition \ref{coeffvircoroot} holds. That is, $\mbox{C}9)$ occurs and
      $\gamma^\vee=\beta_3^\vee+\beta_4^\vee$.
     Then  we just take $\alpha_j=\beta_4$  and   use (Ob1). Thus (L2) still holds.
   \end{proof}

\subsection{Two consequences}\label{subsectwoconseq}
In this subsection, we derive two propositions with the help of our notion of virtual coroot.

\begin{prop}\label{gracomponemore}
   Let $u\in W^P$ and $1\leq j\leq r$. Then $\sigma^u\star \sigma^{s_j}=\sigma^{us_j}+\sum_{w, \lambda}b_{w, \lambda}q_\lambda \sigma^w$
    with $gr(q_\lambda w)<gr(us_j)$ whenever $b_{w, \lambda}\neq 0$.
\end{prop}

\begin{proof}
  Clearly, $\ell(us_j)=\ell(u)+1$. Thus $N_{u, s_j}^{us_j, 0}=\langle \chi_j, \alpha_j^\vee\rangle =1$ by quantum Chevalley formula (Proposition \ref{quanchevalley}).
   We need to analyze the remaining non-zero terms.

  If   $\ell(us_\gamma)=\ell(u)+1$ and  $\langle \chi_j, \gamma^\vee\rangle \neq 0$, then we have $gr(us_\gamma)\leq gr(u)+gr(s_j)$ by
   part a) of the Key Lemma. Note that $gr(us_j)=\mathbf{e}_j+\ell(u)\mathbf{e}_{r+1}$.
   If the  equality holds, then we have  $us_\gamma=vs_j$ where $\ell(v)=\ell(u)$ and $v\in W^P$.
 By Lemma \ref{strongexchange},   an expression of $u\in W^P$ is obtained
         by deleting a simple reflection from a  (fixed)  reduced expression of $vs_j$. Note that this simply reflection cannot
         come from $v$. (Otherwise, we denote by $\bar v$ the element obtained by  deleting such  simple reflection from $v$. Then
             $u=\bar v s_j$ and we would deduce a contradiction, saying
            $1+\ell(u)=\ell(us_j)=\ell(\bar v)<\ell(v)=\ell(u)$.)
         Thus   $u=v$.

  If $\ell(us_\gamma)=\ell(u)+1-\langle 2\rho, \gamma^\vee\rangle$ and $\langle \chi_j, \gamma^\vee\rangle \neq 0$, then
    we have $gr(q_{\gamma^\vee}us_\gamma)\leq gr(u)+gr(s_j)$ by part b) of the Key Lemma. Furthermore, we have
       $\ell(us_\alpha)=\ell(u)-1$ whenever $\langle \alpha, \gamma^\vee\rangle>0$, by Lemma \ref{lengthofposiroot} and Lemma \ref{lengthofpos222}.
    Since $u\in W^P$, $\ell(us_p)=\ell(u)+1$ for any $\alpha_p\in \Delta_P$.
        If  the equality held, then we would deduce a contradiction as follows (mainly by finding $\alpha\in \Delta_P$ satisfying
                 $\langle \alpha, \gamma^\vee\rangle>0$).

   Note that $\ell(s_\gamma)>1$ (otherwise, we would conclude $\gamma=\alpha_j\in \Delta_P$).

   We first assume $r\geq 2$ and write
         $gr(us_\gamma)=(\tilde k_1, \cdots, \tilde k_{r+1})$  and $gr(q_{\gamma^\vee})$ $=(\tilde \lambda_1, \cdots,
          \tilde \lambda_{r+1})$. Since the grading equality holds,  we have $\tilde k_p+\tilde \lambda_p=\delta_{p, j}$ for each $1\leq p\leq r$.
  As before, we discuss all possible coroots by using Proposition \ref{coeffvircoroot}.

   When (1) of Proposition \ref{coeffvircoroot} holds, $\gamma^\vee$ has  a virtual
                  coroot $\dot\gamma^\vee=\sum_{p=1}^{r+1}\dot c_p\dot \alpha_p^\vee$ satisfying one and only one of the followings (from the proof of Proposition
     \ref{mainpropprod}).
     \begin{enumerate}
    \item[a)]   $(\dot c_1, \cdots, \dot c_{r+1})=(0, \cdots, 0)$.
                   In this case, we  have $gr_r(q_{\gamma^\vee})=\mathbf{0}$ and $gr_r(us_\gamma)$ $=gr_r(u)$. In particular,
                     we have  $\tilde k_j=\tilde \lambda_j=0$, deducing a contradiction: $1=\delta_{j, j}=\tilde k_j+\tilde \lambda_j=0+0$.
    \item[b)]  $\dot \gamma^\vee=\sum_{p=1}^{m+1} p\dot \alpha_{r-m+p}$, where $0\leq m<r$.
                 Then
                 we have   $gr_r(q_{\gamma^\vee})=gr_r(q_{\dot \gamma^\vee})=(m-r)\sum_{p=r-m}^r\mathbf{e}_p$ and
                 $\langle \alpha_p, \gamma^\vee\rangle=0$ for   $p\in\{1, \cdots, r\}\setminus\{r-m\}$.
            Note that $gr_{r-m-1}(q_{\gamma^\vee})=\mathbf{0}$ and $gr_{r-m-1}(us_\gamma)=gr_{r-m-1}(u)=\mathbf{0}$.
                       If $j\leq r-m-1$, then we would deduce the  contradiction  $1=\delta_{j, j}=\tilde k_j+\tilde \lambda_j=0+0$ again.
            If $j=r-m$, then we still deduce a contradiction: $1=\delta_{j, j}=\tilde k_j+\tilde \lambda_j=\tilde k_{r-m}+m-r\leq 0$.
                  Hence, we  conclude $j>r-m$. Then we have $r\geq j>j-1\geq r-m>0$ and $\langle \alpha_j, \gamma^\vee\rangle= 0$. Thus we have
                    $\tilde k_j=\delta_{j, j}-\tilde \lambda_j=1+r-m$, $\tilde k_{j-1}=\delta_{j, j-1}-\tilde \lambda_{j-1}=r-m$
                     and consequently $\tilde k_j=\tilde k_{j-1}+1$.  Hence, we have
                      $\ell(us_\gamma s_{j})=\ell(us_\gamma)-1$  by Proposition \ref{gracompare000}.
                      Then by Lemma \ref{strongexchange}, we conclude
                          $us_\gamma (\alpha_{j})\in -R^+$, contrary to $us_\gamma(\alpha_j)=u(\alpha_j)\in R^+$.
    \item[c)]
                   $\sum_{p=1}^{i+1} \dot c_p\dot \alpha_{p}=m\dot \alpha_{i+1}+\sum_{p=1}^{m} p\dot \alpha_{i-m+p}$,
            where $1\leq m\leq i\leq r$.
             Then  we have $\langle\alpha_i, \gamma^\vee\rangle=
                    \langle \dot \alpha_i, \dot \gamma^\vee\rangle>0$ and therefore   deduce a contradiction. 
      \end{enumerate}

  When (2) of Proposition \ref{coeffvircoroot} holds, we have  $\gamma^\vee=\sum_{p=d}^m\beta_p^\vee$ where $o\leq  m\leq o+r$ and $d<m$. Since
    $\langle \chi_j, \gamma^\vee\rangle \neq 0$, we conclude $m\geq o+r$. Thus we find $\alpha=\beta_m\in \Delta_P$ that satisfies $\langle\alpha,\gamma^\vee\rangle>0$.
     Hence, we deduce a contradiction in this case.

  It remains to consider the case when (3) of  Proposition \ref{coeffvircoroot} holds. That is,  $\mbox{C}9)$ occurs and $\gamma^\vee= \beta^\vee_3+\beta_4^\vee$.
  In this case, we note $r=2$ and  deduce a contradiction, saying
             $-4=\tilde \lambda_2=\delta_{2, j}-\tilde k_2\geq 0-2=-2$.

 Hence, our assumption that the grading equality holds is not true, when  $r\geq 2$.

 Now we assume $r=1$. Then $\alpha_j=\alpha_1$ and we have $\tilde \lambda_1=\langle \alpha_1, \gamma^\vee\rangle$ by Corollary \ref{gradforonenode}.
 If $\tilde \lambda_1>0$, then we   find a contradiction by taking $\alpha=\alpha_1\in\Delta_P$.
 If $\tilde \lambda_1<0$, then $gr_1(q_{\gamma^\vee}us_\gamma)<0+\tilde k_1\leq 1=gr_1(us_1)$ and consequently the grading equality does not hold.
 If $\tilde \lambda_1=0$, then $\tilde k_1=gr_1(us_\gamma)=gr_1(u)=0$ and consequently we deduce the contradiction
    $1=\delta_{j, 1}=\tilde k_1+\tilde \lambda_1=0+0$.

   Due to quantum Chevalley formula, we have discussed all the non-zero terms for the quantum product $\sigma^u\star \sigma^{s_j}$. Hence,
   the statement follows.
\end{proof}

 By Theorem \ref{mainthm}, we obtain a filtered-algebra structure on $QH^*(G/B)$, which induces an associated graded subalgebra along the $\mathbb{Z}\mathbf{e}_{r+1}$
  direction. Thanks to the Peterson-Woodward comparison formula and
       our definition of $gr(q_j)$'s (with the help of PW-lifting), we wish to obtain
       an algebra isomorphism between $QH^*(G/P)$ and (at least a subalgebra of) this graded subalgebra.
       For this it is necessary that
       the gradings of  $\psi_{\Delta, \Delta_P}(1,q_{\lambda_P})$'s, which are canonical candidates
        in $QH^*(G/B)$ playing the role of the polynomials $q_{\lambda_P}$'s in $QH^*(G/P)$,  are in $\mathbb{Z}\mathbf{e}_{r+1}$.
  Indeed,  the Peterson-Woodward comparison formula, together with
       our definition of $gr(q_j)$'s, has shown that
     $gr_r(\psi_{\Delta, \Delta_P}(1,q_{\lambda_P}))=\mathbf{0}$ whenever $q_{\lambda_P}\in QH^*(G/P)$ occurs in the quantum
      product $\sigma^{u}\star_P\sigma^v$ for $\sigma^{u}, \sigma^v\in QH^*(G/P)$.
 However, apparently  it does not tell us about the behavior when the degree of $q_{\lambda_P}$ is large. Therefore we need
    the following proposition for later use.
 \begin{prop}\label{allgradcomp}
   $gr_r(\psi_{\Delta, \Delta_P}(1,q_{\lambda_P}))=\mathbf{0}$ whenever $q_{\lambda_P}\in QH^*(G/P)$.
\end{prop}

 The idea of the proof is as follows. We write $\psi_{\Delta, \Delta_P}(1,q_{\lambda_P})=q_{\lambda_B}\omega_P\omega'$ as before.
 The case when $r=1$ is easy to handle. When $r\geq 2$, we can use our notion of virtual coroot to obtain
  $\lambda_B$ and consequently $\omega_P\omega'$ and  $gr_r(q_{\lambda_B})$. More precisely,
   we write $\lambda_P=\lambda'+Q^\vee_P$ with $\lambda'=\sum_{\alpha\not\in \Delta_P} a_\alpha \alpha$.
     Consider a virtual coroot $\dot \lambda'$ of $\lambda'$, then we can easily write down the element
      $\dot \lambda'+\sum_{i=1}^ra_i\dot \alpha_i^\vee$ associated to
       $\dot\lambda'+\dot Q_P^\vee\in \dot Q^\vee/\dot Q^\vee_P$, where $\dot Q^\vee_P:=\bigoplus_{i=1}^r\mathbb{Z}\dot \alpha_i^\vee$.
        For instance the case of $m\dot \beta_{o+r+1}^\vee+\dot Q^\vee_P$ has been
        studied in Lemma \ref{polycompforAAtype000000}.  By our definition of virtual coroot, we conclude
         $\lambda'+\sum_{i=1}^ra_i\alpha_i^\vee$ is the element that we expect. In addition, we also show  $a_i$'s are indeed non-negative so that
          $q_{\lambda'+\sum_{i=1}^ra_i  \alpha_i^\vee}\in QH^*(G/B)$.

\begin{proof}[Proof of Proposition \ref{allgradcomp}]
Write $\psi_{\Delta, \Delta_P}(1,q_{\lambda_P})=q_{\lambda_B}\omega_P\omega' $. When  $r=1$, we have   $gr_1(q_{\lambda_B})=\langle \alpha_1,  \lambda_B\rangle=:k_1$
    by Corollary \ref{gradforonenode}. Thus $k_1\in\{0, -1\}$ following from the definition of $\lambda_B$.
  If $k_1=0$, then $\Delta_{P'}=\{\alpha_1\}$, implying  $\omega'=s_1=\omega_P$ and $\omega_P\omega'=1$. Thus
    $gr_1(q_{\lambda_B})+ gr_1(\omega_P\omega')=0+0=0$.
  If, $k_1=-1$. Then we have $\Delta_{P'}=\emptyset$, implying   $\omega_P\omega'=s_1\cdot 1=s_1$. Thus  $gr_1(q_{\lambda_B})+ gr_1(\omega_P\omega')=-1+1=0$.. Hence, the statement holds when $r=1$.

  Now we assume $r\geq 2$. We consider the virtual coroots
    and introduce  some special elements in $\dot Q^\vee$ first.
     Denote $\mu_m=\sum_{k=1}^{m}k\dot\beta_{o+r+1-m+k}^\vee$.  Whenever $o>0$   we denote
      $\nu_m=\sum_{k=0}^{m-1}(m-k)\dot \beta_{o+k}^\vee$ and  $\varrho= \sum_{k=o}^{o+r+1}\dot\beta_{k}^\vee$ where $1\leq m\leq r+1$.
   By direct calculations, we conclude $gr_r(q_{x})=\mathbf{0}$ and $\langle \dot \alpha_i, x\rangle=0$ for all $1\leq i\leq r$ whenever  $x=\mu_{r+1}$,
     $\nu_{r+1}$ or $\varrho$.  That is, $\mu_{r+1}$,  $\nu_{r+1}$ and  $\varrho$ are all virtual coroots of $0\in Q^\vee$.
    Furthermore for $1\leq m\leq r$, we have
 $$ gr_r(q_{\nu_m})=-m\sum\nolimits_{j=1}^{r}\mathbf{e}_{j}
                                  +   \sum\nolimits_{k=1}^{m-1}(m-k)\big((k+1)\mathbf{e}_{k}-(k-1)\mathbf{e}_{k-1}\big)
                                =-m\sum\nolimits_{k=m}^{r} \mathbf{e}_{k},$$
  and $gr_r(q_{\mu_m})=-(r+1-m) \sum_{k=r+1-m}^{r} \mathbf{e}_{k}$ (by Lemma \ref{polycompforAAtype000000}).

Write $\lambda_P=\lambda'+Q^\vee_P$ with $\lambda'=\sum_{j=1}^o b_j\beta_j^\vee+\sum_{j=o+r+1}^nb_j\beta^\vee_j$.
From Table \ref{tabvircoroot},   we obtain a virtual coroot
      $\dot \lambda'= \sum_{j=1}^{\eta} \tilde b_j\dot \beta_j^\vee$ of $\lambda'$  in which we note $\tilde b_o=b_o$ and $o+r+1\leq \eta$.
 If     $\tilde b_o\leq \tilde b_{o+r+1}$, we set $y=\tilde b_o \varrho +a\mu_{r+1}+\mu_m$ where
   $\mu_{0}:=0$ and  $\tilde b_{o+r+1}-\tilde b_{o}=a(r+1)+m$ with    $0\leq m\leq r$ and $a\geq 0$.
    Similarly if  $\tilde b_o> \tilde b_{o+r+1}$, we set $y=\tilde b_{o+r+1} \varrho +a\nu_{r+1}+\nu_m$
      where $\nu_{0}:=0$ and
     $\tilde b_{o}-\tilde b_{o+r+1}=a(r+1)+m$ with   $0\leq m\leq r$ and $a\geq 0$. Clearly, we can write
       $y=\tilde b_o\dot\beta_o+\tilde b_{o+r+1}\dot\beta_{o+r+1}+\sum_{i=1}^rd_i\dot \alpha_i$.
 Note that $\dot \beta_p^\vee$ is a  virtual coroots of $0\in Q^\vee$ whenever  $p< o \mbox{ or } p>o+r+1$.
  Thus we conclude $y$ is virtual coroot of $\lambda_B:=\lambda'+\sum_{i=1}^r(d_i-\tilde b_{o+i})\alpha_i$.
 Furthermore, we note that for  all $\alpha_i\in \Delta_P$ we have
       $\langle \alpha_i, \lambda_B\rangle=\langle \dot\alpha_i, y\rangle=\langle \dot \alpha_i, x\rangle=  \Big\{\begin{array}
         {rl}-1,& \mbox{if } \dot \alpha_i=\dot \alpha_m \mbox{ (resp. } \dot \alpha_{r+1-m})\\0, & \mbox{otherwise}
       \end{array}$  where $x=\nu_m$ (resp. $\mu_m$), if $\tilde b_o>\tilde b_{o+r+1}$ (resp. $\tilde b_o\leq \tilde b_{o+r+1}$).
  Hence, $\lambda_B$ is the very one associated to $\lambda_P$ that we are expecting.
  Correspondingly, we can directly write down $\Delta_{P'}$ as well as $gr(\omega_P\omega')$ $=-gr_r(q_x)$  by Lemma \ref{longestcomp000}. Note that
    $gr_r(q_{\lambda_B})=gr_r(q_x)$. Hence, $gr_r(q_{\lambda_B}\omega_P\omega')=\mathbf{0}$.

  Since $q_{\lambda_P}\in QH^*(G/P)$,  $b_j\geq 0$ for each $j$. It remains to show $q_{\lambda_B}\in QH^*(G/B)$. That is, we need to show
     $d_i-\tilde b_{o+i}$ is  non-negative for each $1\leq i\leq r$.
  Clearly, only the part $\sum_{j=o+r+1}^nb_j\beta^\vee_j$ of $\lambda'$ make contributions for the part
    $\sum_{i=1}^r \tilde b_i\dot \beta_{o+r}^\vee$ of the coroot $\dot \lambda'$ of $\lambda'$. From Table \ref{tabvircoroot}
       we see     that $\tilde b_{o+1}\leq\tilde b_{o+2}\leq \cdots \leq \tilde b_{o+r+1}$. Thus if $\tilde b_o\geq \tilde b_{o+r+1}$,
        then  we have $d_i-\tilde b_{o+i}\geq \tilde b_o-\tilde b_{o+r+1}\geq 0$ for each $1\leq i\leq r$.
       Now we consider the case when $(0\leq b_o=)\tilde b_o< \tilde b_{o+r+1}$ and then note that all $d_i$'s are non-negative from the way we obtain them.
   From Table \ref{tabvircoroot} and Table \ref{tabrelativeposi}, we can make the following observations.
     (i) If case $\mbox{C}1), \mbox{C}9)$ or $\mbox{C}10)$ occurs, then the virtual coroot $\dot \lambda'$ does not
      make contributions on these $\tilde b_{o+i}$'s. That is,  we have  $d_i-\tilde b_{o+r}=d_i\geq 0$ for each $1\leq i\leq r$.
      (ii) For the remaining cases, we have $\tilde b_{o+1}=\cdots=\tilde b_{o+r-1}=0\leq 2\tilde b_{o+r}\leq \tilde b_{o+r+1}$, except for the case
       when $\mbox{C}4)$ occurs with $r\geq 5$ and $o+r=7$.  (iii) For the only exceptional case, we have
            $\tilde b_{o+1}=\tilde b_{o+2}=\cdots=\tilde b_{5}=0$,  $\tilde b_{6}=b_8$, $\tilde b_{7}=2 b_{8}$ and $\tilde b_{8}= 3b_{8}>0$.  Recall that
    $\sum_{i=1}^rd_i\dot \beta_{o+i}=\tilde b_o \varrho +a\mu_{r+1}+\mu_m-(\tilde b_o\dot\beta_o+\tilde b_{o+r+1}\dot\beta_{o+r+1})$ in which
      $\tilde b_{o+r+1}=\tilde b_o+a(r+1)+m$.
  When (ii) holds, we have $d_r-\tilde b_{o+r}\geq\tilde b_o+ar+(m-1)-\tilde b_{o+r} \geq \tilde b_o+ar+m-1-[{\tilde b_{o+r+1}\over 2}]
     \geq ar+m-1-[{\tilde b_{o+r+1}-\tilde b_o\over 2}]
        =ar+m-1-[{a(r+1)+m\over 2}]\geq 0$, and note $d_i-\tilde b_{o+i}=d_i\geq 0$ for $1\leq i\leq r-1$.
  When (iii) holds,
  we have  $\tilde b_8\geq 0+a(5+1)+0$ so that $a\leq {b_8\over 2}$.  Since $a$ is an integer, $a\leq [{b_8\over 2}]$. Thus we have
           $d_r-\tilde b_{o+r}\geq \tilde b_o+ar+m-1-2b_8=\tilde b_8-a-1-2b_8=b_8-a-1\geq b_8-[{b_8\over 2}]-1\geq 0$
        and     $d_{r-1}-\tilde b_{o+r-1}\geq \tilde b_o+a(r-1)+m-2-b_8=\tilde b_8-2a-2-b_8=2(b_8-a-1)\geq 0$.
  For $1\leq i\leq r-2$, we have $d_i-\tilde b_{o+i}=d_i\geq 0$.
Hence, we do show $  d_i-\tilde b_{o+i}\geq 0$ for  $1\leq i\leq r$ for all cases.
\end{proof}

 \begin{remark} In \cite{lamshi},
     Lam and Shimozono have given a combinatorial description of $\lambda_B$. In our case when $\Delta_P$ is of $A$-type, we obtain another way
      to describe $\lambda_B$  and to show the property  $q_{\lambda_B}\in QH^*(G/B)$  in the above proof.
\end{remark}

\subsection{Proof of the Key Lemma  for general $\Delta_P$}\label{sectiongenconn}
 In this subsection, we assume $\Delta_P$ is not of $A$-type. We  give the proof of the Key Lemma,  after describing the formulas for the gradings of all  $q_j$'s.
Recall that   $\varsigma=r-1$ and in this case we have  replaced $r$ with $\varsigma$ in Table \ref{tabrelativeposi}, in order to
fix the order $(\Delta_P, \Upsilon)$. In particular, we have $\kappa=o+\varsigma$ in this subsection.

Using
Definition \ref{defgrading} with respect to the
  ordered subset  $\Delta_\varsigma=(\alpha_1, \cdots, \alpha_\varsigma)$ of $(\Delta_P, \Upsilon)$,   we obtain a grading map
   $$\tilde {gr}=gr_{\Delta_\varsigma}: W\times Q^\vee\longrightarrow \mathbb{Z}^{\varsigma+1}.$$
That is, we define $\tilde{gr}(w)=\sum_{j=1}^{\varsigma+1}\ell(v_j)\mathbf{e}_j$ using the decomposition $w=v_{\varsigma+1}v_{\varsigma}\cdots v_1$ of $w\in W$ associated to
 ordered subset $\Delta_\varsigma=(\alpha_1, \cdots, \alpha_\varsigma)$,
   define $\tilde {gr}(q_1)=2\mathbf{e}_1$ and define the remaining $\tilde{gr}(q_j)$'s recursively with the help of   PW-lifting
    $\{\psi_{\Delta_2, \Delta_1},$  $\psi_{\Delta_3, \Delta_2}, \cdots, \psi_{\Delta_\varsigma, \Delta_{\varsigma-1}}, \psi_{\Delta, \Delta_{\varsigma}}\}$.
 Let $\iota:\mathbb{Z}^{\varsigma+1}=\mathbb{Z}^r\hookrightarrow \mathbb{Z}^{r+1}$ be the natural inclusion.
 Thus we obtain a map  $\iota\circ\tilde {gr}: W\times Q^\vee\longrightarrow \mathbb{Z}^{r+1}$,   which we simply denote as $\tilde{gr} $
whenever there is no confusion.

 As a direct consequence of the definition of $\tilde{gr}$, we can apply Proposition \ref{graquanvar} and Proposition \ref{gradingforconn}
 with respect to   the ordered  subset  $\Delta_\varsigma$, so that  we have
 \begin{lemma}\label{gengralemm111}
   $gr(q_j)=\tilde {gr}(q_j)$ for each $1\leq j\leq \varsigma+1$. Precisely,
  $gr(q_1)=2\mathbf{e}_1;$  $gr(q_j)=(1-j)\mathbf{e}_{j-1}+(1+j)\mathbf{e}_j$ for $2\leq j\leq \varsigma$;
     $gr(q_{\varsigma+1})$ is obtained  by   directly replacing $r$ with $\varsigma$. (Only case  {\upshape c), d), e)} or {\upshape f)} in Table \ref{tabgrading}
      can occur.)
 \end{lemma}

Furthermore we note from Table
\ref{tabrelativeposi} that either $o\geq 1$ or $\kappa+2\leq n$ must
hold. For any
  $\alpha\in \Delta \setminus(\Delta_P\cup \{\beta_o, \beta_{\kappa+2}\})$, we have
  $gr(q_{\alpha^\vee})=2\mathbf{e}_{r+1}<2\mathbf{e}_{\varsigma+1}=\tilde {gr}(q_{\alpha^\vee})$.
For $\lambda_P\in Q^\vee/Q^\vee_P$,  we write
$q_{\lambda_B}\omega_P\omega' =\psi_{\Delta, \Delta_P}(1,
\lambda_P)$ as before.
\begin{lemma}\label{gengralemm222}
 Suppose  $p\in \{o, \kappa+2\}\cap \{1, \cdots, n\}$. Set $\lambda_P=\beta_p^\vee+Q_P^\vee$.
   Then we have   $\lambda_B= \beta_{p}^\vee$  except for   either of the following cases.
  \begin{enumerate}
    \item    $p=o$ and    $\{\beta_o\}\cup \Delta_P$ is of  $C$-type. In this case,    $\lambda_B=\beta_o^\vee+\sum_{j=1}^r\alpha_j^\vee$.
    \item   $p=\kappa+2$ and  {\upshape $\mbox{C}9)$} occurs. In this case,   $ \lambda_B =\beta_{\kappa+2}^\vee+\alpha_{r-1}^\vee+\alpha_r^\vee$.
  \end{enumerate}
       Furthermore, we    write $gr(q_{\beta_p^\vee})=\sum_{j=1}^{r+1}d_j\mathbf{e}_j$ and    $\tilde {gr}(q_{\beta_p^\vee})=\sum_{j=1}^{r+1}\tilde d_j\mathbf{e}_j$,
            and  denote   $\Delta_{\tilde P}=\{\alpha\in \Delta_P~|~ \langle \alpha, \beta_p^\vee   \rangle =0\}$. 
    Then $d_j=\tilde d_j$ for $1\leq j\leq \varsigma$ and  we have
     {\upshape $$\mbox{a) }   gr(q_{\beta_p^\vee})<\tilde {gr}(q_{\beta_p^\vee}); \,\, \mbox{b) } d_{r+1}\leq \ell(\omega\omega_P)+1;\,\,
          \mbox{c) } 
               \sum\nolimits_{j=1}^rd_j\leq - \ell(\omega_P \omega_{\tilde P}).$$
     }
 \end{lemma}

\begin{proof}
    Let $\theta_P =\sum_{j=1}^ra_j\alpha_j$  denote the highest root in $R_P$. Note that $\ell(\omega_P  \omega_{\tilde P})=|R_P^+|-|R^+_{\tilde P}|$,
        $\ell(\omega\omega_P)=|R^+|-|R^+_{P}|$ and $\Delta_{P'}=\{\alpha\in \Delta_P~|~ \langle \alpha, \lambda_B\rangle=0\}$.

    We first assume $p=o$
          and note that $\tilde {gr}(q_{\beta_o^\vee})=(\varsigma+2)\mathbf{e}_{\varsigma+1}-  \sum_{j=1}^\varsigma \mathbf{e}_j$.

    Whenever    $\{\beta_o\}\cup \Delta_P$ is not of  $C$-type,
             we note (Table \ref{tabrelativeposi} and \cite{hum}) that $a_1=1$,  $\langle \alpha_1, \beta_o^\vee\rangle =-1$
             and $\langle \alpha_j, \beta_o^\vee\rangle =0$ for  $2\leq j\leq r$. Hence, we conclude
           $\lambda_B=\beta_o^\vee$, $\Delta_{P'}=\Delta_{\tilde P}=\Delta_P\setminus\{\alpha_1\}$ and consequently
           we have $\omega'=\omega_{\tilde P}$ and $gr(q_{\beta_p^\vee})=(\ell(\omega_P\omega')+2)\mathbf{e}_{r+1}-gr(\omega_P\omega')$.
            Hence, $d_{r+1}=\ell(\omega_P\omega')+2\leq \ell(\omega\omega_P)+1$ by direct calculations.
    Write $\omega_P\omega'=v_ru$, where
        $v_{r}\in W_P^{P_\varsigma}$ and $u=W_{P_\varsigma}$.
            Then we have $u(\alpha_j)\in R_{P_\varsigma}^+$ for all $\alpha_j\in \Delta_\varsigma\cap \Delta_{P'}=\Delta_\varsigma\setminus\{\alpha_1\}$ (otherwise, we
        would conclude $v_ru(\alpha_j)\in -R^+$, contrary to $\omega_P\omega'\in W_{P}^{P'}$).
       Noting $u(\alpha_1)\in -R^+$,
           we deduce  $u=s_k\cdots s_2s_1$ for some $1\leq k\leq \varsigma$, by Lemma \ref{longestcomp000}.
           If $\Delta_P$ is of $B$-type (resp. $D$-type), then we conclude
              $\omega_P\omega'= v_rs_{\varsigma}\cdots s_1$ with $v_r=s_1\cdots s_r$ (resp. $v_r=s_1\cdots s_{r-2} s_r$) by easily checking
                 such element satisfies the condition in Lemma \ref{charalongest}.
            If $\Delta_P$ is of $E$-type, we note that $s_\varsigma\cdots s_1(\alpha_r)=u(\alpha_r)+\sum_{j=k+1}^\varsigma b_j\alpha_j$ for
             non-negative integers $b_j$'s. Consequently, we have $v_rs_\varsigma\cdots s_1\in W_P^{P'}$ and
                 $\ell(v_rs_\varsigma\cdots s_1)=\ell(v_ru)+\varsigma-k=\ell(\omega_P)-\ell(\omega')+\varsigma-k$. Thus  $k=\varsigma$  by Lemma
                  \ref{charalongest}, implying  $\ell(v_r)=\ell(\omega_P\omega')-\varsigma$.
             For all these  cases,
               we deduce $gr(\omega_P\omega')=(\ell(\omega_P\omega')-\varsigma)\mathbf{e}_{\varsigma+1}+\sum_{j=1}^\varsigma \mathbf{e}_j$.
         Hence, $d_j=-1$ for $1\leq j\leq \varsigma$,  $d_r=\varsigma-\ell(\omega_P\omega')$ and $gr(q_{\beta_o^\vee})<\tilde {gr}(q_{\beta_o^\vee})$.
                 Thus $\sum_{j=1}^rd_j=-\ell(\omega_P\omega')=-\ell(\omega_P\omega_{\tilde P})$.

     Assume   $\{\beta_o\}\cup \Delta_P$ is of  $C$-type,
         in which case there are only two possibilities saying (i) case $\mbox{C}1)$ with $\Delta$ being of $C$-type and (ii) case $\mbox{C}10)$ with $r=2$.
      Then  $\Delta_P$ itself is of $C$-type. Thus
       we have   $\ell(\omega_P  \omega_{\tilde P})=r^2-(r-1)^2=2r-1$ and $\ell(\omega\omega_P)=n^2-r^2\geq (r+1)^2-r^2=2r+1$.
      Furthermore,
     we conclude
           $\lambda_B=\beta_o^\vee+\sum_{j=1}^{r}\alpha_j^\vee$, by noting such element satisfies
            $\langle \alpha_j, \lambda_B\rangle =0$ for each $1\leq j\leq r$. Thus $\Delta_{P'}=\Delta_P$, $\omega_P\omega'=1$  and then
                     $gr(q_{\beta_o^\vee})=(2\varsigma+4)\mathbf{e}_{\varsigma+2}-(\varsigma+2)\mathbf{e}_{\varsigma+1}-\sum_{j=1}^\varsigma \mathbf{e}_j$ by definition.
    In particular, we have  $gr(q_{\beta_o^\vee}) <\tilde {gr}(q_{\beta_o^\vee})$,
     $d_{r+1}= 2\varsigma+4=2r+2\leq \ell(\omega\omega_P)+1$  and
       $\sum_{j=1}^rd_j=d_r-\varsigma=-2\varsigma-2=-2r<-\ell(\omega_P \omega_{\tilde P})$.

       Now we assume $p=\kappa+2$,  which holds only  if case $ \mbox{C} 5), \mbox{C}7), \mbox{C}9)$ or $\mbox{C}10)$ in Table \ref{tabrelativeposi} occurs.
  Note that $\tilde {gr}(q_{\beta_{\kappa+2}^\vee})=2\mathbf{e}_{\varsigma+1}$.
         If $\mbox{C}9)$ does not occur, then we       conclude $ a_r=1$,  $\langle \alpha_r, \beta_{\kappa+2}^\vee\rangle =-1$
             and $\langle \alpha_j, \beta_{\kappa+2}^\vee\rangle =0$ for  $1\leq j\leq \varsigma$. Hence,
           $\lambda_B=\beta_{\kappa+2}^\vee$, $\Delta_{P'}=\Delta_{\tilde P}=\Delta_\varsigma$ and consequently
           we have  $gr(q_{\beta_{\kappa+2}^\vee})=(\ell(\omega_P \omega_{\tilde P})+2)\mathbf{e}_{\varsigma+2}-\ell(\omega_P\omega_{\tilde P})\mathbf{e}_{\varsigma+1}$.
   Therefore,   a),  and c) follow, so does b) by direct  calculations.
    If $\mbox{C}9)$ occurs,   then
$|R^+|=24$, $|R_P^+|=r^2$,  $|\Delta_{\tilde P}|={r(r+1)\over 2}$,  $n=4=\kappa+2$ and  $r\in \{2, 3\}$.
By direct calculations, we conclude $\lambda_B=\beta_4^\vee+\alpha_{r-1}^\vee+\alpha_r^\vee$. Furthermore if $r=2$,
 then $\Delta_{P'}=\Delta_P$ and consequently we have $gr(q_{\beta_4^\vee})=6\mathbf{e}_{\varsigma+2}-4\mathbf{e}_{\varsigma+1}$.
 If $r=3$, then $\Delta_{P'}=\{\alpha_2, \alpha_3\}$. Consequently, $\omega_P\omega'=s_1s_2s_3s_2s_1$ with $gr(\omega_P\omega')=(1,1,3,0)$.
 Hence, we have $gr(q_{\beta_4^\vee})=11\mathbf{e}_{\varsigma+2}-9\mathbf{e}_{\varsigma+1}$. For either of the cases, it is easy to check all the statements hold.
  \end{proof}

 From the above discussions, we note that $gr_{\varsigma}(q_j)=\tilde {gr}_{\varsigma}(q_j)$ for all $j$.
 Using these discussions  together with  Lemma \ref{coroofsumgrad}, we obtain the following immediately.

\begin{lemma}\label{grachange}
  Let $\gamma\in R^+$. Write  $gr(q_{\gamma^\vee})=\sum_{j=1}^{r+1}d_j\mathbf{e}_j$ and    $\tilde {gr}(q_{\gamma^\vee})=\sum_{j=1}^{r+1}\tilde d_j\mathbf{e}_j$.
      Then we have       $d_{r}+d_{r+1}=\tilde d_{r}+\tilde d_{r+1}=\tilde d_r$ and $d_j=\tilde d_j$ for each $1\leq j\leq \varsigma$.
\end{lemma}

Now we give the proof of the Key Lemma as follows.

  \begin{proof}[Proof of the Key Lemma]
 Let $w\in W$ and take its   decomposition $w=v_{r+1}\cdots v_1$  associated to   $(\Delta_P, \Upsilon)$.
  Suppose $\ell(ws_\gamma)<\ell(w)$, then by Lemma \ref{strongexchange} we conclude
     $ws_\gamma=v_{r+1}\cdots v_{m+1}\bar v_m v_{m-1}\cdots v_1$ for a unique  $1\leq m\leq r+1$, in which
     $\bar v_m$ is obtained by deleting a unique simple reflection from (a fixed reduced expression of) $v_m$.
   Set  $D:=\big(gr(ws_\gamma)-gr(w)\big)-\big(\tilde {gr}(ws_\gamma)-\tilde {gr}(w)\big)$.
  If $1\leq m\leq r$, then we have   $D=0$ and  $\gamma\in R_P$. Furthermore,   we have $gr(q_\gamma)=\tilde {gr}(q_\gamma)$ and $gr(s_i)=\tilde {gr}(s_i)$
       whenever $\langle \chi_i, \gamma^\vee\rangle\neq0$. In particular, the Key Lemma holds  for such $\gamma$, by using Proposition \ref{mainsecondprop}
        and Proposition \ref{mainpropprod}  with respect to the ordered subset $\Delta_\varsigma$.
  If $m=r+1$,   we write $\bar v_{r+1}v_r=\tilde v_{r+1}\tilde v_r u'$ with
       $\tilde v_{r+1}\in W^P$, $\tilde v_r\in W_P^{P_\varsigma}$ and $ u'\in W_{P_\varsigma}$.
   Thus $ws_\gamma=\tilde v_{r+1}\cdots \tilde v_1$ with $\tilde v_j\in W_{P_j}^{P_{j-1}}$ for each $1\leq j\leq r+1$.

   In order to show a), it remains to consider the case when $m=r+1$. Set $w:=us_\gamma$ and note that
     $\tilde {gr}(us_\gamma)=\sum_{j=1}^{\varsigma}\ell(v_j)\mathbf{e}_j+(\ell(v_r)+\ell(v_{r+1}))\mathbf{e}_{r}$.
      Thus we have
             $-D=(\ell(v_{r+1})-\ell(\tilde v_{r+1}))\mathbf{e}_{r+1}+(\ell(v_r)-\ell(\tilde v_r)-\ell(v_{r+1}v_r)+\ell(\tilde v_{r+1}\tilde v_r))\mathbf{e}_r
             = (\ell(v_{r+1})-\ell(\tilde v_{r+1}))(\mathbf{e}_{r+1}-\mathbf{e}_r)\leq \mathbf{e}_{r+1}-\mathbf{e}_r$.
 Note that $\gamma\in R\setminus R_P$ (by Lemma \ref{weylonelemma}).
 Therefore we have $gr(us_\gamma)-gr(u)=-D+ \big(\tilde {gr}(us_\gamma)-\tilde {gr}(u)\big)\leq \mathbf{e}_{r+1}-\mathbf{e}_r+
        \min\{\tilde {gr}(s_i)~|~ \langle \chi_i, \gamma^\vee\rangle\neq0\}= \mathbf{e}_{r+1}=\min\{gr(s_i)~|~ \langle \chi_i, \gamma^\vee\rangle\neq0\}$.
Thus     a) follows.

   To show b), we   set $w:=u$ in the rest of the proof and use induction on $\ell(s_\gamma)$.

 First we observe that  $gr(q_j)+gr(us_j)\leq  gr(u)+ gr(s_j)$ for any $1\leq j\leq n$.
  Indeed, this inequality holds if $1\leq j\leq r$ with the discussion in the beginning.
  If $\langle \alpha, \alpha_j^\vee\rangle=0$ for all $\alpha\in \Delta_P$, then for $\gamma=\alpha_j$ we have
    $m=r+1$, $\bar v_{r+1}\in  W^P$ (by Lemma \ref{gracompare11111}) and consequently
      $ gr(us_j)- gr(u)=-\mathbf{e}_{r+1}=-gr(q_j)+gr(s_j)$.
  Otherwise, we much have $\alpha_j=\beta_p$ with $p\in\{o, \kappa+2\}\cap \{1, \cdots, n\}$.
 Then the inequality still holds by using Lemma \ref{lemforgenone} a) (with $\gamma=\beta_p$) and   Lemma \ref{gengralemm222} c).

 Now we assume $\gamma\notin\Delta$.  Take any simple root $\alpha_j$ satisfying  $\langle \alpha_j, \gamma^\vee\rangle >0$,
  and write   $\beta=s_j(\gamma)$,
         $gr(q_{\beta^\vee})=(\lambda_1, \cdots, \lambda_{r+1}),
                         \min\{gr(s_i)~|~ \langle \chi_i, \beta^\vee\rangle\neq0\}=\mathbf{e}_c \mbox{ and }$
       \begin{align*} gr(q_j)+gr(us_j)&= gr(u)+ (a_1, \cdots, a_{r+1}),\\
                   gr(q_{\beta^\vee})+gr(us_js_\beta) &=  gr(us_j)+\mathbf{e}_c+(\mu_1,\cdots, \mu_{r+1}),\\
                   gr(us_js_\beta s_j)&= gr(us_js_\beta)+(b_1, \cdots, b_{r+1}).
                  \end{align*}
 In addition, we use the notations $\tilde c$, $\tilde a_j$'s, $\tilde b_j$'s and $\tilde \mu_j$'s, whenever
                replacing ``$gr$" with ``$\tilde {gr}$" in the above three equalities.
   Then we have    $\min\{gr(s_i) |  \langle \chi_i, \gamma^\vee\rangle\neq0\}\!=\min\{\mathbf{e}_c, gr(s_j)\}$ and
        by the induction hypothesis   $(\mu_1, \cdots, \mu_{r+1})\leq (0, \cdots, 0)$. Due to Lemma \ref{coroofsumgrad},
         it suffices to show   $\sum_{i=1}^r(a_i+b_i+\mu_i)\mathbf{e}_i\leq \mathbf{0}$.
   Furthermore, we note that $\tilde a_{r+1}=\tilde b_{r+1}=\tilde \mu_{r+1}=0, \mu_{r}+\mu_{r+1}=  \tilde \mu_r,
     \mathbf{e}_{\tilde c}\geq \mathbf{e}_c$ and $\tilde a_k=a_k, \tilde b_k=b_k, \tilde \mu_k=\mu_k$ for each $1\leq k\leq \varsigma$.
 Clearly, either of  the followings must hold.

  \begin{enumerate}
    \item[(i)] There is $\beta_p\in \Delta$ such that  $p\in \{1 \cdots, n\}\setminus\{o, \kappa+1, \kappa+2\}
              \mbox{ and } \langle \beta_p, \gamma^\vee\rangle >0;$
    \item[(ii)] whenever $\beta_p\in \Delta$ satisfies   $\langle \beta_p, \gamma^\vee\rangle >0$, we have $p\in \{o, \kappa+1, \kappa+2\}$. In this case, we note
               the constrain  $\ell(s_\gamma)=\langle 2\rho, \gamma^\vee\rangle-1$ on $\gamma$, which is deduced from
                    our assumption  by using Lemma \ref{lengthofpos222}.
  \end{enumerate}

  Suppose (i) holds.  Then we just take any one such $\alpha_j=\beta_p$.
  If $p\notin\{o, o+1, \cdots, \kappa+2\}$, then $\langle \alpha, \alpha_j^\vee\rangle =0$ for all $\alpha\in \Delta_P$;
   consequently, it is done
   by noting  $a_k=b_k=0$ for $1\leq k\leq r$ (using Lemma \ref{gracompare11111}) and $\mathbf{e}_c=\mathbf{e}_{r+1}$.
  Otherwise, there exists    $o+1\leq p\leq o+\varsigma$ such that  $\langle \beta_p, \gamma^\vee\rangle >0$. Recall that $\alpha_i=\beta_{o+i}$ for
    each $1\leq i\leq r$. For any one choice
   $\alpha_j=\beta_p$ among such roots, we always have   $a_{i}=b_{i}=0$ for $i\notin\{j-1, j\}$ and consequently
                          $\sum_{i=j+1}^{r}(a_i+b_i+\mu_i)\mathbf{e}_i=\sum_{i=j+1}^r\mu_i\mathbf{e}_i$.
        In addition from the proof of Proposition \ref{mainpropprod}, we can
                    always take  a certain  $\alpha_j=\beta_p$ among such roots such that
              both $\mathbf{e}_{\tilde c}<\mathbf{e}_j$ and     $\sum_{i=1}^{j}(\tilde a_i+\tilde b_i+\tilde \mu_i)\mathbf{e}_i\leq \mathbf{0}$ hold
               by considering  $\tilde{gr}$.
    Since $j\leq r-1$, we have $\tilde a_i=a_i$, $\tilde b_i=b_i$ and  $\tilde \mu_i=\mu_i$ for each $1\leq i\leq j$.
    Thus   for such a choice $\alpha_j=\beta_p$,
          both $\mathbf{e}_c\leq \mathbf{e}_{\tilde c}<\mathbf{e}_j$ and     $\sum_{i=1}^{j}(a_i+b_i+\mu_i)\mathbf{e}_i\leq \mathbf{0}$ hold.
       Hence, the Key Lemma holds for such $\gamma$ by using the induction hypothesis.

 Suppose (ii) holds.
  Then the constrains are so strong that there are only very few roots.  We  discuss all such roots with respect to each type of $\Delta$ and label the method we will use.

  Assume $\Delta$ is of $B$-type.  (That is, part of case $\mbox{C}1)$ in Table 1 occurs.) There are only two coroots satisfying the conditions, saying
        $\beta_{o-1}^\vee+2\sum_{i=o}^{n-1}\beta_i^\vee+\beta_n^\vee$ (with $o\geq 2$) or    $\sum_{i=o}^{n}\beta_i^\vee$. (See the proof of Proposition \ref{coeffvircoroot}.)

 \begin{enumerate}
   \item[(M1):] For the former coroot, we note that   $\Delta_{\bar P}=\{\alpha\in \Delta_P~|~ \langle \alpha, \gamma^\vee\rangle=0\}=\Delta$
                  and   $gr(q_{\gamma^\vee})=\sum_{i=1}^{r+1}d_i\mathbf{e}_i=d_{r+1}$ by direct calculations.
                    Hence,   $\sum_{i=1}^{r}d_i =0= -\ell(\omega_P\omega_{\bar P})$. Thus the inequality holds by using Lemma \ref{lemforgenone} a).

    \item[(M2):] For the latter coroot, we take $\alpha_j=\beta_n$; that is, $\alpha_j=\alpha_r$. Then $\beta^\vee=\gamma^\vee-\beta_n^\vee$ and
                         $gr(q_{\beta^\vee})=d_{r+1}\mathbf{e}_{r+1}-r\mathbf{e}_r+\varsigma\mathbf{e}_\varsigma$.
                     Write $gr(u)=(i_1, \cdots, i_{r+1}), gr(us_r)=(i_1', \cdots, i_{r+1}'), gr(us_rs_\beta)=(k_1, \cdots, k_{r+1})$ and
                      $gr(us_\gamma)=gr(us_rs_\beta s_r)=(k_1', \cdots, k_{r+1}')$.
                     Noting that  $\langle \alpha_t, \alpha_r^\vee\rangle=\langle \alpha_t, \beta^\vee\rangle=0$ for $1\leq t\leq \varsigma-1$,
               we conclude $a_t=b_t=\mu_t=0$ for $t\leq \varsigma-1$ by Lemma \ref{gracompare11111} and consequently
                                $\mu_{\varsigma}\leq 0$ by
                  the induction hypothesis. Furthermore, we note that
                  $a_\varsigma+a_{r}=1,$ $b_\varsigma+b_{r}=-1$,
   $a_{\varsigma}+b_\varsigma=-2\varsigma+i_\varsigma'-i_\varsigma+k_{\varsigma}'-k_{\varsigma}\leq 0$.
                     If $a_\varsigma+b_\varsigma+\mu_\varsigma<0$, then it is done. Otherwise, we conclude
                      $i_\varsigma'=k_{\varsigma}'=\varsigma$ and $i_\varsigma=k_{\varsigma}=\mu_{\varsigma}=0$.
     Consequently, we have $a_r+b_r+\mu_r=\mu_r=\mu_{\varsigma+1}\leq 0$ and  it is done.

\end{enumerate}

 Assume $\Delta$ is of $C$-type. (That is, part of case $\mbox{C}1)$ in Table 1 occurs.) There are only one such  coroots, saying
        $\gamma^\vee=\sum_{i=o}^{n}\beta_i^\vee$. Thus the inequality holds by  (M1).

 Assume $\Delta$ is of $D$-type.  (Case $\mbox{C}2)$ is used.)
  There are only two such coroots, saying
        $\beta_{o-1}^\vee+2\sum_{i=o}^{n-2}\beta_i^\vee+\beta_{n-1}^\vee+\beta_n^\vee$ (with $o\geq 2$) or    $\sum_{i=o}^{n-2}\beta_i^\vee+\beta_n^\vee$.
   For the former coroot,   the inequality holds by using (M1).

 \begin{enumerate}

    \item[(M3):] For the latter coroot, we have $gr(q_{\gamma^\vee})=2r\mathbf{e}_{r+1}+\varsigma\mathbf{e}_r-\varsigma\mathbf{e}_\varsigma$ by direct calculations.
            Using the notations of Lemma \ref{lemforgenone}, we conclude $\sum_{i=1}^rd_i=0$,
                  $\Delta_{\hat P}=\{\beta_o, \beta_{o+1}, \cdots, \beta_n\}$,
                    $\Xi_1   =\{\beta_n\}\cup \{\sum_{i=k}^{r-2}\beta_{o+i}+ \beta_n~|~  1\leq k\leq r-2 \}$
      and $\Xi_2:=   \{\beta_o+\sum_{i=o+k}^{n-2}\beta_i+\beta_{n-1}+\beta_n~|~ o+1\leq k\leq n-2 \}\cup \{\sum_{i=o}^n\beta_i\}$.
        Note that   $n=o+r$     in this case.
         Hence, we have $|\Xi_1|-|\Xi_2|=r-1-(r-1)=0= \sum_{i=1}^rd_i$. Hence, the Key Lemma holds for  $\gamma=\sum_{i=o}^{n-2}\beta_i^\vee+\beta_n^\vee$
          by using Lemma \ref{lemforgenone} b).
     \end{enumerate}

 It remains to discuss the cases when $\Delta$ is of either $E$-type or $F$-type.  Since there are only finite exceptional types (among which only a few roots satisfy
   (ii)) and the arguments are similar, we leave the details in the appendix (see section \ref{exceptional}).

  Hence, the statement follows.
\end{proof}

It remains to show the following, which was used in the proof of the Key Lemma.

\begin{lemma}\label{lemforgenone}
  Let    $u\in W$ and $\gamma\in R^+\setminus R_P$. 
        Write  $gr(q_{\gamma^\vee})=\sum_{j=1}^{r+1}d_j\mathbf{e}_j$.
         Then  Key Lemma {\upshape b)} holds, if either of the followings holds.
          \begin{enumerate}
            \item[a)]  $\sum_{j=1}^{r}d_j \leq -\ell(\omega_P\omega_{\bar P}) $, where   $\Delta_{\bar P}:=\{\alpha\in \Delta_P~|~ \langle \alpha, \gamma^\vee\rangle=0\}$.
            \item[b)] $\sum_{j=1}^{r}d_j\leq |\Xi_1|-|\Xi_2|$,  where  $\Xi_1:=\{\alpha\in R_P^+~|~   \langle \alpha, \gamma^\vee\rangle>0\}$ and
             $\Xi_2:=\{\alpha\in R_{\hat P}^+\setminus R_P~|~ \alpha-\gamma\in R^+, \langle \alpha, \gamma^\vee\rangle>0\}$
               with $\Delta_{\hat P}:=\Delta_P\cup\{\alpha_i\in \Delta~|~ \langle \chi_i, \gamma^\vee\rangle\neq 0\}$.

          \end{enumerate}

\end{lemma}

\begin{proof}  Let $u=v_{r+1}\cdots v_1$ (resp. $us_\gamma=\tilde v_{r+1}\cdots \tilde v_{1}$) be its decomposition associated to
      $(\Delta_P, \Upsilon)$. Since  $\gamma\in R^+\setminus R_P$,  we have
       $\min\{gr(s_i)~|~ \langle \chi_i, \gamma^\vee\rangle\neq0\}=\mathbf{e}_{r+1}$.
     Note that $gr_{\varsigma}(q_{\lambda  }w)= \tilde{gr}_{\varsigma}(q_\lambda w)$ for
      any $q_\lambda w$. Applying Proposition \ref{mainpropprod} with respect to $\Delta_\varsigma$, we have
       $\sum_{j=1}^{r-1}(d_j+\ell(\tilde v_j))\mathbf{e}_j\leq \sum_{j=1}^{r-1}\ell(v_j)\mathbf{e}_j$.
       If ``$<$" holds, it is already done.
       If ``$=$" holds,  we conclude  $\sum_{j=1}^{r-1}d_j+\ell(\tilde v_{r-1}\cdots \tilde  v_1)=\ell(v_{r-1}\cdots    v_1)$.
              Due to Lemma \ref{coroofsumgrad}, we
       have  $\sum_{j=1}^{r+1}(d_j+\ell(\tilde v_j))=\langle 2\rho, \gamma^\vee\rangle +\ell(us_\gamma)=1+\ell(u)= 1+\sum_{j=1}^{r+1}\ell(v_j)$. It
       remains to show         $d_r+\ell(\tilde v_r)\leq \ell(v_r)$, or equivalently to show
          $\ell(v_{r+1})\leq \ell(\tilde v_{r+1})+d_{r+1}-1=\ell(\tilde v_{r+1})+(\sum_{j=1}^{r+1}d_j-1)-\sum_{j=1}^{r}d_j=\ell(\tilde v_{r+1})+\ell(s_\gamma)
                 -\sum_{j=1}^{r}d_j$.

   a): Since $v_{r+1}, \tilde v_{r+1}\in W^P$, we
          conclude $\ell(v_r\cdots v_1)=|A_1|$ and
            $\ell(\tilde v_r\cdots \tilde  v_1)=|A_2|$ where
             $A_1:=\{\beta\in R_P^+~|~ u(\beta)\in-R^+\}$ and $A_2:= \{\beta\in R_P^+~|~ us_\gamma(\beta)\in-R^+\}$.
         Note that    $us_\gamma(\beta)=u(\beta)$ for $\beta\in R_{\bar P}$. Hence, $\beta\in A_2\setminus A_1$ only if $\beta \in
                   R^+_P\setminus R_{\bar P}$. Thus $|A_2|-|A_1|\leq |A_2\setminus A_1|\leq |  R^+_P\setminus R_{\bar P}|=\ell(\omega_P\omega_{\bar P})$.
       Consequently,  we have
               $\ell(\tilde v_r)-\ell(v_r)=\big(\ell(\tilde v_r\cdots \tilde  v_1)-\ell(v_r\cdots v_1)\big)+\sum_{j=1}^{r-1}d_j
                     =|A_2|-|A_1|+\sum_{j=1}^{r-1}d_j\leq \ell(\omega_P\omega_{\bar P})+\sum_{j=1}^{r-1}d_j\leq -d_r$.

   b): Since $v_{r+1}\in W^P$ and $\Delta_{\hat P}\supset \Delta_P$,  we
          can write $v_{r+1}=v_{r+2}'v_{r+1}'$ in which $v_{r+2}'\in W^{\hat P}$ and $v_{r+1}'\in W_{\hat P}^P\subset W^P$.
          Similarly, we write  $\tilde v_{r+1}=\tilde v_{r+2}'\tilde v_{r+1}'$ with $\tilde v_{r+2}'\in W^{\hat P}$ and $\tilde v_{r+1}'\in W_{\hat P}^P$.
      Note that
            $ \ell(v_{r+1}')=|A_3|$ and $\ell(\tilde v_{r+1}')=|A_4|$, where
             $A_3:=\{\alpha\in R_{\hat P}^+\setminus R_P$ $~|~ u(\alpha)\in-R^+\}$ and
             $ A_4:=\{\alpha\in R_{\hat P}^+\setminus R_P~|~ us_\gamma(\alpha)\in-R^+\}$.
          We claim $A_3$ can be written as a disjoint union $B_1\sqcup B_2\sqcup B_3$ such that $|B_1|\leq \ell(s_\gamma)-|\Xi_1|$,
            $B_2\subset \Xi_2$ and $B_3\subset A_4$.
          Hence, $\ell(v_{r+1}')=|A_3|=|B_1|+|B_2|+|B_3|\leq \ell(s_\gamma)-|\Xi_1|+|\Xi_2|+\ell(\tilde v_{r+1}')$.
         Since   $\gamma\in R_{\hat P}$, we have  $\tilde v_{r+2}'=v_{r+2}'$. Therefore
            $\ell(v_{r+1})-\ell(\tilde v_{r+1})=\ell(v_{r+1}')-\ell(\tilde v_{r+1}')  \leq \ell(s_\gamma)-|\Xi_1|+|\Xi_2|\leq
                     \ell(s_\gamma)- \sum_{j=1}^{r}d_j$.

          It remains to show our claim. Clearly, $A_3$ is a  disjoint union of $B_i$'s in which the corresponding  sets are given by
         $B_1:=
                  \big\{\alpha\in R_{\hat P}^+\setminus R_P~\big|~  u(\alpha)\in-R^+,  s_\gamma(\alpha)\in -R^+  \big\}$,
          $$B_2:=  \left\{\alpha\in R_{\hat P}^+\setminus R_P \Bigg|   \begin{array}{l} u(\alpha)\in-R^+,\\ s_\gamma(\alpha)\in R^+,\\ \langle \alpha, \gamma^\vee\rangle>0\end{array}\right\},
              B_3:=  \left\{\alpha\in R_{\hat P}^+\setminus R_P \Bigg|   \begin{array}{l} u(\alpha)\in-R^+,\\ s_\gamma(\alpha)\in R^+,\\ \langle \alpha, \gamma^\vee\rangle\leq 0\end{array}\right\}.
       $$
 Note  that $B_1\subset \big\{\alpha\in R_{\hat P}^+\setminus R_P~\big|~    s_\gamma(\alpha)\in -R^+  \big\}=
                   \big\{\alpha\in R_{\hat P}^+ ~\big|~    s_\gamma(\alpha)\in -R^+  \big\}
                   -  \big\{\alpha\in R_{P}^+ ~\big|~    s_\gamma(\alpha)\in -R^+  \big\}$.
    Since $\gamma\not\in R_P$, for any $\alpha\in R_P^+$ we conclude that  $s_\gamma(\alpha)=\alpha-\langle \alpha, \gamma^\vee \rangle\gamma\in -R^+$
              if and only if $\langle \alpha, \gamma^\vee\rangle >0$.
    Hence, $|B_1|\leq  \big|  \{\alpha\in R_{\hat P}^+ ~|~    s_\gamma(\alpha)\in -R^+   \} \big|-
                    \big| \{\alpha\in R_{P}^+ ~|~  s_\gamma(\alpha)\in -R^+   \} \big|=\ell(s_\gamma)- |\Xi_1|$.
        It is obvious that $B_2\subset \Xi_2$.
   Since $\ell(us_\gamma)<\ell(u)$, we have $u(\gamma)\in -R^+$ by Lemma \ref{strongexchange}. Consequently, if
      $\alpha\in B_3$ then we have $us_\gamma(\alpha)=u(\alpha)+(-\langle \alpha, \gamma^\vee\rangle)u(\gamma)\in -R^+$, implying
         $\alpha\in A_4$. Hence, $B_3\subset A_4$.
\end{proof}

\section{Proofs of main results}\label{secproof22}

In this section, we prove all the theorems mentioned in the introduction.
Recall that the proof of Theorem \ref{mainthm} has been given in section \ref{subsecmainthm}.

When $\Delta_P$ is not of $A$-type, we have denoted $\varsigma=r-1$.
For convenience, we denote $\varsigma=r$ if $\Delta_P$ is of $A$-type. Recall that $\Delta_\varsigma=\{\alpha_1, \cdots, \alpha_\varsigma\}$, $P_\varsigma=P_{\Delta_\varsigma}$
 and  $Q^\vee_\varsigma=\bigoplus_{i=1}^\varsigma\mathbb{Z}\alpha_i^\vee$. (In particular when $\varsigma=r$, we have  $P_\varsigma=P$
    and $Q^\vee_\varsigma=Q^\vee_P$.)

\begin{lemma}\label{uniqAtypegrad}
    $gr(W_{P_\varsigma}\times Q^\vee_\varsigma)=\bigoplus_{i=1}^\varsigma \mathbb{Z}\mathbf{e}_i$,
     where we have treated  $W_{P_\varsigma}\times Q^\vee_\varsigma$ as a subset of $W\times Q^\vee$ naturally.
     Furthermore for any $\mathbf{d}=\bigoplus_{i=1}^\varsigma d_i\mathbf{e}_i$,
          we have
  \begin{enumerate}
    \item $\mathbf{d}=gr(wq_\lambda)$ for a unique $wq_\lambda\in W_{}\times Q^\vee$. In fact, $wq_\lambda\in W_{P_\varsigma}\times Q^\vee_\varsigma$.
    \item Take the unique $wq_\lambda$ as  in (1). Then  $wq_\lambda\in QH^*(G/B)$ if $d_i\geq 0$ for all $i.$
  \end{enumerate}
\end{lemma}

\begin{proof}
    Define a  matrix $M=\big(m_{i, j}\big)_{\varsigma\times \varsigma}$ by using the gradings $gr(q_i)$'s. That is,
      we define $\sum_{j=1}^\varsigma m_{i, j}\mathbf{e}_j=(1-i)\mathbf{e}_{i-1}+(1+i)\mathbf{e}_i (=gr(q_i))$ for each $1\leq i\leq \varsigma$.
 Note that $M$ is a lower-triangular matrix.
   Hence, there exist unique sequences $\mathbf{a}=(a_1, \cdots, a_{\varsigma}),
  \mathbf{b}=(b_1, \cdots, b_{\varsigma})$ of  integers  such that $\mathbf{d}=\mathbf{a}M+\mathbf{b}$ and
    $0\leq b_{i}\leq m_{i, i}-1=i$  for $1\leq i\leq \varsigma$. Furthermore if $d_i\geq 0$ for all $i$, then we conclude
      $a_i\geq 0$ for all $i$, by noting  $m_{i, j}\leq 0$ whenever $j<i$.
    Since $W_{P_i}^{P_{i-1}}=\{u_k^{(i)}~|~ 0\leq k\leq i\}$, each $0\leq b_i\leq i$ corresponds to a unique
      element in $W_{P_i}^{P_{i-1}}$ say  $u_{b_i}^{(i)}$.
      Hence, we find a unique $(w, \lambda):=(u_{b_\varsigma}^{(\varsigma)}\cdots  u_{b_1}^{(1)}, \sum_{i=1}^\varsigma a_i\alpha_i^\vee)\in
      W_{P_\varsigma}\times Q^\vee_\varsigma$  such that $gr(wq_\lambda)=\mathbf{d}$; furthermore,
        $wq_\lambda\in QH^*(G/B)$ whenever $d_i\geq 0$ for all $i$.

     It remains to show $gr(uq_\mu)\not \in\bigoplus_{i=1}^\varsigma \mathbb{Z}\mathbf{e}_i$ whenever $(u, \mu)\not\in   W_{P_\varsigma}\times Q^\vee_\varsigma$.
    Indeed, it follows directly from Definition \ref{defgrading} that
           $gr_{[k, r+1]}(q_{\alpha^\vee})=x_k\mathbf{e}_k$ with $x_k\geq 2$, whenever $\alpha\in \Delta_k\setminus\Delta_{k-1}$.
           Take the decomposition $u=v_{r+1}\cdots v_1$ of $u$ associated to $(\Delta_P, \Upsilon)$ and note that  $gr(uq_\mu)=
          \bigoplus_{i=1}^{r+1}\ell(v_i)\mathbf{e}_i+gr(q_\mu)$.  Thus
          if  $gr(uq_\mu)\in\bigoplus_{i=1}^\varsigma \mathbb{Z}\mathbf{e}_i$,
          then we have $\ell(v_{r+1})=0$ and $\mu\in Q_P^\vee(=Q_{r}^\vee)$. When $\varsigma=r$, it is done. When $\varsigma=r-1$, we proceed to conclude
                       $\ell(v_r)=0$ and $\mu\notin Q^\vee_r\setminus Q_\varsigma^\vee$.  Thus $u\in W_{P_\varsigma}$ and $\mu\in Q_{\varsigma}^\vee$.
\end{proof}

\bigskip

\begin{proof}[Proof of Lemma \ref{lemmasemigp}] We need to show for any $q_\mu u, q_\nu v\in QH^*(G/B)$ there exists
             $q_\lambda w\in QH^*(G/B)$ such that $gr(q_\mu u)+gr(q_\nu v)=gr(q_\lambda w)$.  Note that 
     $gr(q_\mu)+gr(q_\nu)=gr(q_{\mu+\nu})$ and $gr(q_\mu)+gr(u)=gr(q_\mu u)$. Thus
    it remains  to show $gr(u)+gr(v)\in S$. It suffices to show $\mathbf{x}=(x_1, \cdots, x_{r+1})\in S$ for any $\mathbf{x}\in (\mathbb{Z}_{\geq 0})^{r+1}$.

  We first assume $\varsigma=r$.  Take any simple root in $\Delta_{\varsigma+1}\setminus\Delta_\varsigma(=\Delta\setminus \Delta_P)$, saying $\alpha$.
    From Table \ref{tabgrading}, we conclude $gr(q_{\alpha^\vee})=\sum_{i=1}^{\varsigma+1}d_i\mathbf{e}_i$ with $2\leq d_{\varsigma+1}\leq
           1+ \ell(\omega_{P_{\varsigma+1}}\omega_{P_\varsigma})(=\ell(\omega\omega_P)+1)$ and $d_i\leq 0$ for $i\leq \varsigma$. Since
         $x_{\varsigma+1}\geq 0$, we can write  $x_{\varsigma+1}=a_{\varsigma+1} d_{\varsigma+1}+b_{\varsigma+1}$ for unique
         $a_{\varsigma+1}\geq 0$ and $ 0\leq b_{\varsigma+1}\leq \ell(\omega_{P_{\varsigma+1}}\omega_{P_\varsigma})$. Note that
             $ \ell(\omega_{P_{\varsigma+1}}\omega_{P_\varsigma})=\max\{\ell(v)~|~ v\in W_{P_{\varsigma+1}}^{P_\varsigma}\}$
              so that we can choose $v_{\varsigma+1}\in W_{P_{\varsigma+1}}^{P_\varsigma}$ satisfying $\ell(v_{\varsigma+1})=b_{\varsigma+1}$.
         Furthermore,   $(x_1-a_{\varsigma+1}d_1, \cdots, x_\varsigma-a_{\varsigma+1}d_\varsigma)$ is again a sequence of
          non-negative integers. Thus it is the grading of a unique $(w', \lambda')\in W_{P_{\varsigma}}\times Q^\vee_{\varsigma}$ with
             $q_{\lambda'}\in QH^*(G/B)$  by Lemma \ref{uniqAtypegrad}. Set $w=v_{\varsigma+1}w'$ and $\lambda=a_{\varsigma+1}\alpha^\vee+\lambda'$. Then
             $wq_\lambda\in QH^*(G/B)$ is the element as required.

    Now we assume $\varsigma=r-1$. By Lemma \ref{gengralemm222} there exists $\alpha'\in \Delta\setminus\Delta_P$ such that
       $gr(q_{\alpha'^\vee})=\sum_{i=1}^{r+1}d_i'\mathbf{e}_i$ with $2\leq d_{r+1}'\leq
           1+ \ell(\omega\omega_P)$ and $d_i'\leq 0$ for $i\leq r=\varsigma+1$.
    Repeating the above discussions, we can reduce it to the question of  finding
            a element in $q_{\lambda}w\in W_{P_{\varsigma+1}}\times Q^\vee_{\varsigma+1}$ with
                $q_\lambda w\in  QH^*(G/B)$ and the grading of it being equal to   $\sum_{i=1}^{\varsigma+1}x_i'\mathbf{e}_i$ for
                   given non-negative integers $x_i'$'s.
       Thus the statement follows by using the same arguments again.
\end{proof}

\begin{remark}
      $\mathbb{Z}_{\geq 0}\mathbf{e}_{r+1}$ is a sub-semigroup of $S$. Indeed, we can take $\alpha\in \Delta$ such that $gr_{[r+1, r+1]}(q_{\alpha^\vee})=d_{r+1}\mathbf{e}_{r+1}$ with
      $2\leq d_{r+1}\leq 1+\ell(\omega \omega_P)$, from the above proof. For any $c\in\mathbb{Z}_{\geq 0}$, $c=ad_{r+1}+b$ with $0\leq b\leq d_{r+1}-1$.
      Then we can choose $v\in W^P$ such that $\ell(v)=b$. Note that $-gr_r(q_{a\alpha^\vee})=\sum_{i=1}^rx_i\mathbf{e}_i$ with $x_i$'s being in $\mathbb{Z}_{\geq0}$.
      Hence, it is a grading of certain element $q_\lambda u \in QH^*(G/B)$ where $(u, \lambda)\in W_P\times Q_P^\vee$. Then $q_{a\alpha^\vee+\lambda} vu\in QH^*(G/B)$ and
       its grading is equal to $c\mathbf{e}_{r+1}$.
\end{remark}

 The  next lemma proves the first half of   Theorem \ref{inclusisom}.

 \begin{lemma}\label{lemforinclus}
    The subspace $\mathcal{I}$ defined in Theorem \ref{inclusisom} is an ideal of $QH^*(G/B)$.
 \end{lemma}

 \begin{proof} We need to show for any $q_\mu u\in \mathcal{I}$ and $q_\nu v\in QH^*(G/B)$, the   product
                     $q_\mu u\star q_\nu v=\sum_{w, \lambda}N_{u, v}^{w, \lambda}q_{\lambda+\mu+\nu}w$ also lies in $\mathcal{I}$.
                     That is, we need to show $d_{r+1}\geq 1$ where
              $gr_{[r+1, r+1]}( q_{\lambda+\mu+\nu}w)=d_{r+1} \mathbf{e}_{r+1}$, whenever $N_{u, v}^{w, \lambda}\neq 0$.
              Clearly, this is true if either $\mu$ or $\nu$ lies in $Q^\vee\setminus Q^\vee_P$, which follows directly from Definition \ref{defgrading}.
             When $\mu, \nu\in Q^\vee_P$, we must have $u\in W\setminus W_P$. Then we shall $gr_{[r+1, r+1]}( q_{\lambda}w)\geq  \mathbf{e}_{r+1}$ whenever
              $N_{u, v}^{w, \lambda}\neq 0$, by using induction on $\ell(v)$.

        If $\ell(v)=0$, then $v=\mbox{id}$ and it is done. If $\ell(v)=1$, then $v$ is a simple reflection and therefore
            we can use quantum Chevalley formula (Proposition \ref{quanchevalley}).
         When $\ell(us_\gamma)=\ell(u)+1$, we take the decomposition $us_\gamma=v_{r+1}\cdots v_1$ associated to $(\Delta_P, \Upsilon)$,
          and note $u$ is obtained by deleting  a reflection in some $v_m$. Since $u\notin W_P$, we conclude $v_{r+1}\neq 1$. In particular,
          we have  $gr_{[r+1, r+1]}(us_\gamma)\geq  \mathbf{e}_{r+1}$. When  $\ell(us_\gamma)=\ell(u)+1-\langle 2\rho, \gamma^\vee\rangle$, then we also conclude
            $gr_{[r+1, r+1]}(q_{\gamma^\vee} us_\gamma) \geq  \mathbf{e}_{r+1}$, by noting
            $gr_{[r+1, r+1]}(q_{\gamma^\vee}us_\gamma)\geq gr_{[r+1, r+1]}(us_\gamma)=gr_{[r+1, r+1]}(u)  \geq  \mathbf{e}_{r+1}$ if $\gamma\in Q^\vee_P$, and
            $gr_{[r+1, r+1]}(q_{\gamma^\vee}us_\gamma)\geq gr_{[r+1, r+1]}(q_{\gamma^\vee}) \geq  2\mathbf{e}_{r+1}$ if $\gamma\notin Q^\vee_P$.
     Thus  if $\ell(v)=1$,  then $\sigma^v\star \mathcal{I}\subset \mathcal{I}$. Now we assume $\ell(v)>1$.   By Lemma \ref{lemmaformainthm}, there exist $v'\in W$ and  $1\leq j\leq n$ such that
                     $gr(v)=gr(v')+gr(s_j)$ and $\sigma^{v'}\star \sigma^{s_j}=c \sigma^v+\sum_{ {w, \lambda}}c_{w, \lambda} q_{\lambda}w$,
         where $c>0$ and the summation is  only over  finitely many non-zero terms  for which  $c_{w, \lambda}>0$.
       In particular, we have $\ell(v')=\ell(v)-1$. Using the induction hypothesis, we have $\sigma^{v'}\star \mathcal{I}\subset \mathcal{I}$.
        Thus
   $(c \sigma^v+\sum_{ {w, \lambda}}c_{w, \lambda} q_{\lambda}w)  \star \mathcal{I} =\sigma^{s_j}\star(\sigma^{v'}
   \star \mathcal{I})\subset \mathcal{I}$.
 Since   all the  structure constants are non-negative, there is no cancellation in the summation on the left hand side of the equality.
       Hence, we conclude $\sigma^v   \star \mathcal{I}\subset \mathcal{I}$. %
 \end{proof}

It remains to show the second half of Theorem \ref{inclusisom}. There are combinatorial characterizations of
   $QH^*(G/B)$ (see e.g. \cite{mare00}), or more generally on its torus-equivariant extension \cite{mih}.  In particular, intuitively, $QH^*(G/B)$ should
     also have a non-equivariant version of Mihalcea's criterion
  \cite{mih} for torus-equivariant quantum cohomology of $G/B$. That is, an algebra $(\bigoplus_{w\in W}\mathbb{Q}[\mathbf{q}]\sigma^w, *)$
         should  be isomorphic to $QH^*(G/B)$ as algebras, if it satisfies quantum Chevalley formula together with some natural properties (e.g. commutativity and
            associativity). However, we did not find any explicit reference for this. In our case, we obtain a natural  algebra of this form which
              has one more (strong) property saying that  $(\bigoplus_{w\in W}\mathbb{Q}[\mathbf{q}]\sigma^w, *)|_{\mathbf{q}=\mathbf{0}}$ is
                canonically isomorphic to $H^*(G/B)$. Thus it becomes easy to show the algebra isomorphism (by using induction).
                We would like to thank
                  A.-L. Mare and  L. C. Mihalcea for their comments for such a criterion and the proof.

\begin{proof}[Proof of Theorem \ref{inclusisom}]
  Due to Lemma \ref{lemforinclus}, it remains to show $QH^*(P/B)$ is canonically isomorphic to $QH^*(G/B)/\mathcal{I}$.
  Note that $P/B$ is isomorphic to the complete flag variety determined by the pair $(\Delta_P, \emptyset)$.
  Hence,  $QH^*(P/B)$ has a natural basis of Schubert classes $\{\sigma^w~|~ w\in W_P\}$ over $\mathbb{Q}[q_1, \cdots, q_r]$,
   and the formula of  $\sigma^{u}\star_f \sigma^{s_i}$ (where $u\in W_P$ and $\alpha_i\in \Delta_P$) is directly obtained from
     Proposition \ref{quanchevalley} by restriction of $\gamma\in \Delta$ to $\gamma\in \Delta_P$ in the summation.
     Here we denote the quantum product of  $QH^*(P/B)$ by $\star_f$, in order to distinguish it with the quantum product $\star$ of
         $QH^*(G/B)$.
      On the other hand,  $QH^*(G/B)/\mathcal{I}$ has a natural algebra structure induced from  $QH^*(G/B)$. Thus it is also
       commutative and associative, and we denote the product of it by the same  $\star$  by abuse of notations.

     It is clear that for any $wq_\lambda \in W\times Q^\vee$, $gr_{[r+1, r+1]}(wq_\gamma)=\mathbf{0}$, i.e. $wq_\lambda\not\in \mathcal{I}$,
       if and only if $wq_\lambda \in W_P\times Q^\vee_P$.
      We define a map  $\varphi: QH^*(G/B)\longrightarrow QH^*(P/B)$,  given  by    $\varphi(q_\lambda)=q_\lambda$ if $wq_\lambda\not\in \mathcal{I}$, or
                     $0$ if $wq_\lambda\in \mathcal{I}$.
            Clearly,  $\varphi$ induces
                    a natural isomorphism $\bar\varphi$ of vector spaces,
          $\bar\varphi:  QH^*(G/B)/\mathcal{I}\longrightarrow QH^*(P/B)$, given by $\bar\varphi(\overline{wq_\lambda}):=\varphi(wq_\lambda).$
       In particular, it is easy to check
          $\bar \varphi(\overline{\sigma^{s_i}}\star\overline{\sigma^{s_j}})= \sigma^{s_i}\star_f \sigma^{s_j}$
                for any $\alpha_i, \alpha_j\in \Delta_P$.
                It is a
         well-known fact that $QH^*(P/B)$ is generated by $\{\sigma^{s_\alpha}~|~\alpha\in \Delta_P\}$  over $\mathbb{Q}[q_1, \cdots, q_r]$.
         Thus it is
          sufficient  to show  $QH^*(G/B)/\mathcal{I}$ is generated by  $\{\overline{\sigma^{s_\alpha}}~|~\alpha\in \Delta_P\}$.
         Since our filtration on $QH^*(G/B)$ generalizes the classical filtration on $H^*(G/B)$ (by Proposition  \ref{qredclas}) naturally,
          $QH^*(G/B)/\mathcal{I}\big|_{\overline{q_\lambda}=0}$ is canonically isomorphic to $H^*(P/B)$ as  algebras. In particular,
           it is generated by $\{\overline{\sigma^{s_1}}, \cdots, \overline{\sigma^{s_r}}\}$ with respect to the induced cup product.
            Hence, the statement follows by using quantum Chevalley formula and induction (for instance
               one can follow the proof of Lemma 2.1 of \cite{sietian} exactly).
   \end{proof}

  \begin{remark}
     For the  classical case, the induced  map $i^*: H^*(G/B)\longrightarrow H^*(P/B)$ is given by
           $i^*(\sigma^w)=\sigma^w$ if $w\in W_P$, or $0$ otherwise. And the ideal $I$ is given by
             $I=\mathbb{Q}\{\sigma^w~|~ w=vu \mbox{ with } u\in W_P, v\in W^P, v\neq 1\}$. Note that  for any $w\in W$, $w\notin I$ if and only if
               $gr_{[r+1, r+1]}(\sigma^w)=\mathbf{0}$. Clearly,
              $\mathcal{I}$ is a $\mathbf{q}$-deformation of $I$   and $\varphi$ is a natural generalization of $i^*$.

  \end{remark}

\begin{lemma}\label{achieve}
  Let $w=v_{r+1}\cdots v_1$ be the decomposition of $w\in W$ associated to $(\Delta_P, \Upsilon)$. For any $1\leq m\leq \varsigma$, the followings hold.
  \begin{enumerate}
    \item If  $\ell(v_m)<m$, then there exists $\gamma\in R^+$ such that
                 $\langle \chi_m, \gamma^\vee\rangle=1$,
               $\ell(w s_\gamma)=\ell(w)+1$ and  $gr(w s_\gamma)=gr(w)+\mathbf{e}_m$.
    \item If  $\ell(v_m)=m$, then there exists $\gamma\in R^+$ such that
                 $\langle \chi_m, \gamma^\vee\rangle=1$,
               $\ell(w s_\gamma)=\ell(w)+1-\langle 2\rho, \gamma^\vee\rangle$ and  $gr(q_{\gamma^\vee}w s_\gamma)=gr(w)+\mathbf{e}_m$.

  \end{enumerate}
\end{lemma}

 \begin{proof}
   Note that $Dyn(\{\alpha_1, \cdots, \alpha_\varsigma\})$ is of $A$-type. We have
     $v_k=u_{i_k}^{(k)}$ with $i_k=\ell(v_k)$ whenever $1\leq k\leq \varsigma$.

   (1) If $i_m<m$, we set $\gamma:=(v_{m-1}\cdots v_1)^{-1}(\alpha_{m-{i_m}}+\alpha_{m-i_m+1}+\cdots+\alpha_m)$.
        Then $\gamma$ is of the form $\alpha_m+\sum_{j=1}^{m-1}a_j\alpha_j\in R$. Thus $\gamma\in R^+$ and   $\langle \chi_m, \gamma^\vee\rangle=1$.
         Furthermore, we conclude
    $ws_{\gamma}=v_{r+1}\cdots v_{m+1}u_{i_m+1}^{(m)}v_{m-1}\cdots v_1$. Thus $(1)$ follows.

   (2) Denote    $k=1+\max\{j~|~ i_j=0, 0\leq j\leq m-1\}$, where $i_0:=0$, and set $\gamma=\alpha_k+\alpha_{k+1}+\cdots+\alpha_m\in R^+$.
       Clearly,  $\langle \chi_m, \gamma^\vee\rangle=1$.    For each $j\leq m$, we denote $\gamma_j=\alpha_j+\alpha_{j+1}+\cdots+\alpha_{m}$.
        Then   $\gamma_{j}=\gamma_{j+1}+\alpha_j=s_j(\gamma_{j+1})$. Thus for any $i_j\geq 1$
        we have  $u_{i_j}^{(j)}s_{\gamma_j}=u_{i_j}^{(j)}s_js_{\gamma_{j+1}}s_j=u_{i_j-1}^{(j-1)}s_{\gamma_{j+1}}s_j=s_{\gamma_{j+1}}u_{i_j-1}^{(j-1)}s_j=
            s_{\gamma_{j+1}}u_{i_j}^{(j)}$.
         Furthermore, we have
           $s_1s_2\cdots s_j u_{i_j}^{(j)}=u_{i_j-1}^{(j)}s_1s_2\cdots s_{j-1}$  by Lemma \ref{prodAtype}.
         Note that $\gamma=\gamma_k, \gamma_m=\alpha_m$ and denote $u=v_m\cdots v_1$. Hence,
         \begin{align*} u s_\gamma =u_{i_m}^{(m)}\cdots u_{i_k}^{(k)}s_{\gamma_k} u_{i_{k-2}}^{(k-2)}\cdots u_{i_1}^{(1)}
                              &= u_{i_m}^{(m)}\cdots u_{i_{k+1}}^{(k+1)} s_{\gamma_{k+1}} u_{i_k}^{(k)}s_\gamma u_{i_{k-2}}^{(k-2)}\cdots u_{i_1}^{(1)}\\
                                  &=u_{i_m}^{(m)}s_{\gamma_m}u_{i_{m-1}}^{(m-1)}\cdots u_{i_k}^{(k)} u_{i_{k-2}}^{(k-2)}\cdots u_{i_1}^{(1)}\\
                                  &=s_1s_2\cdots s_{m-1}u_{i_{m-1}}^{(m-1)}\cdots u_{i_k}^{(k)} u_{i_{k-2}}^{(k-2)}\cdots u_{i_1}^{(1)}\\
                                 &=u_{i_{m-1}-1}^{(m-1)}\cdots u_{i_k-1}^{(k)}(s_1\cdots s_{k-1}) u_{i_{k-2}}^{(k-2)}\cdots u_{i_1}^{(1)}.
         \end{align*}
        Note that $i_m=m$ and $i_j=\ell(v_j)$ for $j\leq \varsigma$. Thus
            \begin{align*} \ell(w s_\gamma)& =\sum\nolimits_{p=m+1}^{r+1}\ell(v_p)+\big(\sum\nolimits_{j=k}^{m-1}(i_j-1)\big)+k-1+ \sum\nolimits_{j=1}^{k-2}i_j \\
                                           &= \sum\nolimits_{p=1}^{r+1}\ell(v_p) -(m-k)+k-1-m\\
                                           &=\ell(w)+1-2\langle \rho, (\alpha_k+\cdots+\alpha_m)^\vee\rangle
                                            =\ell(w)+1-\langle 2\rho, \gamma^\vee\rangle.
            \end{align*}
     Furthermore, we conclude  $gr(q_{\gamma^\vee}w s_\gamma)=gr(w)+\mathbf{e}_m$, by noting that $ i_m=m$, $i_{k-1}=0$ and
        $gr(q_{\gamma^\vee})=(1-k)\mathbf{e}_{k-1}+(m+1)\mathbf{e}_m+\sum_{j=k}^{m-1}\mathbf{e}_j$.
     \end{proof}

\bigskip

 Since $QH^*(G/B)$ has  an $S$-filtration  $\mathcal{F}$, we obtain an associated  $S$-graded algebra
      $Gr^{\mathcal{F}}(QH^*(G/B))=\bigoplus_{\mathbf{a}\in S} Gr_\mathbf{a}^{\mathcal{F}},\mbox{ where }
                Gr_{\mathbf{a}}^{\mathcal{F}}:=F_{\mathbf{a}}\big/\cup_{\mathbf{b}<\mathbf{a}}F_{\mathbf{b}}.$
For each $j\leq r+1$, we denote
    $Gr_{(j)}^{\mathcal{F}}(QH^*(G/B)):=\bigoplus_{i\geq 0}Gr_{i\mathbf{e}_j}^{\mathcal{F}}.$ Note that
for the iterated fibration $\{P_{j-1}/P_0\rightarrow P_j/P_0\longrightarrow P_j/P_{j-1}\}_{j=2}^{r+1}$ associated to
     $(\Delta_P, \Upsilon)$, we have $P_{r+1}/P_{r}=G/P$ and $P_j/P_{j-1}\cong \mathbb{P}^j$ whenever $j\leq \varsigma$.
 Take
    the canonical isomorphism  $QH^*(\mathbb{P}^k)\cong {\mathbb{Q}[x_k, t_k]\over \langle x_k^{k+1}-t_k\rangle }$  for each $k\leq \varsigma$. Then we can state Theorem
     \ref{isomfordirection} more concretely  as follows (in which we denote $u^{(0)}_0:= 1$).

\vspace{0.25cm}

\noindent\textbf{Theorem \ref{isomfordirection}.} {\itshape
   There exist   canonical
    isomorphisms $\Psi_k$'s of algebras as follows.
        \begin{enumerate}
    \item[] \hspace{-1.0cm}{\upshape}  For $k\leq \varsigma$, \,\,$\,\Psi_k: QH^*(\mathbb{P}^k) \longrightarrow  Gr_{(k)}^{\mathcal{F}}(QH^*(G/B));
              \quad x_k\mapsto            \overline{u^{(k)}_1},\,\,  t_k\mapsto \overline{q_k u_{k-1}^{(k-1)}}  $.
     \item[]\hspace{-1.0cm} 
          $\Psi_{\varsigma+1}: QH^*(P_{\varsigma+1}/P_\varsigma) \longrightarrow  Gr_{(\varsigma+1)}^{\mathcal{F}}(QH^*(G/B));
                  \,\,     q_{\lambda_{P_\varsigma}}\sigma^v\mapsto           \overline{\psi_{\Delta_{\varsigma+1}, \Delta_{\varsigma}}(q_{\lambda_{P_\varsigma}}\sigma^v)}.$
   \end{enumerate}
 }
 In particular if $\Delta_P$ is of $A$-type (i.e. if $\varsigma=r$), then we have $P_{\varsigma+1}/P_\varsigma=G/P$, $\Delta_{r+1}=\Delta$ and $\Delta_\varsigma=\Delta_P$.
 Thus in this case, Theorem \ref{isomfordirection} gives  an   isomorphism
                $QH^*(G/P) \overset{\simeq}{\longrightarrow}  Gr_{(r+1)}^{\mathcal{F}}(QH^*(G/B))$.
  \begin{proof}[Proof of Theorem \ref{isomfordirection}] By Lemma \ref{uniqAtypegrad}, for any $\mathbf{a}\in \bigoplus_{i=1}^\varsigma \mathbb{Z}\mathbf{e}_i$
             there exists a unique  $q_\lambda u\in QH^*(G/B)$ such that $gr(q_\lambda u)=\mathbf{a}$.  
             Thus
                $\dim_\mathbb{Q} Gr^{\mathcal{F}}_{\mathbf{a}}=1$ and $Gr^{\mathcal{F}}_\mathbf{a}=\mathbb{Q}\overline{q_\lambda u}$. Then
                      we   simply denote   $A_\mathbf{a}:=\overline{q_\lambda u}$.
              In particular, we conclude
                 $A_{\mathbf{a}}\star A_{\mathbf{e}_j}=c_jA_{\mathbf{a}+\mathbf{e}_j}$ whenever  $j\leq  \varsigma$. Furthermore, we have
                  $c_j=1$ by Lemma \ref{achieve}.
      When $\ell(v)>1$, there  exists    $v'\in W_{P_\varsigma}$ satisfying $gr(v')+gr(s_p)=gr(v)$ with $p\leq \varsigma$ by using Lemma \ref{achieve} again.
 Thus by induction on $\ell(v)$, we conclude  $A_{\mathbf{a}}\star A_{gr(v)}=A_{\mathbf{a}+gr(v)}$ for any $v\in W_{P_\varsigma}$.  Hence,
    $A_{\mathbf{a}}\star A_{\mathbf{b}}=A_{\mathbf{a}+\mathbf{b}}$ for any ${\mathbf{a}}, {\mathbf{b}}\in \bigoplus_{i=1}^\varsigma \mathbb{Z}\mathbf{e}_i$.
    As a consequence, we obtain a  canonical isomorphism   $QH^*(\mathbb{P}^k)\cong Gr^{\mathcal{F}}_{(k)}(QH^*(G/B))$ for each $1\leq k\leq \varsigma$,
     given by
             $x_k\mapsto\overline{u^{(k)}_1} $ and $t_k\mapsto\overline{q_k u_{k-1}^{(k-1)}}$. 

  To analyze $\Psi_{\varsigma+1}$, we need to compare the algebra structure of $QH^*(P_{\varsigma+1}/P_{\varsigma})$ with
               the   filtered-algebra structure of $QH^*(G/B)$.
         Note that if $\varsigma=r-1$ then $P_{\varsigma+1}/B=P/B$. Due to Theorem \ref{inclusisom},
           essentially we need to compare  $QH^*(P_{\varsigma+1}/P_{\varsigma})$ with  $QH^*(P/B)$ by using Peterson-Woodward comparison formula in this case.
            Thus without loss of generality, we can assume   $\varsigma=r$ in the rest, which  is of main
             interest to us and can bring convenience on the notations.

       Denote the quantum product of $QH^*(G/P)$ by $\star_P$. Write  $\psi_{\Delta, \Delta_P}(q_{\lambda_P}\sigma^v)=q_{\lambda_B}\sigma^{v\omega_P\omega'}$, where
               $q_{\lambda_P}\sigma^v\in QH^*(G/P)$. Then we have
         $gr_r(q_{\lambda_B}\sigma^{v\omega_P\omega'})=\mathbf{0}$  by Proposition \ref{allgradcomp}.
        On the other hand, if $gr_r(q_\lambda \sigma^{vu})=\mathbf{0}$ with $\lambda_P=\lambda+Q^\vee_P$ and $u\in W_P$,
        then we conclude $gr(q_{\lambda_B-\lambda}\omega_P\omega')=gr(u)$ where $\lambda_B-\lambda\in Q_P^\vee$.
            By the uniqueness (from Lemma \ref{uniqAtypegrad}), we conclude $\lambda_B=\lambda$ and $\omega_P\omega'=u$.
            Hence, $\Psi_{r+1}$ is an isomorphism of vector spaces.

    By  Proposition \ref{comparison}, we have
                  $\Psi_{r+1}(\sigma^u \star_P \sigma^v)=
                   \Psi_{r+1}(\sigma^u)\star \Psi_{r+1}(\sigma^v)$ for  $u, v\in W^P$.
      To show $\Psi_{\varsigma+1}$ is an algebra isomorphism, it remains to show (i)
       $\Psi_{r+1}(q_{\lambda_P}\star_Pq_{\mu_P})=\Psi_{r+1}(q_{\lambda_P})\star  \Psi_{r+1}(q_{\mu_P})$ and (ii)
        $\Psi_{r+1}(q_{\lambda_P} \star_P \sigma^v)
                  = \Psi_{r+1}(q_{\lambda_P})\star \Psi_{r+1}(\sigma^v)$.
     For (i) we write $\lambda_P=\lambda'+Q^\vee_P$ and $\mu_P=\mu'+Q^\vee_P$ where $\lambda', \mu'$ are
      elements in $\bigoplus_{\alpha\in \Delta\setminus\Delta_P}\mathbb{Z} \alpha^\vee$.
          Note that $gr_{[r+1, r+1]}(q_{\lambda'})-gr_{[r+1, r+1]}(q_{\lambda_B}\omega_P\omega')=\mathbf{0}$.
            Hence, $q_{\lambda_B}\omega_P\omega'=q_{\lambda'}x$ with $x$ being the unique element in $W_P\times Q^\vee_P$ determined
             by the grading $-gr_r(\lambda')=:\mathbf{a}$.
               Similarly, we   $\psi_{\Delta, \Delta_P}(q_{\mu_P})=q_{\mu'}y$ and
                $ \psi_{\Delta, \Delta_P}(q_{\lambda_P+\mu_P})=q_{\lambda'+\mu'}z$  where $gr(y)=-gr_r(q_{\mu'})=:\mathbf{b}$ and $gr(z)=-gr_r(q_{\lambda'+\mu'})$.
            Hence, $\Psi_{r+1}(q_{\lambda_P})\star  \Psi_{r+1}(q_{\mu_P})=\overline{ q_{\lambda'}x}\star \overline{ q_{\mu'}y}
                       =\overline{ q_{\lambda'+\mu'}}\star A_\mathbf{a}\star A_{\mathbf{b}}=\overline{ q_{\lambda'+\mu'}}\star A_{\mathbf{a}+\mathbf{b}}=
                       \Psi_{r+1}(q_{\lambda_P+\mu_P})$.
        For (ii), we first           conclude $\overline{\sigma^{s_j}}\star \overline{\sigma^{v}}=
            \overline{\sigma^{vs_j}}$ where $1\leq j\leq r$ and $v\in W^P$,
                        by Proposition \ref{gracomponemore}.
          Thus by induction on $\ell(u)$ where $u\in W_P$, we conclude
           $\overline{\sigma^{u}}\star \overline{\sigma^{v}}=
            \overline{\sigma^{vu}}$.
            As a consequence,  (ii) follows.
                        Hence, $\Psi_{r+1}$ is an algebra isomorphism.
\end{proof}

As a consequence, we obtain the following.

 \vspace{0.25cm}

\noindent\textbf{Theorem \ref{baseisom}.}
{\itshape
Denote $\Gamma:=\{gr(q_{\lambda} {w})~|~   gr(q_{\lambda} {w})<\mathbf{0},   q_{\lambda} {w} \in QH^*(G/B)\}$.
  Let  $\mathcal{A}=\bigoplus_{gr(q_\lambda \sigma^w)\in \mathbb{Z}\mathbf{e}_{r+1}\cup \Gamma} \mathbb{Q}q_\lambda\sigma^w$ and
  $\mathcal{J}=\bigoplus_{gr(q_\lambda \sigma^w)\in   \Gamma} \mathbb{Q}q_\lambda\sigma^w$.
 Then $\mathcal{A}$ is a subalgebra of  $QH^*(G/B)$ and $\mathcal{J}$ is an ideal of $\mathcal{A}$.
 Furthermore if $\Delta_P$ is of $A$-type, then  there is a canonical algebra isomorphism:
    \begin{align*}
          &QH^*(G/P) \overset{\simeq}{\longrightarrow} \mathcal{A}/\mathcal{J};\\
           &\,\quad\quad q_{\lambda_P}\sigma^v\mapsto \psi_{\Delta, \Delta_P}(q_{\lambda_P}\sigma^v)+\mathcal{J}.
    \end{align*}
}

\begin{proof}
    Note that for any $q_\lambda \sigma^w\in QH^*(G/B)$,  $gr(q_\lambda \sigma^w)\in \mathbb{Z}\mathbf{e}_{r+1}$
    if and only if $gr(q_\lambda \sigma^w)\in \mathbb{Z}_{\geq 0}\mathbf{e}_{r+1}$.
  Clearly, $\mathbb{Z}_{\geq 0}\mathbf{e}_{r+1}\cup \Gamma$ is a sub-semigroup of $S$. Hence, $\mathcal{A}$ is a subalgebra of $QH^*(G/B)$, due to Theorem \ref{mainthm}.
  From Definition \ref{defgrading}, we note $gr_{[r+1, r+1]}(q_\lambda w)\geq \mathbf{0}$ whenever $q_\lambda w\in QH^*(G/B)$. Thus
    for such element,  $gr(q_\lambda w)<\mathbf{0}$ if and only if  $gr_{r}(q_\lambda w)< \mathbf{0}$. In particular, we conclude
     $\mathcal{J}$ is an ideal of $\mathcal{A}$, by use Theorem \ref{mainthm} again.
      As a consequence, we obtain a natural isomorphism $\mathcal{A}/\mathcal{I}\overset{\simeq}{\longrightarrow} Gr_{(r+1)}^{\mathcal{F}}(QH^*(G/B))$.
      Hence, the statement follows from Theorem \ref{isomfordirection}.
\end{proof}

\begin{remark}
   In fact,   $\mathcal{A}=\bigcup_{i\geq 0} F_{i\mathbf{e}_{r+1}}$. If we use the    $\mathbb{Z}^{r+1}$-filtration on $QH^*(G/B)$ that is naturally extended
    from the $S$-filtration, then we have $\mathcal{J}=F_{-\mathbf{e}_{r+1}}$.
   Furthermore it is obvious that
       $\mathcal{A}$, $\mathcal{J}$ are $\mathbf{q}$-deformations of $A=\pi^*(H^*(G/P))$ and $J=0$  respectively.
       Note that $\pi^*(\sigma^v)=\sigma^v$ for any $v\in W^P$.   $\psi_{\Delta, \Delta_P}$  is a natural generalization of $\pi^*$.

\end{remark}
Recall that in Definition \ref{defgrading}, we have given the gradings for all $q_\lambda w$'s.
  Clearly, $S$ is  contained in $\mathbb{Z}^r\times \mathbb{Z}_{\geq 0}$  as a sub-semigroup.
Combining Lemma \ref{uniqAtypegrad} and (part of) the proof of Lemma \ref{lemmasemigp}, we conclude
        that $S$ is naturally
  extended to the whole   $\mathbb{Z}^r\times \mathbb{Z}_{\geq 0}$ with negative powers of $\{q_1, \cdots, q_r\}$ allowed.  That is,
     $\mathbb{Z}^r\times \mathbb{Z}_{\geq 0}=\{gr(q_\lambda \sigma^w)~|~ q_\lambda\sigma^w\in QH^*(G/B)[q_1^{-1}, \cdots, q_{r}^{-1}]\}$.
       Therefore we obtain a natural $\mathbb{Z}^r\times \mathbb{Z}_{\geq 0}$-filtration
  on $QH^*(G/B)[q_1^{-1}, \cdots, q_{r}^{-1}]$, making it a $\mathbb{Z}^r\times \mathbb{Z}_{\geq 0}$-filtered algebra,  due to Theorem \ref{mainthm}.
  By abuse of notations, we also denote this filtration as $\mathcal{F}$.
  Consequently, we obtain a natural embedding  of graded algebras $Gr^{\mathcal{F}}(QH^*(G/B))\hookrightarrow
             Gr^{\mathcal{F}}(QH^*(G/B)[q_1^{-1}, \cdots, q_{r}^{-1}])$.
 For simplicity, we assume $\Delta_P$ is of $A$-type. Then
       for each $1\leq k\leq r$, we  note $\Psi_k(t_k^{k})=\overline{\prod_{i=1}^kq_i^{i}}$.
 By defining $t_k^{-1}\mapsto  {\Psi_k(t_k^{k-1})\star\overline{\prod_{i=1}^kq_i^{-i}}}$,
  we can extend the algebra isomorphism $\Psi_k$ to a larger algebra isomorphism
    $$QH^*(\mathbb{P}^k)[t_k^{-1}] \overset{\simeq}{\longrightarrow}  \bigoplus_{j\in \mathbb{Z}}Gr_{j\mathbf{e}_k}^{\mathcal{F}}(QH^*(G/B)[q_1^{-1}, \cdots, q_{r}^{-1}]),$$
        which we also  simply  denote   as $\Psi_k$.
Thus the next theorem follows as a direct consequence of   Theorem \ref{mainthm} and Theorem \ref{isomfordirection}.

 \vspace{0.25cm}

 \noindent\textbf{Theorem \ref{localizisom}.}
 {\itshape
   $QH^*(G/B)[q_1^{-1}, \cdots, q_{r}^{-1}]$ has a $\mathbb{Z}^{r} \times\mathbb{Z}_{\geq 0}$-filtration $\mathcal{F}$.
 If  $\Delta_P$ is of $A$-type, then combining $\Psi_k$'s gives an  isomorphism   of  $\mathbb{Z}^r\times \mathbb{Z}_{\geq 0}$-graded algebras,
$$ \Psi: \,\,\big(\bigotimes_{k=1}^r QH^*(\mathbb{P}^k)[t_k^{-1}]\big)\bigotimes
                  QH^*(G/P)\longrightarrow  Gr^{\mathcal{F}}(QH^*(G/B)[q_1^{-1}, \cdots, q_{r}^{-1}]).$$
 That is, $\Psi=\Psi_1\star \cdots\star \Psi_{r+1}: \big(\otimes_{k=1}^rf_k\big)\otimes q_{\lambda_P}\sigma^v             \longmapsto
                             \big(\prod_{k=1}^r{\Psi_k(f_k)}\big)\star \Psi_{r+1}(q_{\lambda_P}\sigma^v).$

\noindent (Note  we have an isomorphism   $H^*(\mathbb{P}^1)\otimes \cdots \otimes H^*(\mathbb{P}^r)\otimes H^*(G/P)\cong Gr^{\mathcal{F}}(H^*(G/B))$ of
    graded algebras, coming from the Leray spectral sequence.)}

\section{Conclusions}\label{secconclusion}

 All the theorems in the induction can be easily   generalized to all cases by dropping our assumption that $\Delta_P$ is connected. We give a brief description as follows.

  Write
          $\Delta_P=\bigsqcup_{k=1}^m\Delta_{(k)}$ such that
                $Dyn(\Delta_{(k)})$ is a connected component of   $Dyn(\Delta_P)$  for each $k$.
  Then the Weyl subgroup $W_P$ also splits into direct product of $W_k$'s which are the corresponding Weyl subgroups of $\Delta_{(k)}$'s.
   That is,  $W_P=W_1\times \cdots \times W_m$.
       Among these $\Delta_{(k)}$'s, there is at most one which is not of $A$-type. If such a subbase exists, then we just assume it to be the last one,
         saying $\Delta_{(m)}$. For each $k$,
  we denote $r_k=|\Delta_{(k)}|$. Set $M=\sum_{k=1}^{m}r_k$ and then
 take   the standard basis   $\{\mathbf{e}_{1, 1},\cdots, \mathbf{e}_{1, r_1}, \cdots, \mathbf{e}_{m, 1}, \cdots, \mathbf{e}_{m, r_m}, \mathbf{e}_{m+1, 1}\}$
    of $\mathbb{Z}^{M+1}$.

For each $k$, we fix the canonical order $(\Delta_{(k)},
\Upsilon_k)$ as described in section \ref{arrangement}. Then we
obtain a grading map
   $gr_{\Delta_{(k)}}: W\times Q^\vee \longrightarrow \bigoplus_{i=1}^{r_k+1}\mathbb{Z}\mathbf{e}_{k, i}$, using Definition \ref{defgrading} with respect to $(\Delta_{(k)}, \Upsilon_k)$.
 In particular for any $x\in W_k$ or   $x=q_{\alpha^\vee}$ with $\alpha\in \Delta_{(k)}$, we have
          $gr_{\Delta_{(k)}}(x)\in \bigoplus_{i=1}^{r_k}\mathbb{Z}\mathbf{e}_{k, i}\hookrightarrow \mathbb{Z}^{M+1}$ which we treat as an element of $\mathbb{Z}^{M+1}$ naturally.

 \begin{defn}\label{defgradingforgeneral}

          We define a  {grading map}     $gr: W\times Q^\vee  \longrightarrow \mathbb{Z}^{M+1}$ associated to
            $(\Delta_P, \Upsilon)$ as follows, where $\Upsilon=\prod_{k=1}^m\Upsilon_k$.
  \begin{enumerate}
     \item Write
          $w=v_{m+1}v_m\cdots    v_1$ (uniquely), in which $(v_1, \cdots, v_m, v_{m+1})\in W_1\times \cdots \times W_m\times W^P$.
          Then $gr(w)\triangleq \ell(v_{m+1})\mathbf{e}_{m+1, 1}+\sum_{k=1}^m  gr_{\Delta_{(k)}}(v_k)$.  
     \item For each $\alpha_{k, i}\in \Delta_{(k)}$, $gr(q_{\alpha_{k, i}^\vee})\triangleq  gr_{\Delta_{(k)}}(q_{\alpha_{k, i}^\vee})$.
             For   $\alpha\in \Delta\setminus\Delta_P$,
              we write    $\psi_{\Delta, \Delta_P}(q_{\alpha^\vee+Q^\vee_P})=\omega_P\omega'q_{\alpha^\vee}\prod\limits_{k=1}^m\prod\limits_{i=1}^{r_k}q_{{\alpha_{k, i}^\vee}}^{a_{k,i}}$ and then define
                $$gr(q_{\alpha^\vee})\triangleq \big( \ell(\omega_P\omega')+2+\sum\limits_{k=1}^m\sum\limits_{i=1}^{r_k}2a_{k,i}\big)\mathbf{e}_{m+1, 1}-
                       gr(\omega_P\omega')-\sum\limits_{k=1}^m\sum\limits_{i=1}^{r_k}a_{k,i}gr(q_{\alpha_{k,i}^\vee}).$$
     \item In general, $x=w\prod_{\alpha\in \Delta}q_{\alpha^\vee}^{b_\alpha}$, then   $gr(x)\triangleq gr(w)+\sum_{\alpha\in \Delta}b_{\alpha}gr(q_{\alpha^\vee})$.
  \end{enumerate}
\end{defn}

 As in section \ref{subsecmainthm}, we can define a subset, consisting of the gradings of $q_\lambda w$'s in $QH^*(G/B)$. This subset also turns out to
  be a (totally-ordered) sub-semigroup of $\mathbb{Z}^{M+1}$ and we also simply denote it as $S$ by abuse of notations. In addition, we obtain a family of subspaces of
    $QH^*(G/B)$ in the same way, which we also simply denote as $\mathcal{F}=\{\mathcal{F}_\mathbf{a}\}_{\mathbf{a}\in S}$ by abuse of notations.
 Then all the theorems in the introduction  can be easily generalized. For instance, we state part of them in   summary as follows.

\begin{thm}\label{genthmsum} $\mbox{}$
\begin{enumerate}
  \item  $QH^*(G/B)$ has an $S$-filtered algebra structure  with filtration $\mathcal{F}$, which naturally extends to a $\mathbb{Z}^{M+1}$-filtered algebra structure on
         $QH^*(G/B)$.
   \item  There is a canonical algebra isomorphism
      $$ QH^*(G/B)/\mathcal{I}  \overset{\simeq}{\longrightarrow}QH^*(P/B)$$
      for an ideal  $\mathcal{I}$ (which is explicitly defined) of $QH^*(G/B)$.
  \item Assume $P/B$ is isomorphic to product of $F\ell_{1+r_k}$'s (i.e.   $\!\Delta_{(k)}\!$'s are of $A\mbox{-type).}$
         \begin{enumerate}
           \item There exists a subalgebra $\mathcal{A}$ of $QH^*(G/B)$ together with an ideal $\mathcal{J}$ of $\mathcal{A}$, such that
               $QH^*(G/P)$ is canonically isomorphic to $\mathcal{A}/\mathcal{J}$  as algebras.
            \item  As graded algebras, (after localization)  $Gr^{\mathcal{F}}(QH^*(G/B))$ is isomorphic to
                        $\big(\bigotimes_{k=1}^m\bigotimes_{i_k=1}^{r_k}QH^*(\mathbb{P}^{i_k})\big)\bigotimes QH^*(G/P)$ .
         \end{enumerate}
\end{enumerate}
\end{thm}

We would like to point out again that our assumption `` all
$\Delta_{(k)}\!$'s are of $A$-type" is already general enough. This
situation has covered all
 $G/P$'s for $G$ being   of $A$-type or $G_2$-type, and more than half of $G/P$'s for  each  remaining type.
  Unfortunately, Theorem \ref{genthmsum} (3.b) is not true in a more general case when  $\Delta_{(m)}$ is not of $A$-type. In fact in this case,
   $QH^*(G/P)$  is only
    canonically  isomorphic to a proper subspace of $Gr^{\mathcal{F}}_{({M+1})}(QH^*(G/B))=\bigoplus_{i\geq 0}F_{i\mathbf{e}_{M+1}}/\cup_{\mathbf{b}<i\mathbf{e}_M}F_{\mathbf{b}}$
     as vector spaces. However, we could still expect

\begin{conjecture}\label{conj11}
  There exists a canonical algebra isomorphism between   $QH^*(G/P)$ and a subalgebra of  $Gr^{\mathcal{F}}_{({M+1})}(QH^*(G/B))$.
\end{conjecture}

 As a direct  consequence of Conjecture \ref{conj11}, we can conclude Theorem \ref{genthmsum} (3.a) always holds for any $G/P$.
 Part of the points of the proof for this
  is to show (i) and (ii) in the proof of Theorem \ref{isomfordirection}. That is,
   we need to show  the behavior of $\overline{\psi_{\Delta, \Delta_P} (q_{\lambda_P})}$'s do be like polynomials.
 Indeed, when    $\Delta$ is of $C$-type,   (i) and (ii) become  trivial.
     (Precisely, we use the notations in case $\mbox{C}1)$ in Table $\ref{tabrelativeposi}$ and assume $\Delta_P$ to be of $C$-type.
                           Then for any $\lambda_P=\sum_{j=1}^ob_j\beta_j^\vee+Q^\vee_P\in Q^\vee/Q^\vee_P$,
                            we conclude $\overline{\psi_{\Delta, \Delta_P} (q_{\lambda_P})}=\overline{q_{\lambda_B}\cdot 1}$ with $
                             \lambda_B=\sum_{j=1}^ob_j\beta_j^\vee+  b_o\sum_{p=o+1}^n \beta_{p}^\vee$
                               by direct calculations.) In this case, we could still prove Conjecture \ref{conj11}  together with some other arguments.  On the other hand, it is shown in \cite{lamshi} that
    after taking  torus-equivariant extension and localization, Theorem \ref{genthmsum} (3.a) is true
     in terms of the localization of equivariant homology of a based loop group.
  Hence, we   believe that Theorem \ref{genthmsum} (3.a) also holds  without taking  equivariant extension and localization. Both of these provide
    evidence for  our conjecture.

  In addition, we would like to ask the following.

\begin{qus}\label{ques11}
  What is the difference between  $QH^*(G/P)$ and   $Gr^{\mathcal{F}}_{({M+1})}(QH^*(G/B))$?
\end{qus}

 The   ring structure  of  $Gr^{\mathcal{F}}_{({M+1})}(QH^*(G/B))$, or equivalently $\mathcal{A}/\mathcal{J}$ which is defined in the same
     form as in Theorem \ref{baseisom},  seems close   to  the  ring structure of $QH^*(G/B)$. Especially, there might be one way to
  obtain a nice presentation   of $Gr^{\mathcal{F}}_{({M+1})}(QH^*(G/B))$  from the   presentation   of $QH^*(G/B)$ \cite{kim}. Suppose there were such a way and
   we knew the answer to Question \ref{ques11}, then we would have a better understanding on $QH^*(G/P)$.

\section{Appendix} \label{exceptional}

 In this section, we  show the Key Lemma also holds for  all the roots that  satisfy    condition (ii)  in the proof (of the Key Lemma in section \ref{sectiongenconn}) whenever
                 $\Delta$ is of  $F_4$-type  or $E$-type.
           Since all the arguments are similar, we just list all such roots as well as the corresponding methods for them.
           One can see \cite{leungliappend} for more details.

When $\Delta$ is of $F_4$-type, case $\mbox{C}9)$ or $\mbox{C}10)$ will occur. For instance for $\mbox{C}9)$,
   $\gamma^\vee$  must be either of the form $\sum_{i\leq t\leq k}\beta_t^\vee$ or equal to one of the following five coroots:
         $\beta_1^\vee+2\beta_2^\vee+\beta_3^\vee, \beta_1^\vee+2\beta_2^\vee+\beta_3^\vee+\beta_4^\vee,
       \beta_1^\vee+2\beta_2^\vee+2\beta_3^\vee+\beta_4^\vee,$ $ \beta_1^\vee+3\beta_2^\vee+2\beta_3^\vee+\beta_4^\vee, 2\beta_1^\vee+3\beta_2^\vee+2\beta_3^\vee+\beta_4^\vee$,
        by noting   $\ell(s_\gamma)=\langle 2\rho, \gamma^\vee\rangle -1$. Then we have

  \begin{center}
 $\mbox{}$ \hspace{0.3cm} Table for  case $\mbox{C} 9)$  \hspace{3.25cm} Table for case $\mbox{C}10)$

 \vspace{0.15cm}

 \hspace{-0.4cm}\begin{tabular}{|c|c|c|}
       \hline
       Coroots &  $r=2$&  $r=3$ \\ \hline\hline
         $\beta_2^\vee+\beta_3^\vee$ &\multicolumn{2}{c|}{done ${(\gamma\in R_P)}$} \\       \hline
         $\beta_3^\vee+\beta_4^\vee$ & \multicolumn{2}{c|}{(M3)}   \\       \hline
         $\beta_1^\vee+\beta_2^\vee+\beta_3^\vee$ & (M3) & done \\       \hline
         $\beta_2^\vee+\beta_3^\vee+\beta_4^\vee$ & \multicolumn{2}{c|}{(M1)}   \\       \hline
         $\beta_1^\vee+\beta_2^\vee+\beta_3^\vee+\beta_4^\vee$ & (M1) & done \\       \hline
         $\beta_1^\vee+2\beta_2^\vee+2\beta_3^\vee+\beta_4^\vee$ & \multicolumn{2}{c|}{(M2)}   \\       \hline
        $2\beta_1^\vee+3\beta_2^\vee+2\beta_3^\vee+\beta_4^\vee$ & (M1) & done \\       \hline
     \end{tabular}     \begin{tabular}{|c|c|c|}
       \hline
       Coroots &  $r=2$&  $r=3$ \\ \hline\hline
         $\beta_3^\vee+\beta_4^\vee$ & (M3) & (M2)    \\       \hline
         $\beta_1^\vee+\beta_2^\vee+\beta_3^\vee$ & (M1) & done \\       \hline
         $\beta_1^\vee+\beta_2^\vee+\beta_3^\vee+\beta_4^\vee$ & (M1) & done  \\       \hline
         $\beta_2^\vee+2\beta_3^\vee+\beta_4^\vee$ & (M3) & (M2)\\       \hline
        $\beta_1^\vee+\beta_2^\vee+2\beta_3^\vee+\beta_4^\vee$ & (M2) &  done \\       \hline
         $\beta_1^\vee+2\beta_2^\vee+3\beta_3^\vee+\beta_4^\vee$ &  \multicolumn{2}{c|}{(M2)}  \\       \hline
          $\beta_1^\vee+2\beta_2^\vee+3\beta_3^\vee+2\beta_4^\vee$ &  \multicolumn{2}{c|}{(M1)}   \\       \hline
     \end{tabular}
  \end{center}

\bigskip

We would like to make some comments for the   tables in this section.

\begin{enumerate}
  \item By ``done",  we mean that there exists $\alpha_j\in \{\alpha_1, \cdots, \alpha_{r-1}\}$ such that $\langle \alpha_j, \gamma^\vee\rangle >0$.
           Thus it is   done by the arguments for condition (i) in the proof of the Key Lemma.
              By ``done $(\gamma\!\in\! R_P)$", we mean   $\gamma\!\in\! R_P$ and thus it is done by the arguments at the beginning of the proof of the Key Lemma.
  \item By  ``(M1)" (resp. ``(M3)"), we mean the corresponding method, especially the use of part a) (resp. b)) of Lemma \ref{lemforgenone}.
  \item By ``(M2)", we mean the induction hypothesis is used. 
                    In fact, whenever referring to   (M2)   in the tables, we
         can use the same  arguments as follows. For instance, we consider the case when $\mbox{C}10)$ occurs,
                   $\gamma^\vee=\beta_2^\vee+2\beta_3^\vee+\beta_4^\vee$ and
                $r=3$.    In this case, we can take $\alpha_j=\beta_3(=\alpha_3)$. Then
                  $\beta^\vee=s_j(\gamma^\vee)=\beta_2^\vee+\beta_3^\vee+\beta_4^\vee$ and consequently $(\lambda_1, \lambda_2, \lambda_3, \lambda_4)=gr(q_{\beta^\vee})
                     =(-1, 1, -2, 8)$.  Furthermore, we have $(a_1, a_2, a_3, a_4)=(0, a_2, a_3, 0)$
                    and   $(b_1, b_2, b_3, b_4)=(0, b_2, b_3, 0)$ with $a_2+a_3=1$ and $b_2+b_3=-1$ by noting $\langle \alpha_1, \alpha_j^\vee\rangle=0$.
             If $\mu_1<0$, then it is done. If $\mu_1=0$, then by the induction hypothesis we have $\mu_2\leq 0$. We claim $\mu_2=0$.
            Thus $\mu_3\leq 0$ and $a_2+b_2=a_2+\mu_2+b_2\leq 0$ (by considering $\tilde{gr}$). Since $a_2+a_3+b_2+b_3=0$, if
               $a_2+b_2<0$ then it is done; otherwise, we have $a_3+b_3=-(a_2+b_2)=0$ so that $a_3+b_3+\mu_3=\mu_3\leq 0$. Thus it is done. It remains to show our claim.
              Indeed, we note that $\mu_2+i_2'=k_2+\lambda_2=k_2+1$. Since
              $\ell(us_js_\beta)<\ell(us_j)$, $us_j(\beta)\in -R^+$. Then if $\langle \alpha, \beta^\vee\rangle \leq 0$,
                 $us_j(\alpha)\in -R^+$ implies $us_js_\beta(\alpha)\in-R^+$.
              Hence $i_2'=\sharp \{\alpha\in  R_{P_2}^+\setminus R_{P_{1}}~|~ us_j(\alpha)\in -R^+\}
                          =\sharp \{\alpha\in  R_{P_2}^+\setminus R_{P_{1}}~|~ us_j(\alpha)\in -R^+, \langle \alpha, \beta^\vee\rangle \leq 0\}+
                                   \sharp \{\alpha\in  R_{P_2}^+\setminus R_{P_{1}}~|~ us_j(\alpha)\in -R^+, \langle \alpha, \beta^\vee\rangle > 0\}
                        \leq    \sharp \{\alpha\in  R_{P_2}^+\setminus R_{P_{1}}~|~ us_js_\beta(\alpha)\in -R^+\}+   \sharp \{\alpha\in  R_{P_2}^+\setminus R_{P_{1}}~|~  \langle \alpha, \beta^\vee\rangle > 0\}
                         =k_2+\sharp\{\beta_2\}$.   Thus $\mu_2=k_2+ 1-i_2'\geq 0$ and consequently we have  $\mu_2=0$.
\end{enumerate}

  Now we assume  $\Delta$ is of $E$-type. Denote $\Xi :=\{\beta_i~|~ \langle \beta_i, \gamma^\vee\rangle>0\}$. Recall that we should replace $\kappa=o+r$ with $\kappa=o+\varsigma (=o+r-1)$ in Table \ref{tabrelativeposi}
    when 
     $\Delta_P$ is not of $A$-type.
 Note that any $\gamma\in R^+$ is of length $\langle 2\rho, \gamma^\vee\rangle-1$.
It suffices to assume $n=8$. It remains to discuss at most the roots in
 the tables as below.

  \begin{center}
      Table for case $\mbox{C}4)$ with $r=6$ or $r=7$

      \begin{tabular}{|c|c|c|}
       \hline
       Roots with $\Xi \subset\{\beta_1, \beta_2, \beta_8\}$    &   $r=6$&   $r=7$ \\ \hline \hline
          $\beta_1+\beta_2$ & done & done  \\       \hline
       $\beta_2+\beta_3+\beta_4+\beta_5+\beta_8$ & (M3) & done  \\       \hline
        $\beta_1+\beta_2+\beta_3+\beta_4+\beta_5 +\beta_8$ & done & (M3)  \\       \hline
       $\beta_1+2\beta_2+2\beta_3+2\beta_4+2\beta_5+\beta_6+ \beta_8$ & (M1) & done  \\       \hline
       $\beta_3+2\beta_4+3\beta_5+2\beta_6+\beta_7+2\beta_8$ &\multicolumn{2}{c|}{done ${ (\gamma\in R_P)}$} \\       \hline
        $\beta_2+\beta_3+2\beta_4+3\beta_5+2\beta_6+\beta_7+2\beta_8$ & (M3) & done  \\       \hline
         $\beta_1+\beta_2+\beta_3+2\beta_4+3\beta_5+2\beta_6+\beta_7+2\beta_8$ & done &  (M3) \\       \hline
       $\beta_1+2\beta_2+2\beta_3+2\beta_4+3\beta_5+2\beta_6+\beta_7+2\beta_8$ & (M2) &  done \\       \hline
         $\beta_1+2\beta_2+3\beta_3+4\beta_4+5\beta_5+3\beta_6+\beta_7+3\beta_8$ & (M2) & (M3)  \\       \hline
           $\beta_1+3\beta_2+4\beta_3+5\beta_4+6\beta_5+4\beta_6+2\beta_7+3\beta_8$ & (M1)  & done  \\       \hline
          $2\beta_1+3\beta_2+4\beta_3+5\beta_4+6\beta_5+4\beta_6+2\beta_7+3\beta_8$ & done & (M1)  \\       \hline
    \end{tabular}
  \end{center}

\bigskip

 \begin{center}
     Table for case $\mbox{C}5)$ with $r=5$

    \begin{longtable}{|c|c|}
       \hline
       Roots with $\Xi \subset\{\beta_5, \beta_6\}$   &   Method \\ \hline \hline
       $\beta_5+\beta_6$ &  (M3) \\       \hline
         $\beta_2+2\beta_3+\beta_4+2\beta_5+\beta_6$ & (M3) \\       \hline
         $\beta_2+2\beta_3+\beta_4+2\beta_5+2\beta_6+\beta_7$ & (M1)\\       \hline
     $\beta_1+2\beta_2+3\beta_3+\beta_4+3\beta_5+2\beta_6+\beta_7$ & (M3) \\       \hline
  $\beta_1+2\beta_2+3\beta_3+\beta_4+3\beta_5+3\beta_6+2\beta_7+\beta_8$ & (M1) \\       \hline
   $\beta_1+2\beta_2+4\beta_3+2\beta_4+4\beta_5+3\beta_6+2\beta_7+\beta_8$ & (M2) \\       \hline
   $2\beta_1+4\beta_2+6\beta_3+3\beta_4+5\beta_5+3\beta_6+2\beta_7+\beta_8$ & (M2) \\       \hline
    $2\beta_1+4\beta_2+6\beta_3+3\beta_4+5\beta_5+4\beta_6+2\beta_7+\beta_8$ & (M1) \\       \hline
      \end{longtable}
  \end{center}

\bigskip

  \begin{center}
       Table for case $\mbox{C}7)$ with $0\leq o\leq 3$
           \begin{tabular}{|c|c|c|c|}
         \hline
       & Roots with $\Big\{\begin{array}{l}\Xi \subset  \{\beta_1, \beta_2, \beta_3, \beta_7, \beta_8\}\\
                          |\Xi\cap \{\beta_1, \beta_2, \beta_3\}|\leq 1\end{array}$   &  Constraint &   Method \\ \hline \hline
       1) &        $\beta_3+\beta_4+\beta_5+ \beta_7$ & $o=3$ &    \\ \cline{1-3}
       2) &        $\beta_2+\beta_3+\beta_4+\beta_5+\beta_7$ &  $o=2$  & (M3)  \\ \cline{1-3}
       3) &        $\beta_1+\beta_2+\beta_3+\beta_4+\beta_5+\beta_7$ &  $o=1$  &     \\ \hline
       4) &        $\beta_2+2\beta_3+2\beta_4+2\beta_5+\beta_6+\beta_7$ &  $o=3$  &     \\      \cline{1-3}
       5) &        $\beta_1+2\beta_2+2\beta_3+2\beta_4+2\beta_5+\beta_6+\beta_7$ &  $o=2$  & \raisebox{1.5ex}[0pt]{(M1)}   \\       \hline
       6) &        $\beta_7+\beta_8$ &  $o\geq 0$  & (M2,2,3,3)   \\       \hline
       7) &        $\beta_3+\beta_4+\beta_5+\beta_7+\beta_8$ &  $o=3$  &     \\       \cline{1-3}
       8) &        $\beta_2+\beta_3+\beta_4+\beta_5+\beta_7+\beta_8$ &  $o=2$  &    \\      \cline{1-3}
       9) &        $\beta_1+\beta_2+\beta_3+\beta_4+\beta_5+\beta_7+\beta_8$ &  $o=1$  &  (M1)  \\      \cline{1-3}
      10) &        $\beta_2+2\beta_3+2\beta_4+2\beta_5+\beta_6+\beta_7+\beta_8$ &  $o=3$  &     \\       \cline{1-3}
      11) &        $\beta_1+2\beta_2+2\beta_3+2\beta_4+2\beta_5+\beta_6+\beta_7+\beta_8$ &  $o=2$  &     \\       \hline
      12) &        $\beta_4+2\beta_5+\beta_6+2\beta_7+\beta_8$ &  $o\geq 0$  & (M2,2,3,3)   \\       \hline
      13) &        $\beta_3+\beta_4+2\beta_5+\beta_6+2\beta_7+\beta_8$ &  $o=3$  & (M3)   \\       \hline
      14) &        $\beta_2+\beta_3+\beta_4+2\beta_5+\beta_6+2\beta_7+\beta_8$ &  $o=2$  &    \\       \cline{1-3}
      15) &        $\beta_1+\beta_2+\beta_3+\beta_4+2\beta_5+\beta_6+2\beta_7+\beta_8$ &  $o=1$  &    \\      \cline{1-3}
      16) &        $\beta_2+2\beta_3+2\beta_4+2\beta_5+\beta_6+2\beta_7+\beta_8$ &  $o=3$  &  \raisebox{1.5ex}[0pt]{(M2)}  \\       \cline{1-3}
      17) &        $\beta_1+2\beta_2+2\beta_3+2\beta_4+2\beta_5+\beta_6+2\beta_7+\beta_8$ &  $o=2$  &    \\       \hline
      18) &        $\beta_1+2\beta_2+3\beta_3+3\beta_4+3\beta_5+\beta_6+2\beta_7+\beta_8$ &  $o=3$  & (M1)   \\       \hline
      19) &        $\beta_2+2\beta_3+3\beta_4+4\beta_5+2\beta_6+3\beta_7+\beta_8$ &  $o\geq 0$  &     \\      \cline{1-3}
      20) &        $\beta_1+\beta_2+2\beta_3+3\beta_4+4\beta_5+2\beta_6+3\beta_7+\beta_8$ &  $o=1$  &    \\       \cline{1-3}
      21) &        $\beta_1+2\beta_2+2\beta_3+3\beta_4+4\beta_5+2\beta_6+3\beta_7+\beta_8$ &  $o=2$  & \raisebox{1.5ex}[0pt]{(M2)}     \\    \cline{1-3}
      22) &        $\beta_1+2\beta_2+3\beta_3+3\beta_4+4\beta_5+2\beta_6+3\beta_7+\beta_8$ &  $o=3$  &     \\       \hline
      23) &        $\beta_2+2\beta_3+3\beta_4+4\beta_5+2\beta_6+3\beta_7+2\beta_8$ &  $o\geq 0$  &    \\       \cline{1-3}
      24) &        $\beta_1+\beta_2+2\beta_3+3\beta_4+4\beta_5+2\beta_6+3\beta_7+2\beta_8$ &  $o=1$  &    \\      \cline{1-3}
      25) &        $\beta_1+2\beta_2+2\beta_3+3\beta_4+4\beta_5+2\beta_6+3\beta_7+2\beta_8$ &  $o=2$  &  \raisebox{1.5ex}[0pt]{(M1)}    \\      \cline{1-3}
      26) &        $\beta_1+2\beta_2+3\beta_3+3\beta_4+4\beta_5+2\beta_6+3\beta_7+2\beta_8$ &  $o=3$  &     \\       \hline
      27) &        $\beta_1+2\beta_2+3\beta_3+4\beta_4+5\beta_5+2\beta_6+4\beta_7+2\beta_8$ &  $o\geq 0$  & (M3,2,2,2)   \\       \hline
      28) &        $\beta_1+2\beta_2+4\beta_3+5\beta_4+6\beta_5+3\beta_6+4\beta_7+2\beta_8$ &  $o=3$  &    \\       \cline{1-3}
      29) &        $\beta_1+3\beta_2+4\beta_3+5\beta_4+6\beta_5+3\beta_6+4\beta_7+2\beta_8$ &  $o=2$  &  (M1)  \\        \cline{1-3}
      30) &        $2\beta_1+3\beta_2+4\beta_3+5\beta_4+6\beta_5+3\beta_6+4\beta_7+2\beta_8$ &  $o=1$  &     \\       \hline
      \end{tabular}
  \end{center}

In the above table, by ``(M2,2,3,3)" for the root $\beta_7+\beta_8$, we mean (M2) (resp. (M2), (M3) and (M3)) is used when $o=0$ (resp. 1, 2 and 3).
  Similar notations are used for the case no. 12) and no. 27).

\end{document}